\documentclass{amsart}
\usepackage{verbatim,amsfonts,color,mathrsfs,stmaryrd,graphicx, mathtools}
\usepackage{tikz, tikz-cd}
\usetikzlibrary{patterns} 

\theoremstyle{plain}
\newtheorem{Thm}{Theorem}
\newtheorem{Cor}[Thm]{Corollary}

\newtheorem{Lem}[Thm]{Lemma}
\newtheorem{Prop}[Thm]{Proposition}

\theoremstyle{definition}
\newtheorem{Def}[Thm]{Definition}
\newtheorem{Rmk}[Thm]{Remark}
\newtheorem{Eg}[Thm]{Example}

\theoremstyle{remark}

\ifx\undefined\fiverm \font\fiverm=cmr5 \fi

\usepackage{hyperref}

\begin{document}


\title[Synchronization  for an infinite class of continued fraction   transformations]{Synchronization is full measure for all  $\alpha$-deformations of an infinite class of continued fraction   transformations}
\author{Kariane Calta}
\address{Vassar College} 
\email{kacalta@vassar.edu }

\author{Cor Kraaikamp}
\address{Technische Universiteit Delft and Thomas Stieltjes Institute of Mathematics\\ EWI\\ Mekelweg 4\\ 2628 CD Delft, the Netherlands}
\email{c.kraaikamp@tudelft.nl}

\author{Thomas A. Schmidt}
\address{Oregon State University\\Corvallis, OR 97331}
\email{toms@math.orst.edu}%
\thanks{The second and third named authors thank the Mathematics Department of Vassar College for support of a stimulating   visit.}
\date{16 January 2017}%


\begin{abstract}     We study an infinite family of one-parameter deformations, so-called $\alpha$-continued fractions,  of interval maps associated to distinct triangle Fuchsian groups. 
  In general for such one-parameter deformations,  the function giving the entropy of the map indexed by $\alpha$ varies in a way directly related to whether or not the orbits of the endpoints of the map synchronize.     
For two cases of one-parameter deformations associated to the classical case of the modular group $\text{PSL}_2(\mathbb Z)$, the set of $\alpha$ for which synchronization occurs has been determined (see \cite{CT, CIT}, \cite{KraaikampSchmidtSteiner}). 
 
Here, we explicitly determine the  synchronization sets for each  $\alpha$-deformation in our infinite family.  (In general, our Fuchsian groups are not subgroups of the modular group, and hence the tool of relating $\alpha$-expansions back to regular continued fraction expansions is not available to us.)  A curiosity here is that all of our  synchronization sets can be described in terms of a single tree of words.    In a paper in preparation,  we identify the natural extensions of our maps, as well as the entropy functions associated to each deformation.  
\end{abstract}

\maketitle

\tableofcontents

\section{Introduction} 

\subsection{Main results}  Associated to each of the infinite family of groups $G_{3,n}$ defined below in \eqref{e:generators},  we introduce the continued fraction type maps $T_{3,n,\alpha}\,$, defined in \eqref{e:maps} below, parametrized  by $\alpha \in [0,1]$.   When $\alpha = 0$   this gives the (unaccelerated) maps treated in \cite{CaltaSchmidt}. We show that for each $n$, the set of those $\alpha$ such that the $T_{3,n,\alpha}$-orbits of the endpoints of the interval of definition, denoted $\ell_0(\alpha)$ and $r_0(\alpha)$ respectively,  eventually agree has full Lebesgue measure.    We call such agreeing of orbits {\em synchronization}.  We give a full description of the set of $\alpha$ for which synchronization occurs.   The following is a simply stated implication of this detailed description.

\begin{Thm}\label{t:FullMeasure}   For  $n\ge 3$,   
the set of $\alpha \in (0, 1)$ such that there exists $i = i_{\alpha}, j = j_{\alpha}$ with 
$T_{3,n,\alpha}^{i}(\,r_0(\alpha)\,) = T_{3,n,\alpha}^{j}(\,\ell_0(\alpha)\,)$ is of full Lebesgue measure.  
\end{Thm} 

Synchronization is key to determining planar models of natural extensions for many continued fraction type maps.   See for example  \cite{CT, KraaikampSchmidtSteiner}.   Indeed,  in a forthcoming paper we apply the results obtained in the present to give the natural extensions of the $T_{3, n, \alpha}$.  

A key phenomenon of our setting is that for all $n\ge 3$ there are two {\em synchronization relations},  in the sense that for $n$ fixed  there is a large subinterval of the values of $\alpha$ along which {\em all} values where synchronization occurs is announced by a basic relation in the group $G_{3,n}$ being satisfied by the elements $R, L$ such that  $T_{3,n,\alpha}^{i-1}(\,r_0(\alpha)\,) = R\cdot r_0(\alpha)$ and $T_{3,n,\alpha}^{j-1}(\,\ell_0(\alpha)\,)= L\cdot \ell_0(\alpha)$.      These relations,   discovered by computational investigation and justified in Propositions~\ref{p:longWordInWmIsThree} and ~\ref{l:longWordInU}, determine intervals along which synchronization occurs.

These {\em synchronization intervals} are in particular defined by admissibility of both the digits of the expansion of $T_{3,n,\alpha}^{i-1}(\,r_0(\alpha)\,)$ and of $T_{3,n,\alpha}^{j-1}(\,\ell_0(\alpha)\,)$.   An appropriate endpoint of each interval is determined by one of the expansions no longer being admissible, the applicable synchronization relation allows us to determine the other expansion at that endpoint.  This  leads to the tree of words, $\mathcal V$, defined in Definition~\ref{d:descendents}.

The intervals are  indexed by  $\mathbb Z_{\neq 0} \times \mathcal V$, where $\mathcal V$ is the  tree of words defined in Definition~\ref{d:descendents}.   (A proper subset of $\mathcal V$ is necessary when $-1$ is the indexing natural number.)   That the complement in $[0,1]$ of the collection of the intervals of synchronization is a thin Cantor set is proven by use of the ergodicity result of \cite{CaltaSchmidt} for the setting of $\alpha = 0$.  

The admissibility of the two expansions defining a synchronization interval is shown by induction, with the expansion related to the indexing pair being straightforward and the other expansion requiring a more delicate induction argument.    A new phenomenon presents itself in the proofs of admissibility:   the interval of admissibility of a candidate expansion of digits for some endpoint (in other terms, the corresponding higher rank cylinder) has  an endpoint determined by the longest string of digits having a property that we name full-branched, see \S~\ref{sss:fullBranched}. 

The determination of the synchronization set of $\alpha$ proceeds by treating three partitioning subintervals of $(0,1)$, after initial results in Sections~\ref{s:AlpZero}, ~\ref{s:alpOne}  describing the dynamics in the setting of $\alpha = 0,1$ for all $n, m \ge 3$.   The indexing is such that positive $k$ correspond to the leftmost subinterval of $\alpha$ where we express the synchronization intervals in terms of right digits; the rightmost subinterval corresponds to  $k<-1$ for this and the middle case of $k=1$, the synchronization intervals are expressed in terms of left digits.      

We expect that the case of $m>3$ will be very similar,  although the synchronization relations will involve longer words and thus some arguments will become awkwardly tedious.  Our work raises the question of whether there is a simple characterization of when a one-parameter deformation of interval maps has a set of synchronization relations. 

\subsection{Motivation} 

The $\alpha$-continued fractions of Nakada \cite{N} are associated to the modular group $\text{SL}_2(\mathbb R)$, whose projective quotient is $G_{2,3}$.   Nakada determined natural extensions and more for the setting of $\alpha \ge 1/2$.  Kraaikamp \cite{K} gave a more direct method for treating these values.   Intermediate results were obtained in particular by Marmi-Cassa-Moussa \cite{MCM}.  The  work of  Luzzi-Marmi \cite{LM} reinvigorated this area.   Nakada-Natsui \cite{NN} confirmed a numeric observation of \cite{LM} by showing what in our terminology is that a certain synchronization relation holds for the Nakada $\alpha$-continued fractions.  It was left to Tiozzo et al \cite{CIT} and, independently, \cite{KraaikampSchmidtSteiner} to show that the relation accounts for all synchronization.      These authors also showed that the entropy function, assigning to $\alpha$ the entropy of the interval indexed by $\alpha$, varies in a way previsely described in terms of the synchronization intervals.  (Note that some authors refer to synchronization as {\em matching}.) 

Partial results when the underlying groups being the Hecke triangle groups (thus, the $G_{2,q}, q\ge 3$) were given by \cite{DKS, KraaikampSchmidtSmeets}.    Tiozzo and others \cite{CIT} treat a one-parameter deformation sitting inside a two-dimensional family of continued fractions with underlying group the modular group introduced by Katok and Ugarcovici \cite{KatokUgarcovici10a}.   For both of the families associated to the modular group studied to date, Tiozzo and coauthors \cite{CIT, CT2, BCIT} relate the entropy function to explicit subsets of the Mandelbrot set; see also Tiozzo's thesis, \cite{T}.      
 
 Our goal is to study deformation families of continued fractions defined over higher degree fields, and determine how the entropy function varies.   The current work is the key step in this, in particular highlighting the central nature of synchronization relations.  In work in preparation, we apply these results to give the entropy functions.

\subsection{The basics of our maps}  
  We use the groups considered in \cite{CaltaSchmidt}.    Fix integers $n\ge m \ge 3$, and let $\mu = \mu_m = 2 \cos \pi/m, \,\nu = \nu_n =  2 \cos \pi/n$. Also let $t= \mu+\nu$ that is,

\[ t  := t_{m,n} = 2 \cos \pi/m + 2 \cos \pi/n.\]
  
Let  $G_{m,n}$ be generated by

\begin{equation}\label{e:generators} 
A = \begin{pmatrix} 1& t\\
                                      0&1\end{pmatrix},\, B = \begin{pmatrix} \nu& 1\\
                                      -1&0\end{pmatrix},\,   C = \begin{pmatrix} -\mu& 1\\
                                      -1&0\end{pmatrix}\,,
\end{equation}

\medskip
\noindent
and note that $C= AB$.    We work projectively,  hence $B, C$ are of order $n,m$ respectively while $A$ is of infinite order.    That is, $G_{m,n}$ is a Fuchsian triangle group of signature $(m,n,\infty)$.  

 \bigskip
 Fix  $\alpha \in [0,1]$ and define 
\[ \mathbb I_{m,n, \alpha}  := \mathbb I_{\alpha} = [\,(\alpha - 1)t, \alpha t\,)\,.\]

\noindent
Let 
\begin{equation}\label{e:maps} 
  T_{\alpha} = T_{m,n,\alpha}: x \mapsto A^{k} C^{l}\cdot x, 
\end{equation}
\bigskip 

\noindent
 where as usual, any  $2 \times 2$ matrix {\small$\begin{pmatrix}a & b \\ c & d\end{pmatrix}$} acts on real numbers by {\small$\begin{pmatrix}a & b \\ c & d\end{pmatrix} \cdot x = \dfrac{a x + b}{c x + d}$}, and  
 
\begin{itemize} 
\item  $l$ is minimal such that $C^{l} \cdot x \notin \mathbb I$

\medskip 
\item $k = -\lfloor (C^{l} \cdot x)/t + 1 - \alpha\, \rfloor$.

\end{itemize}

\medskip 
\noindent
We consider $T_{\alpha}$ as a map on the closed interval taking values in the half-open interval $ \mathbb I_{\alpha}$,
\[  T_{\alpha}:  [\,(\alpha - 1)t, \alpha t\,] \to \mathbb I_{\alpha}\,.\]

\bigskip 
\subsection{Geometric perspective, well-definedness}
The reader may well ask if there always does exist an $l$ such that  $C^{l} \cdot x \notin \mathbb I_{\alpha}$.   For the special cases of $\alpha = 0, 1$ see below; here we briefly indicate the setting for all other $\alpha$.     A quick study of the graph of the function $x \mapsto C\cdot x$ shows that this has horizontal asymptotes given by $y = \mu$, a pole at $x = 0$ and a zero at $x = 1/\mu$.  Of course this is an increasing function.   In fact,   $C$ is an elliptic matrix   (that is, its trace is of absolute value less than 2) that fixes a point in the upper half-plane of real part $\mu/2$.  It thus acts as a rotation about that fixed point.   Indeed, it acts as a rotation on a hyperbolic $m$-gon; from the words above,    this $m$-gon has consecutive vertices   $1/\mu, 0, \infty , \mu$.  (Note that when $m=3$, we have $\mu = 1/\mu =1$.)  Therefore the remaining $m-4$-vertices lie between $1/\mu$ and $\mu$; let us denote the set of all vertices by  $v_1 = \mu, v_2, \dots, v_{m-3}, v_{m-2} = 1/\mu, v_{m-1} = 0, v_{m} = \infty$.       Thus, $C$ acts on the real line as  $(-\infty, 0) \to (\mu, \infty) \to (v_2, \mu)\to (v_2, v_3) \to \cdots \to (1/\mu, v_{m-3})\to (0, 1/\mu) \to (-\infty, 0)$.  For $x \in \mathbb I_{\alpha}$,  the map $T_{\alpha}$ is thus the composition of rotating by powers of $C$ until $C^l\cdot x$ is no longer in $\mathbb I_{\alpha}$, and then shifting by applying the appropriate power of $A$ to bring this image back into $\mathbb I_{\alpha}$.

Certainly the left endpoint of $\mathbb I_{\alpha}$, being negative,  is sent by $C$ to a positive real number.  We now briefly indicate why this value is greater than  $\alpha t$. It then follows that $T_{\alpha}$ on the left endpoint is given by some $A^{-k}C$ with $k>0$. 
In fact,  for all negative $x\in \mathbb I_{\alpha}$, we claim that $C\cdot x > x+t$.  Elementary calculus shows that the graph of $x \mapsto C\cdot x$ has a tangent line with slope $1$ of equation $y= x + \mu +2$.   Since $\nu < 2$,  the tangent line lies below the line $y=x+t$.  Since the map has a pole at $x=0$, it easily follows that the claim holds.     The claim implies that there is a leftmost subinterval sent outside of $\mathbb I_{\alpha}$ by $C$; we can partition $\mathbb I_{\alpha}$ by applications of powers of $C^{-1}$ to this leftmost subinterval (with the rightmost image subinterval  restricted to its intersection with $\mathbb I_{\alpha}$).   In particular, it follows that  there always does indeed exist an $l$ such that  $C^{l} \cdot x \notin \mathbb I_{\alpha}$. 
\bigskip 

\subsection{Continued fraction perspective}   
The main aim of this subsection is to assure the reader that the $T_{m,n,\alpha}$ indeed yield continued fractions, and therefore are reasonably called continued fraction-like maps.

Since 
\[
C\cdot x = \frac{-\mu x+1}{-x} = \frac{-1}{x} + \mu,
\]
we immediately find that
\begin{equation}\label{cf:Cpart1}
x = \frac{1}{\mu-C\cdot x},
\end{equation}
Now
\[
C^2\cdot x = C\cdot (C\cdot x) =  \frac{-1}{C\cdot x} + \mu,
\]
yielding that
\begin{equation}\label{cf:Cpart2}
C\cdot x = \frac{1}{\mu-C^2\cdot x}.
\end{equation}
Now~(\ref{cf:Cpart1}) and~(\ref{cf:Cpart2}) yield that
\[
x = \frac{1}{\mu - \displaystyle{\frac{1}{\mu - C^2\cdot x}}} .
\]
After $l$ times we find
\begin{equation}\label{cf:Cpart3}
x = \frac{1}{\mu - \displaystyle{\frac{1}{\mu - \displaystyle{\ddots \displaystyle{-\frac{1}{\mu - C^l\cdot x}}}}}}.
\end{equation}
Since $A^k\cdot x = x+kt$, we see that $T_{\alpha}(x) = A^k\cdot (C^{l}\cdot x) = C^{l}\cdot x +kt$, yielding that
$C^{\ell}\cdot x = T_{\alpha}(x) - kt$. Substituting this in~(\ref{cf:Cpart3}) gives
$$
x = \frac{1}{\mu - \displaystyle{\frac{1}{\mu - \displaystyle{\ddots \displaystyle{-\frac{1}{\mu - \displaystyle{\frac{1}{\mu + kt - T_{\alpha}(x)}}}}}}}}.
$$
Continuing in this way we find a continued fraction expansion of $x$ with partial quotients given by $\mu$ and $\mu+kt$, with $k\in\mathbb N$.

 
\subsection{Digits, cylinders,  admissible words, ordering}\label{ss:digitsCylsEtc}

When studying the dynamics of our maps, the orbits of the interval endpoints  of $\mathbb I_{\alpha}$ are of utmost importance.  We define
\[ 
\begin{aligned}
\ell_0&=(\alpha-1) t &\textrm{ and }&\; \ell_i=T^{i}_\alpha(\ell_0), \textrm{ for }i\geq1\,,\\
r_0&=\; \alpha t      &\textrm{ and }& \; r_j=T^{j}_\alpha(r_0), \textrm{ for }j\geq1\,.
\end{aligned}
\]

\subsubsection{Cylinders,  notation for digit sequences, full cylinders}\label{sss:cylsSequence}
   \noindent

\begin{figure}[h]
\noindent
\scalebox{.7}{
\begin{tikzpicture}[x=3.5cm,y=5cm]
\draw  (-1.62, 0)--(.38, 0);
\draw  (-0.8, -1)--(1.2, -1);
\draw  (-0.38, -2)--(1.62, -2);
 
\draw  (0, -0.1)--(0, 0.2);
\draw  (0, -1.1)--(0, -0.8);
\draw  (0.54, -1.1)--(0.54, -0.8);
\draw  (1, -1.1)--(1, -0.8);
\draw  (0, -2.1)--(0, -1.8);
\draw  (0.68, -2.1)--(0.68, -1.8);
\draw  (1, -2.1)--(1, -1.8);
\draw[thin,dashed] (-0.72, 0)--(-0.72, 0.2);
\draw[thin,dashed] (-0.28, 0)--(-0.28, 0.2);
\draw[thin,dashed] (0.24, 0)--(0.24, 0.2);
\draw[thin,dashed] (-.48, -1.1)--(-.48, -0.8);        
\draw[thin,dashed] (-0.18, -1)--(-0.18, -0.8); 
\draw[thin,dashed] (0.25, -1)--(0.25, -0.8);
\draw[thin,dashed] (0.85, -1)--(0.85, -0.8);
\draw[thin,dashed] (1.15, -1)--(1.15, -0.8);
\node at (-1.3,  0.12) {$\Delta_{\alpha}(-1,1)$};
\node at (-0.5, 0.12) {$\Delta_{\alpha}(-2,1)$};
\node at (0.45, 0.12) {$\Delta_{\alpha}(k,1)$};
\foreach \x/\y in {-1.62/-0.2, -0.8/-1.2,-0.38/-2.2, 
} { \node at (\x,\y) {$\ell_0$}; }  
\node at (0.38, -0.2) {$r_0$};    
\node at (1.2, -1.2) {$r_0$}; 
\node at (1.62, -2.2) {$r_0$}; 
\node at (0.5, -0.2) {$\mathfrak b$};  
\node at (0.55, -1.2) {$\mathfrak b$};
\node at (0.71, -2.2) {$\mathfrak b$}; 
%
\foreach \x/\y in {-0.15/0.12, 0.15/0.12, 
-0.1/-0.88,  0.1/-0.88,0.93/-0.88, 1.08/-0.88,
-0.09/-1.88, 0.09/-1.88, 1.045/-1.88 
} { \node at (\x,\y) {$\cdots$}; }  
\foreach \x/\y in {-1.62/0, .38/0, .5/0,
-0.8/-1,  1.2/-1, 
 -0.38/-2, 1.62/-2
} { \node at (\x,\y) {$\bullet$}; }   
\node at (0, -0.2) {$0$};  
\node at (0, -1.2) {$0$};
\node at (0, -2.2) {$0$};
\node at (1, -1.2) {$1$};
\node at (1, -2.2) {$1$};
\node at (-0.65,-0.88) {$(-1,1)$};   
\node at (-0.325,-0.88) {$(-2,1)$};
\node at (0.4,-0.88) {$(1,1)$};
\node at (0.7,-0.88) {$(-1,2)$};
\node at (1.3,-0.88) {$(k,2)$}; 
\node at (-0.4, -1.88) {$(-k,1)$};   
\node at (0.3, -1.88) {$(2,1)$};      
\node at (0.55, -1.88) {$(1,1)$};   
\node at (0.84, -1.88) {$(-k,2)$}; 
\node at (1.29, -1.88) {$(2,2)$}; 
\node at (1.6, -1.88) {$(1,2)$}; 
\draw[thin,dashed] (-0.22, -2)--(-0.22, -1.8);        
\draw[thin,dashed] (0.19, -2)--(0.19, -1.8); 
\draw[thin,dashed] (0.42, -2)--(0.42, -1.8);
\draw[thin,dashed] (1, -2)--(1, -1.8);
\draw[thin,dashed] (1.18, -2)--(1.18, -1.8);
\draw[thin,dashed] (1.4, -2)--(1.4, -1.8);   

\end{tikzpicture}
}
\caption{Schematic representation of cylinders for three values of $\alpha$ when $m=n=3$.  Top: $\alpha< \gamma_{3,3}$; middle $\alpha < \epsilon_{3,3}$; bottom: $\alpha> \epsilon_{3,3}$.  For the bottom two,  $(k,l)$ denotes $\Delta_{\alpha}(k,l)$.}%
\label{f:someIalphaWithCylinders}%
\end{figure}
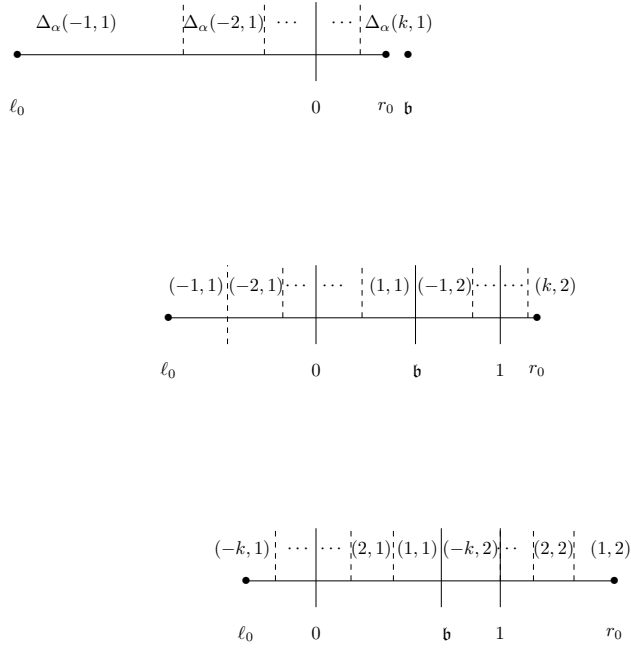

The {\em cylinders} for the map $T_{\alpha}$ are
\[ \Delta(k, l) =\Delta_{\alpha}(k, l) := \{ \, x \,|\, T_\alpha(x) = A^{k}C^{l}\cdot x\,\}\,.\]  
See Figure~\ref{f:someIalphaWithCylinders} for a representation of some explicit cylinders.
 Note that since each of 
$A, C$ are of positive determinant,  $T_{\alpha}$ is an {\em increasing function} on each of its cylinders.
We call $(k,l)$ the $\alpha$-digit of $x$ if $x\in \Delta_{\alpha}(k, l)$.  
 Define 
 \[\underline{b}_{[1, \infty)}^{\alpha} =  (k_1, l_1)(k_2, l_2)\dots \]  to be the sequence of digits for the orbit of $\ell_0 = (\alpha-1)t$; that is,  $\ell_0 \in \Delta(k_1, l_1)$, $\ell_1 \in \Delta(k_2, l_2)$, etc.  
 Similarly, define $\overline{b}{}^{\alpha}_{[1, \infty)}$ as the word giving the digits of the orbit of $r_0$.

A cylinder $\Delta_{\alpha}(k,l)$ is called {\em full} if its image under $T_{\alpha}$ is all of $\mathbb I_{\alpha}$.      Since the action by $C$ has a pole at $x=0$,  for all $\alpha \neq 0,1$ there are full cylinders $\Delta_{\alpha}(k, 1)$ with $k\in \mathbb Z$ of arbitrarily large absolute value.

  The only cylinder with $l = 1$ and $k<0$ that could be non-full is the leftmost cylinder of $\mathbb I_{\alpha}$, thus the cylinder of $\ell_0(\alpha)$.     Let 
  \[\mathfrak b = \mathfrak b_{\alpha} = C^{-1}\cdot \ell_0(\alpha).\]   Note that since $\ell_0(\alpha)<0$, one has $\mathfrak b < 1/\mu$.   If $\mathfrak b \notin \mathbb I_{\alpha}$,  then the rightmost cylinder of $\mathbb I_{\alpha}$ has $l=1, k>0$ and may be non-full.   If $\mathfrak b \in \mathbb I_{\alpha}$,  then all cylinders of index $(k,l), l=1, k>0$ are full; indeed,  (with possible exception of the rightmost cylinder) the cylinder of index $(k,2)$ is full if and only if the cylinder of index $(k,1)$ is (this, as $C$ acts so as to send $\Delta_{\alpha}(k, l+1)$ to $\Delta_{\alpha}(k, l)$\;). Continuing with analysis of this type shows that in general,  the only candidates for non-full cylinders are those of index $\{ (k, l), (k, l+1), \dots, (k, m-1),  (k', l')\}$ where $(k, l)$ is the $\alpha$-digit of $\ell_0(\alpha)$ and $(k', l')$ that of $r_0(\alpha)$.     Note that   the  $T_{\alpha}$ image of $\Delta_{\alpha}(k,l+j)$ is exactly the $T_{\alpha}$ image of $\Delta_{\alpha}(k,l)$, which is $[\ell_1, r_0)$.  (These are what one calls {\em right full} cylinders.)   Also,  the $T_{\alpha}$ image of $\Delta_{\alpha}(k',l')$  is $[\ell_0, r_1)$.   (It is a  {\em left full} cylinder.)

\begin{figure}
\scalebox{.5}{
{\includegraphics{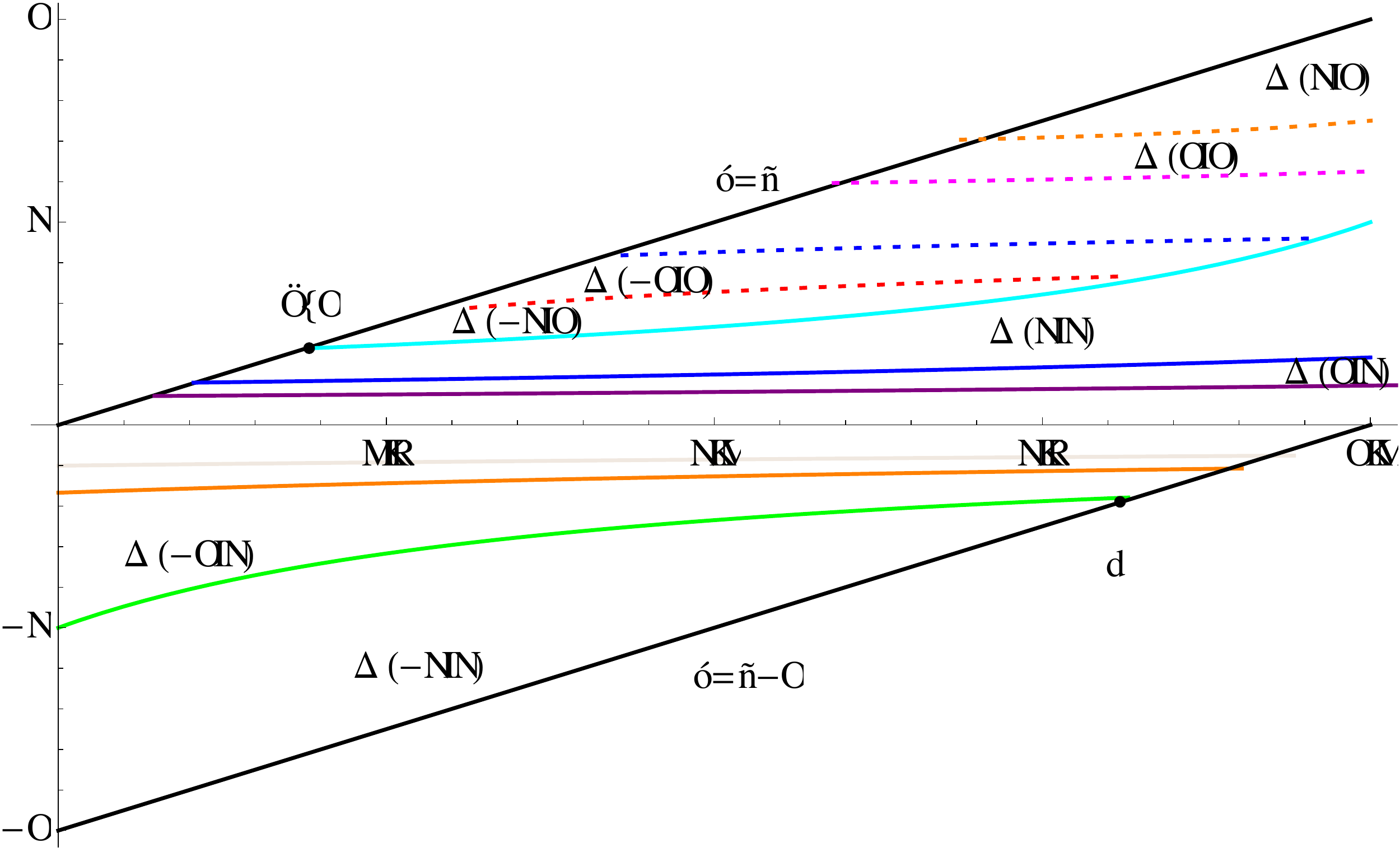}}}
\caption{The unions of the various $\Delta_{3,3,\alpha}(k,l)$, see Subsection~\ref{ss:digitsCylsEtc}.  Each $\mathbb I_{\alpha}$ is given as a vertical fiber (of coordinate $x = \alpha t$ with $t=t_{3,3}= 2$), from the left endpoint at the bottom up to the right  endpoint at the top of the fiber.   For each $\alpha$ such that they exist, 
the $\Delta_{\alpha}(k, 2)$ have limiting value as $|k|\to \infty$ of $1$.  Similarly, 
the  $\Delta_{\alpha}(k, 1)$ have limiting value as $|k|\to \infty$ of $0$.  Here $2\epsilon_{3,3} = G= (1 +\sqrt{5})/2$ and $2\gamma_{3,3} = g^2 = (G-1)^2$ correspond to the main division points in the interval of the parameter $\alpha$.}
\label{allDigs}
\end{figure}

\subsubsection{Admissibility and orders, definitions}\label{sss:admissOrd}
A word $U$ in the letters $A$, $C$ is called {\em admissible} for a pair $\alpha$ and $x\in \mathbb I_\alpha$ if  $U =A^{k_u}C^{l_u}\cdots A^{k_1}C^{l_1}$ and for each $j$, with $1\le j \le n$, one has  $A^{k_j}C^{l_j}\cdots A^{k_1}C^{l_1}\cdot x = T_{\alpha}^j(x)$.  Note that this is equivalent to having for each $j$  both  that (1) $A^{k_j}C^{l_j}\cdots A^{k_1}C^{l_1}\cdot x \in \mathbb I_{\alpha}$ and (2) $l_j$ is minimal such that 
$C^{l_j}A^{k_{j-1}}C^{l_{j-1}}\cdots A^{k_1}C^{l_1}\notin \mathbb I_{\alpha}$.    We also simply say that such a word $U$ is admissible for $\alpha$ if there exists an $x \in \mathbb I_{\alpha}$ such that $U$ is admissible for the pair $\alpha, x$.

   The $\alpha$-alphabet is the set of possible single digits for $x \in \mathbb I_{\alpha}$, that is all $(k, l)$ such that  $\Delta_{\alpha}(k, l)\neq \emptyset$.   The standard ordering of real numbers then induces an ordering of this alphabet:   $(k, l) \prec  (k', l')$ if $\Delta_{\alpha}(k, l)$ lies to the left of $\Delta_{\alpha}(k', l')$.  (Confer Figures  ~\ref{f:someIalphaWithCylinders} and ~\ref{allDigs};  in this second figure the order on each $\mathbb I_{\alpha}$ is rather from bottom to top.)  
   
   The analysis for the setting of fullness of cylinders also yields that for $k>0$, $\Delta_{\alpha}(-k,1)$ lies to the left of any  $\Delta_{\alpha}(-k -j,1)$ with $j>0$, as well as to the left of any $\Delta_{\alpha}(k',1)$, with $k'>0$.  Similarly,   $\Delta_{\alpha}(k',1)$ lies to the left of any $\Delta_{\alpha}(k'-j,1), j<k'$.    Now since $C$ acts in an order preserving manner,  we find that any $\Delta_{\alpha}(k,l)$ lies to the left of all $\Delta_{\alpha}(k,l+1)$.   We thus find that for each $\alpha$, the ordering on the $\alpha$-alphabet is a restriction of  the following order.

For each $\alpha$, the $\alpha$-alphabet is a (strict) subset of  $(\mathbb Z \setminus \{0\}\,) \times \{1, 2, \dots, m-1\}$.   We define the   {\em full order}  on $(\mathbb Z \setminus \{0\}\,) \times \{1, 2, \dots, m-1\}$  by  
\begin{equation}\label{e:theFullOrder}
\begin{aligned} 
(k, l) \prec  (k', l') \;\; \text{if and only if}\;\;&(i)\;  l<l', \\
                                                                \text{or}\;\; &(ii)\; l=l'\; \text{and one of}\;  k<k'<0, \;  k<0<k', \;\text{or}\; 0<k'<k.
\end{aligned}
\end{equation} 
  Thus,  
when $m=3$  and $n\ge 3$,   the full ordering is  given in Table~\ref{t:theOrder}.  
\begin{table}
\[ \begin{aligned}
&(1, 2) \succ (2,2) \succ \cdots  \succ (N,  2)\succ \cdots  \\
&---------------- \\
&\succ \cdots \succ (-M, 2) \succ  \cdots \succ  (-2, 2)  \succ (-1,2)\\
&---------------- \\
&\succ (1,1) \succ (2,1) \succ  \cdots \succ (N,  1) \succ  \cdots \\
&---------------- \\
&\succ \cdots \succ (-M,1) \succ \cdots \succ (-2,1) \succ(-1, 1).
\\
\\
\end{aligned}
\]

\caption{  The full order.  Here $m=3, n\ge 3$.}\label{t:theOrder}
\end{table} 
This  ordering extends to the set of all words  (including infinite words) in the usual lexicographic manner. 

 The $\alpha$-alphabet depends only on the first digits of $\ell_0(\alpha)$ and $r_0(\alpha)$, as we now prove.
\begin{Lem}\label{l:alphaAlpha}  Fix $m,n,\alpha$.  Then the  $\alpha$-alphabet depends only on the first digits of $\ell_0(\alpha)$ and $r_0(\alpha)$.  

More precisely,  denote the cylinders of $\ell_0(\alpha)$ and $r_0(\alpha)$ by $(-k, 1), (k', l')$, respectively.    Then the $\alpha$-alphabet is 
\[ \{(k'+j, l')\,|\, j\ge 0\} \cup \bigcup_{1\le l\le l'}\, \{(-k-j, l)\,|\, j\ge 0\}\cup \bigcup_{1\le l< l'}\, \{(j, l)\,|\, j> 0\}.\]
\end{Lem} 
\begin{proof}  Since $\ell_0(\alpha)\in \Delta_{\alpha}(-k,1)$,  all indices  corresponding to cylinders between $\Delta_{\alpha}(-k,1)$ and $x=0$ are certainly in the alphabet.   These indices are
 $(-k-j, 1)$, with $j> 0$.   The pre-images of these cylinders under powers of $C$ are also cylinders of $T_{\alpha}$,  up to and including the $l'^{\text{th}}$ power.     If $l'=1$,  then all  corresponding to cylinders between $0$ and $\Delta_{\alpha}(k',1)$ are in the alphabet.    If $l'>1$,  then all $(j, 1)$ with $j> 0$ are present, and so are all preimages  under powers of $C$ are also cylinders of $T_{\alpha}$,  up to and including the $(l'-1)^{\text{st}}$ power.    The cylinders between the pole of $C^{l'}$ and $\Delta_{\alpha}(k',l')$ have their indices in the language, and we have accounted for all possible indices. 
\end{proof}

 \bigskip 
 \subsubsection{Relating admissibility and orders}\label{sss:admissOrdProofs}

If $x \in \mathbb I_{\alpha}$, then each of $T_{\alpha}(x), T^2_{\alpha}(x), \ldots$ is also in $\mathbb I_{\alpha}$.   This is directly related to the notion of admissibility.  (Experts will note that  we could discuss the following in terms of cylinders of rank greater than one.)
\begin{Lem}\label{l:theLang}  Fix $m,n,\alpha$.   A word 
$A^{k_u}C^{l_u}\cdots A^{k_1}C^{l_1}$, with each $(k_i, l_i)$ in the $\alpha$-alphabet,  is admissible for $\alpha$    if and only if 
\[  \underline{b}_{[1,u-j+1]}^{\alpha}  \preceq (k_j, l_j) \cdots (k_u, l_u) \preceq \overline{b}{}^{\alpha}_{[1, u-j+1]}\]
for each $1\le j\le u$.
\end{Lem} 
\begin{proof}  The forward direction is straightforward as we show first.   By definition,  admissibility   implies that there is an   
$x\in \mathbb I_{\alpha}$ whose $\alpha$-digit sequence begins $(k_1, l_1), (k_2, l_2), \ldots, (k_u, l_u)$.   
Thus,  each 
$A^{k_{j-1}}C^{l_{j-1}}\cdots A^{k_1}C^{l_1}\cdot x$  has $\alpha$-digit sequence beginning  $ (k_j, l_j) \cdots (k_u, l_u)$, and each of these is in 
$\mathbb I_{\alpha}$.     Since  the endpoints of $\mathbb I_{\alpha}$ correspond
to $\underline{b}_{[1, \infty)}^{\alpha}$ and $\overline{b}{}^{\alpha}_{[1, \infty)}$, and  the ordering in the language corresponds 
to the usual ordering of real numbers, the inequalities hold.

On the other hand, if  a word $U$ has its exponents defining letters in the $\alpha$-alphabet all of which are indices of full cylinders, then  it is easy to show that $U$ is admissible for $\alpha$.  Indeed,  given this full cylinder condition,  we can choose any $x_u \in \Delta_{\alpha}(k_u,l_u)$ and then iteratively find some $x_i \in \Delta_{\alpha}(k_i,l_i)$ such that $A^{k_i}C^{k_i}\cdot x_i = x_{i+1}$.   Then $U$ is admissible for $x_1$.

Suppose that there is some positive number $f\le u$ of   non-full cylinder indices among the  $(k_i,l_i)$;  let us enumerate these $(k_{i_1}, l_{i_1}), \ldots, (k_{i_f}, l_{i_f})$.  Each of these  is thus  in $\{ (k, l), (k, l+1), \dots, (k, m-1),  (k', l')\}$ where $(k, l)$ is the $\alpha$-digit of $\ell_0(\alpha)$ and $(k', l')$ that of $r_0(\alpha)$.   Recall that $T_{\alpha}(\, \Delta_{\alpha}(k',l')\,) = [\ell_1, r_0)$ and $T_{\alpha}(\, \Delta_{\alpha}(k,l+j)\,) = [\ell_0, r_1)$ for each $(k,l+j)$ in the $\alpha$-alphabet.  


Since  $ \underline{b}_{[1,u-i_f]}^{\alpha}  \preceq (k_{i_f}, l_{i_f}), (k_{i_{f}+1}, l_{i_{f}+1}), \ldots,  (k_u, l_u) \preceq \overline{b}{}^{\alpha}_{[1, u-i_f]}$,   canceling the respective common first digit, we find that either $\ell_1$ begins with an $\alpha$-digit sequence that is less than or equal to $(k_{i_{f}+1}, l_{i_{f}+1}), \ldots,  (k_u, l_u)$, or else $r_1$ begins with a sequence that is greater than or equal to it.   In both cases, there is thus a subinterval of points $x_{i_f}\in \Delta_{\alpha}(k_{i_f}, l_{i_f})$  whose $\alpha$-digit sequence begins $(k_{i_f}, l_{i_f}), (k_{i_{f}+1}, l_{i_{f}+1}), \ldots,  (k_u, l_u)$.     We now continue iteratively,  noting that if some of the $i_j$ give consecutive integers,  then we will need to cancel more than one digit at steps in this argument.   
\end{proof} 

\bigskip
The following two results  are key tools for proving admissibility of digits by induction, see for example \S~\ref{sss:leftAdmiss}.
\begin{Lem}\label{l:admissBetween2Pts}  Suppose $N\in \mathbb N$ and for some $\alpha', \alpha''$, with  $0\le \alpha'< \alpha''\le 1$, one has $\underline{b}{}^{\alpha'}_{[1,N]} = \underline{b}{}^{\alpha''}_{[1,N]}$.     Then $\underline{b}{}^{\alpha}_{[1,N]} = \underline{b}{}^{\alpha'}_{[1,N]}$  for all $\alpha \in [\alpha', \alpha'']$.   
\end{Lem} 
\begin{proof}  Fix some $\alpha \in (\alpha', \alpha'')$. Recall that for any $\beta\in [0,1]$,  the initial digit of $\ell_0(\beta)$ is $(-k, 1)$ for some $k>0$.   Since the real numbers $\ell_0(\beta)$  strictly increase with $\beta$,   the initial digit of $\ell_0(\alpha)$ is the same as that shared by $\ell_0(\alpha')$ and $\ell_0(\alpha'')$.    Since also the $r_0(\beta)$ are increasing, Lemma~\ref{l:alphaAlpha}  yields that the  $\alpha$-alphabet  contains the $\alpha'$-alphabet.   In particular, each of the $N$ digits in $\underline{d}{}^{\alpha'}_{[1,N]}$ is contained in the $\alpha$-alphabet.   

Again from the increasing nature of the endpoints $\ell_0, r_0$, we find for each $j\le N$   both 
$\underline{b}{}^{\alpha}_{[1,N-j+1]}\preceq   \underline{b}{}^{\alpha''}_{[1,N-j+1]}$ and   $\overline{b}{}^{\alpha'}_{[1,N-j+1]}\preceq \overline{b}{}^{\alpha}_{[1,N-j+1]}$.    Thus by Lemma~\ref{l:theLang}, the admissibility of $\underline{b}{}^{\alpha'}_{[1,N]}$ for $\alpha', \alpha''$ implies the admissibility of this sequence for $\alpha$.   Finally,  $\underline{b}{}^{\alpha}_{[1,N]} = \underline{b}{}^{\alpha'}_{[1,N]}$ for otherwise we would contradict the increasing nature of $\ell_0$ (with respect to one of $\alpha', \alpha''$).
\end{proof} 

\medskip 
The following is proven {\em mutatis mutandi}.
\begin{Lem}\label{l:admissBetween2PtsRvals}  Suppose $N\in \mathbb N$ and for some $\alpha', \alpha''$, with  $0\le \alpha'< \alpha''\le 1$, one has $\overline{b}{}^{\alpha'}_{[1,N]} = \overline{b}{}^{\alpha''}_{[1,N]}$.     Then $\overline{b}{}^{\alpha}_{[1,N]} = \overline{b}{}^{\alpha'}_{[1,N]}$  for all $\alpha \in [\alpha', \alpha'']$.   
\end{Lem} 


\section{Dynamics in the case of $\alpha = 0$  for all signatures }\label{s:AlpZero}  
 
 Calta-Schmidt \cite{CaltaSchmidt} study the dynamics of  what in our notation is $T_{3,n,\alpha}$ with $\alpha=0$,  and of its natural extension.    We briefly generalize that work in this section (our $T(x)$ gives their $g(x)$ upon restricting to $m=3$). 

Fix integers $m,n$ (both greater than 2),   let $\mathbb I := \mathbb I_{m,n,0}$, thus $\mathbb I  = [-t,0]$.  We have
\[ 
\begin{aligned} 
T:= T_{m,n,0}: \mathbb I  &\to \mathbb I \\
           x        &\mapsto A^{-k}C\cdot x\,,
\end{aligned}           
\]
where $k = k(x)$ is the unique positive integer such that $T(x) \in \mathbb I$.         Notice that $T(x) = -k t + 2\cos \pi/m - 1/x$.       Let 
$\Delta_k := \Delta_{\alpha=0}(-k, 1)$.   For $k \ge 2$ we have the full cylinders $\Delta_k = [ \frac{1}{\mu- (k-1) t},\frac{1}{\mu- k t})$; that is, $T$ sends each surjectively onto $\mathbb I$.  
  Setting $\nu = 2\cos \pi/n$, we have that  $\Delta_1 = [\, -t ,  -1/\nu\,) $ and its  image under $T$ is the interval $[-\nu + 1/t,\,0)$. 
The $T$-orbit of $x = -t$ is of central importance,  thus let 
\[\ell_j = T^j(-t)\;  \text{for}\; j = 0, 1, \dots\, .\]

\begin{figure}
\scalebox{.4}{
{\includegraphics{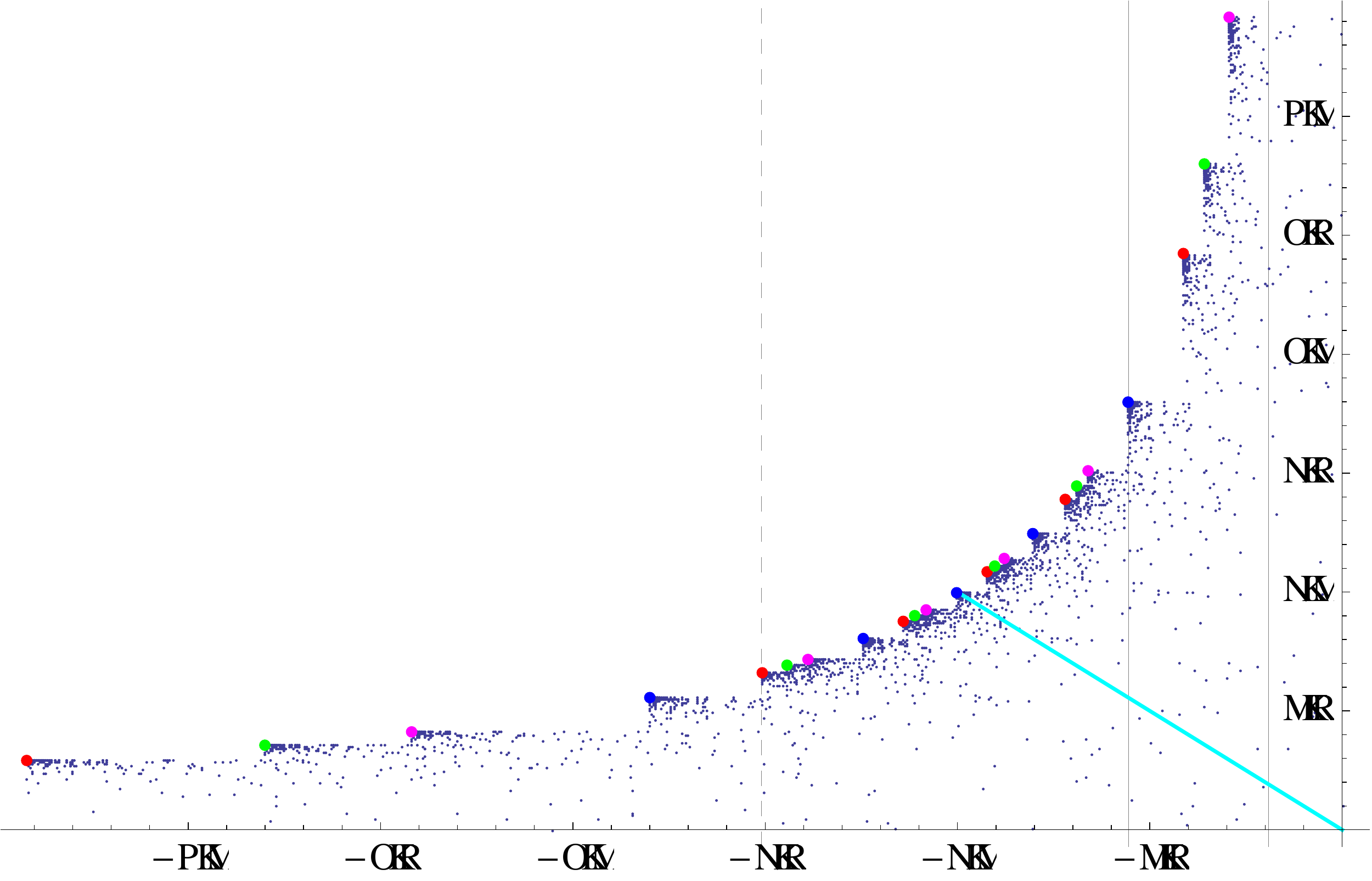}}
}
\caption{A trace of 20,000 consecutive orbit elements to show the natural extension of the determinant plus one map when $(m,n)=(5,7)$ and $\alpha = 0$. Red points: initial orbit of vertex of $x$-coordinate $\ell_0 = - t$; green points: second piece of this orbit; magenta: third piece; blue: fourth and final, before returning to initial red vertex.   Vertical lines: dotted at $\ell_1$, and right ends of cylinders for $k=1$ and $k=2$. Cyan line segment:  the anti-diagonal $y = -x$. Thus, here ${\color{red}\ell_0}< {\color{green}\ell_6}< {\color{magenta}\ell_{12}} < {\color{blue}\ell_{18}} <  {\color{red}\ell_1} <{\color{green}\ell_7}< {\color{magenta}\ell_{13}} <{\color{blue}\ell_{19}} < {\color{red}\ell_2}<{\color{green}\ell_8}< {\color{magenta}\ell_{14}} <{\color{blue}\ell_{20}} < {\color{red}\ell_3}<{\color{green}\ell_9}< {\color{magenta}\ell_{15}} <{\color{blue}\ell_{21}} < {\color{red}\ell_4}<{\color{green}\ell_{10}}< {\color{magenta}\ell_{16}} <{\color{blue}\ell_{22}} < {\color{red}\ell_5}<{\color{green}\ell_{11}}< {\color{magenta}\ell_{17}}$.}
\label{natExtFigFiveSevenZero}
\end{figure}
\bigskip
 
The following element is key to the study of this orbit and therefore to many arguments in this paper. 
\begin{equation}\label{e:Wdefd}  W = A^{-2}C \, (A^{-1}C)^{n-3} \, \big[A^{-2}C (A^{-1}C)^{n-2}\big]^{m-2} \,.
\end{equation}
Since $A^{-1}C  = B$ and $B^n = \text{Id}$ (projectively),   
\[
\begin{aligned}
  W &= A^{-1} B\,    \, B^{n-3} \, \big[A^{-1}B\, B^{n-2}\big]^{m-2}\\
       &= A^{-1} B^{-2} \, \big[A^{-1} B^{-1}\big]^{m-2} \\
       &= A^{-1} B^{-2} \, \big[B C^{-1} B^{-1}\big]^{m-2}\\
       &= A^{-1} B^{-2} \,B \big[C^{-1}\big]^{m-2}B^{-1}\\
       &=A^{-1} B^{-1}  C^{2}B^{-1}\\
       &=A^{-1} B^{-1}  (AB)^{2}B^{-1}\\
       &=A^{-1} B^{-1}   A \, B A \,.
\end{aligned}
\]

Now, since $A \cdot (-t) = 0$, and $B\cdot 0 = \infty$, while $A$ fixes $\infty$, we certainly have that $W$ fixes $x=-t$.
Substituting  $A^{-1}C$ for $B$ gives the form that we will use several times below. 
\begin{equation}\label{e:WiD} 
W = A^{-1} C^{-1}   A C A \,.
\end{equation}

We claim that the right hand side of  \eqref{e:Wdefd} is the admissible word for the corresponding element in the $T$-orbit of $\ell_0=-t$.     That is, all $\ell_j$  lie in $\Delta_1$ other  then $\ell_{n-2 + k(n-1)}$ for $0\le k \le m-3$ and also  $\ell_{2n-4 + (m-3)(n-1)}$; all of these latter orbit entries lie in $\Delta_2$.  Furthermore,   $\ell_{2n-4 + (m-3)(n-1)}$ is the left endpoint of $\Delta_2$, and thus $\ell_{2n-3+ (m-3)(n-1)} = \ell_{0}$.    To justify this,   we show: 
\[\tag{i} \quad (A^{-1}C)^{n-3}\cdot \ell_0 < -1/\nu < (A^{-1}C)^{n-2}\cdot \ell_0\,;\]
\[\tag{ii} \, [A^{-2}C (A^{-1}C)^{n-2}]^{m-2}\cdot \ell_0 = - \nu\,;\]
\[\tag{iii} \, A^{-2}C (A^{-1}C)^{n-2}\cdot (-\nu) = \infty\,.\]
 Since   $A^{-1}C\cdot x$ is increasing and has no pole in $\mathbb I$,  (i) implies that $\ell_j =(A^{-1}C)^{j}\cdot \ell_0$ for $0 \le j \le n-2$ is correct.    Likewise,  (iii) implies that  $x \mapsto A^{-2}C (A^{-1}C)^{n-2}\cdot x$ is increasing on $(-t, -\nu)$, and combined with (ii) that for each $j< m-2$ we have  $\ell_0 < [A^{-2}C (A^{-1}C)^{n-2}]^{j}\cdot \ell_0 < -\nu < -\nu + 1/t = \ell_1$.  This in turn gives the correctness of the various $\ell_j$ corresponding to the sub-words of $[A^{-2}C (A^{-1}C)^{n-2}\big]^{m-2}$.   That the remaining factors of \eqref{e:Wdefd} correspond to the $T$-orbit is easily argued, especially since the fact that $W$ fixes $x=-t$ combines with (ii) to show that $[A^{-2}C \, (A^{-1}C)^{n-3} ]^{-1}\cdot \ell_0 = -\nu$. 
  
Note that (iii) is immediate, as the pole of $A^{-2}C (A^{-1}C)^{n-2}$ is that of $A^{-1}C (A^{-1}C)^{n-2} = B^{-1}$,  and certainly $B\cdot \infty = -\nu$.    As well, 
$[A^{-2}C (A^{-1}C)^{n-2}]^{m-2} = (A^{-1}B^{-1})^{m-2} = (A^{-1}B^{-1})^{-2}$, as $A^{-1}B^{-1}$ is a conjugate of $C^{-1}$ and thus has order $m$.  From this,  we find that 
\[ [A^{-2}C (A^{-1}C)^{n-2}]^{m-2}\cdot \ell_0 = (BA)^2\cdot (-t) = BC\cdot 0 = B\cdot \infty = -\nu\,,\]
 and (ii) also holds.  Finally,  we have $\ell_0 = -t < - \nu < - \nu + 1/t = \ell_1 = B\cdot \ell_0$ and since $B\cdot -1/\nu = 0$,  $B\cdot 0 = \infty$,  $B\cdot \infty = - \nu$, we find that 
 $\ell_0 < B^3\cdot(-1/\nu)< \ell_1$, and now (i) easily follows. 

\medskip 
We thus have that the ordering of the $T$-orbit of $\ell_0$ as real numbers is as given in Table~\ref{t:realOrderOfEllOrbit}.  Note that the orbit elements contained in $\Delta_2$ are found as final entry from each column. 
\bigskip 

\begin{table}
\[ 
\begin{aligned}
& \ell_0&< \ell_{n-1}&< \ell_{2n-2} &< \cdots&<\ell_{(m-2)(n-1)}    \\
 < \,&\ell_1&< \ell_{n}   &< \ell_{2n-1} &< \cdots &<\ell_{(m-2)(n-1)+1} \\
 \vdots &&&\vdots&\vdots\\\
< \,&  \ell_{n-3}&< \ell_{2n-2}   &< \ell_{3n-3} &< \cdots &<\ell_{(m-2)(n-1)+n-3} \\
< \,&  \ell_{n-2}&< \ell_{2n-3}   &< \ell_{3n-4} &< \cdots<\ell_{(m-3)(n-1)+n-2} &&\\
 \end{aligned}
 \]
 \medskip 
 
\caption{The $T$-orbit of $\ell_0$ ordered as real numbers, when $\alpha = 0$.}\label{t:realOrderOfEllOrbit}
\end{table}

\bigskip

We let $\Omega$ be the the union of  $mn - m - n$ rectangles whose bases lie on the $x$-axis with endpoints being  consecutive elements in the orbits of $\ell_0$ under the real ordering, beginning with $[\ell_0, \ell_{n-1}]$,  along with $[\ell_{(m-3)(n-1)+n-2}, 0]$,  and whose heights we label $L_i, 1 \le i \le mn - m - n$, also in accordance with the real order of the bases.   

Set $L_1 = 1/t$; 
 $L_{m-1+i} = (R A^{-1} C R^{-1}) \cdot L_{i}$ for any $1\le i < (m-1)(n-2)$;  and, 
$L_{i} =  (R A^{-2} C R^{-1}) \cdot L_{i-1 +(n-2)(m-1)}$ for $2\le i \le m-1$.  Since these relations accord with $\mathcal T = \mathcal T_{m,n,0}$,   and $-t$ being fixed by $W$ gives that $1/t$ is fixed by $RWR^{-1}$, we have that the left upper vertex $\Omega$ is $(-t, 1/t)$ and in fact  the left upper vertex of the $i^\text{th}$ rectangle is  $(-1/L_i, L_i)$, thus showing that $\Omega$ has infinite $\mu$-measure.   We have seen that $\ell_{mn - m - n-1} = -1/\nu$,  and hence find that 
the rightmost element of that orbit has value 
\[\ell_{(m-3)(n-1)+n-2} = [(A^{-1}C)^{n-3}A^{-2}C]^{-1}\cdot (-1/\nu) =  B^{-1}A B^3 \cdot(-1/\nu) = B^{-1}A\cdot (-\nu) = -1/t\,.\]  
Therefore,  the rightmost rectangle has height $L_{mn - m - n} = t$.  One can show  that   $\mathcal T$ is bijective on $\Omega$ up to a set of measure zero. We will return to this point in our next paper.  

\bigskip
\begin{Rmk} 
In  \cite{CaltaSchmidt}, an acceleration of the $\alpha=0$ interval map (when $m=3$) is defined.  This new interval map is shown to be ergodic with respect to a finite invariant measure (that is absolutely continuous with respect to Lebesgue measure).    
\end{Rmk}

\section{Dynamics in the case of  $\alpha = 1$ for all signatures}\label{s:alpOne}  

\begin{figure}[h!]
\noindent
\parbox{1 in}{
   \scalebox{.22}{ \includegraphics{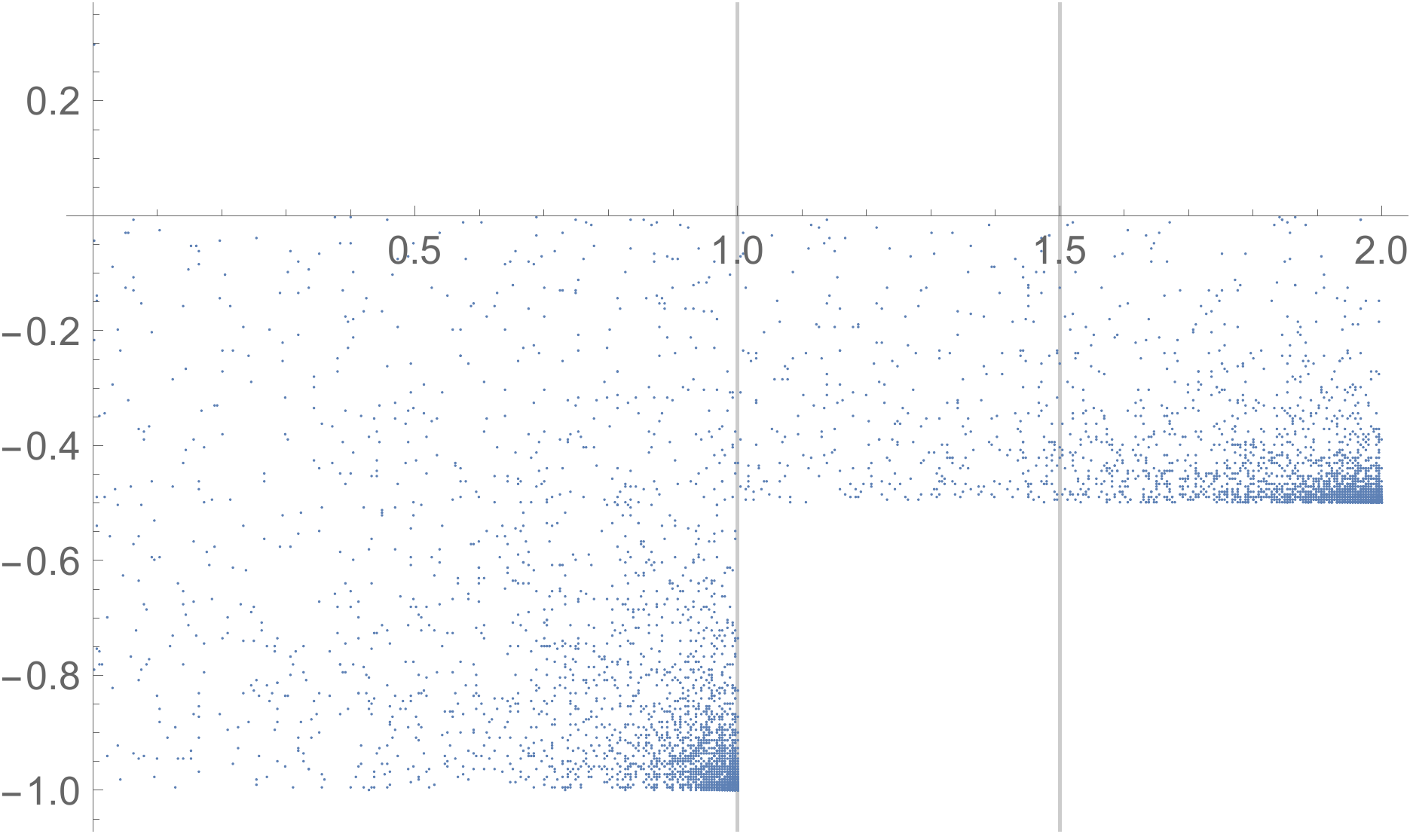}}
}%
\qquad \qquad \qquad \qquad \qquad \qquad
\begin{minipage}{1in}%
\scalebox{.22}{\includegraphics{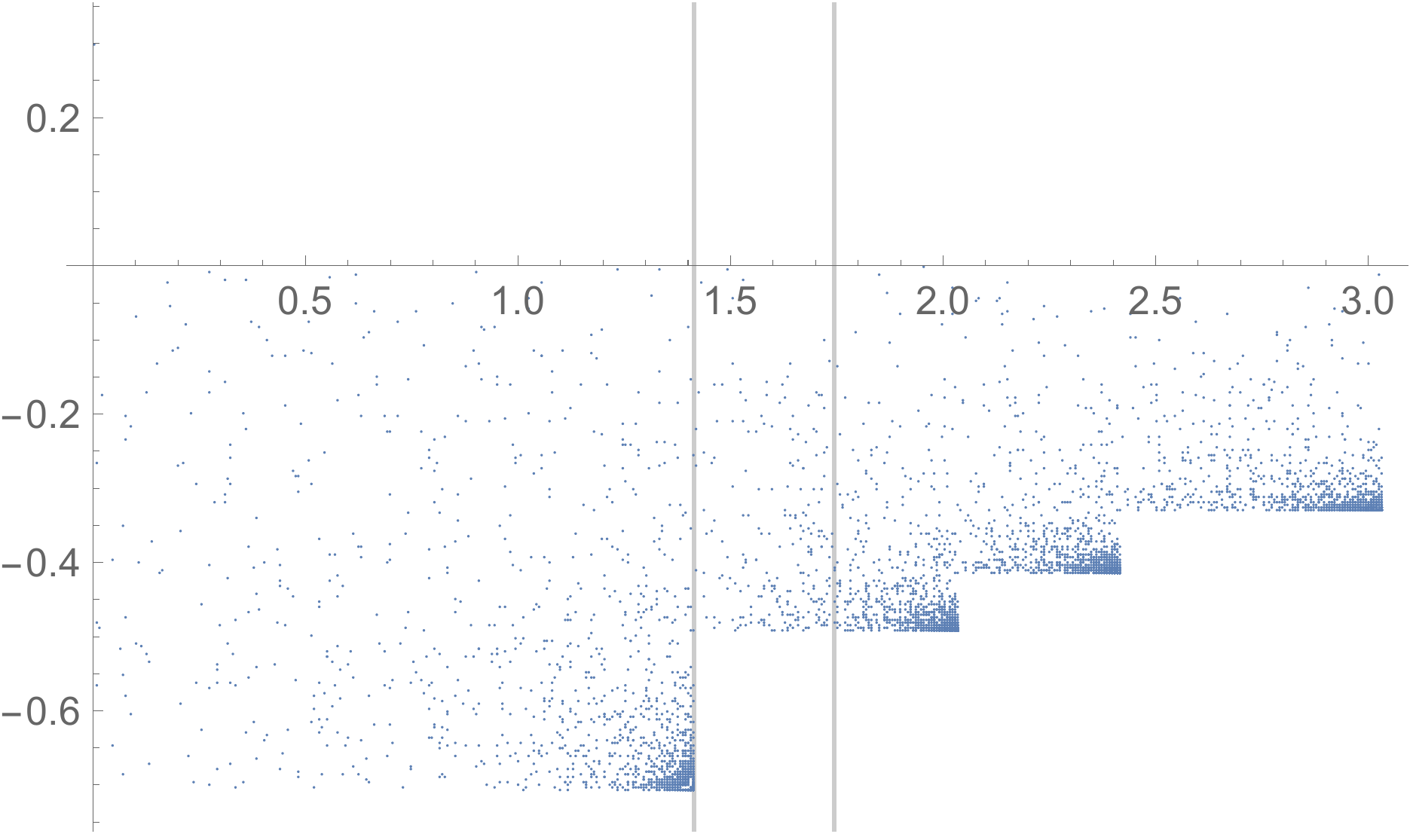}}
\end{minipage}%
\caption{Plots showing the domains of the natural extensions for $(m,n) = (3,3)$ and $(m,n)= (4,5)$ when $\alpha = 1$. Vertical lines marked at $x = \mu$ and $x= \mu + 1/t$ in both cases.}%
\label{f:Alp1NatExtPlots}%
\end{figure}

  This case is dominated by 
\begin{equation}\label{eq:Uee}
U = A C^{m-2} (AC^{-1})^{n-2}\,.
\end{equation}
We now let the interval be $\mathbb I := \mathbb I_{m,n,\alpha=1} = [0, t)$.  
\bigskip

\begin{Prop}\label{p:veeOnTheRight}   For all $m,n\ge 3$,  the map $T_{m,n,1}$ has
\begin{enumerate}
\item $U$ as the admissible word for the  orbit of $t = r_0(t)$;
\item exactly one non-full cylinder, $\Delta(1, n-1) = [\mu + 1/t, t]$;
\item full cylinders of the form $\Delta(k, l)$ for all $l \in \{1, \dots , m-2\}$ and  $k \in \mathbb N$,  as well as      $\Delta(k, m-1)$ for    $k>1$.
\end{enumerate}
\end{Prop}

\begin{proof} 
 We first claim that all possible exponents of $C$ are seen, thus that for each $l \in \{1, \dots , m-1\}$, there are non-empty cylinders $\Delta(k, l) =\Delta_{1}(k, l)$ and in fact for each $l$ these are indexed by $k \in \mathbb N$.  A first application of $C$ sends all of $(0, 1/\mu)$ to negative values, and in particular outside of  $\mathbb I$.   As well, since $1/\mu$ is sent to $0$,  these images are brought back into  $\mathbb I$ by positive powers of $A$.  Furthermore, this shows that each of the cylinders $\Delta(k, 1)$, $k\in \mathbb N$ is full.   Therefore, also  now the cylinders $\Delta(k, l)$, $k\in \mathbb N$ are full for all $1\le l \le m-2$.  
 
We turn attention to the case of $l= m-1$.   In fact, there are full cylinders $\Delta(k, n-1)$ for all  $k>1$ and the sole non-full cylinder is $\Delta(1, n-1)$, as we now briefly explain.    Since $C A^{-1}\cdot 0 = \mu + 1/t<t$,  one has that $\Delta(1, n-1) = [\mu + 1/t, t]$.   As well,  since of course $A^2C\cdot (\mu + 1/t) = A\cdot 0 = t$, the cylinders $\Delta(k, n-1)$ for all  $k>1$ are all indeed full (and their union is $[\mu, \mu + 1/t]$).    It remains only to consider the orbit of $t$, which we know begins by $t \mapsto A C^{m-1}\cdot t$.   We observe that $A C^{m-1} = A C^{-1} = A B^{-1}A^{-1}$, and is thus clearly an elliptic matrix of order $n$.      We now translate by $A$ so as to consider the 
 the orbit of $0$ under powers of $B^{-1}$.  This elliptic matrix fixes a point of real part $-\nu/2$ and rotates a hyperbolic $n$-gon;  one also easily verifies that $-\nu \mapsto \infty \mapsto 0 \mapsto -1/\nu$ is part of the orbit of the vertices of the $n$-gon.   The predecessor of $-\nu$  is $B\cdot -\nu = -\nu+1/\nu$.  We translate back to $\mathbb I$ by $A$, and thus this predecessor corresponds to $\mu + 1/\nu$, a value that is visibly greater than $\mu + 1/t$.      In conclusion,  the $T_{m,n,1}$-orbit of $t$ begins with $(AC^{m-1})^s\cdot t$ for $1\le s \le n-2$.    But,   $(AC^{m-1})^{n-2}\cdot t = (C A^{-1})^2\cdot t = CA^{-1} C\cdot 0 = C \cdot \infty = \mu$.  That is,  the orbit of $t$ reaches the right endpoint of the (full) cylinder $\Delta(2, m-1)$, and thus thereafter returns to $t$.  
\end{proof}

\section{Orbit synchronization on the interval $\alpha < \gamma_{3,n}, n\ge 3$} 
We define $\gamma_{3,n}$ as the value of $\alpha$ such that 
\[C^{-1}\cdot \ell_0(\gamma_{3,n}) = r_0(\gamma_{3,n}).\]    In particular, for all $\alpha \le \gamma_{3,n}$,  the point $\mathfrak b_{\alpha}$ lies outside of $\mathbb I_{\alpha}$.  Since  $\ell_0(\alpha)<0$ for all $\alpha<1$,  and $m=3$,  it follows that $0<r_0(\gamma_{3,n})<1$.    We define $\epsilon_{3,n}$ such that 
\[A^{-1}C\cdot \ell_0(\epsilon_{3,n}) = r_0(\epsilon_{3,n}).\]    One finds that $\ell_1(\alpha) = A^{-1}C\cdot \ell_0(\alpha)$ holds for all $0\le \alpha < \epsilon_{3,n}$.    Elementary calculations show that for all $n\ge 3$, $r_0(\epsilon_{3,n}) \ge (1+\sqrt{5})/2$, and thus this equation holds in particular for all $\alpha < \gamma_{3,n}$.

\begin{figure}[h]
\scalebox{.5}{
{\includegraphics{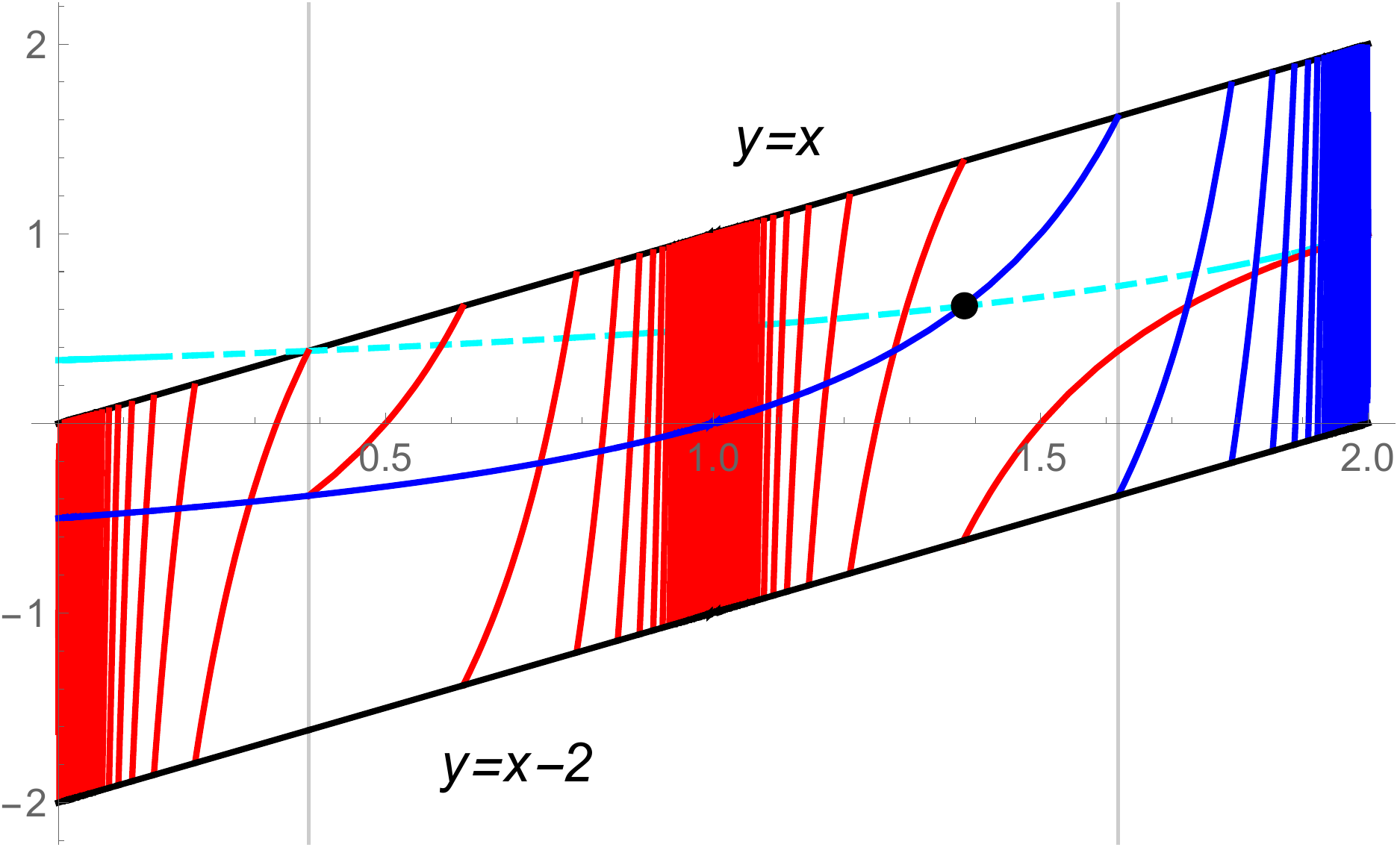}}}
\caption{The first image of the endpoints.    In blue is the graph of $x\mapsto T_{3,3,\alpha}(x-t)$, with $x = \alpha t$; the blue curves thus give the values of $\ell_1(\alpha)$. In red that of $x\mapsto T_{3,3,\alpha}(x)$; the red curves thus give the values of $r_1(\alpha)$.    (Here $t= t_{3,3} = 2$.) 
Gray vertical lines at $x = \gamma_{3,3} t = (G-1)^2, x= \epsilon_{3,3} t = G$,  where $G= (1 +\sqrt{5})/2$.  The dotted cyan curve is the image of the left endpoint under $C^{-1}$ (giving the values of $\mathfrak b_{\alpha}$) and thus  for $\alpha t > \gamma_{3,3}$, above this curve a next digit begins with a $C^2$. Large dot has $x$-coordinate $ \delta_{3,n} t$, see \S\ref{s:LargeAlps}. }
\label{firstBothDigs}
\end{figure}

 In this section we prove the following.
\begin{Thm}\label{t:FullMeasureSmallAlps}   For $m = 3$ and $n\ge m$,  
the set of $\alpha \in (0, \gamma_{3,n})$ such that there exists $i = i_{\alpha}, j = j_{\alpha}$ with 
$T_{3,n,\alpha}^{i}(\,r_0(\alpha)\,) = T_{3,n,\alpha}^{j}(\,\ell_0(\alpha)\,)$ is of full measure.  
\end{Thm} 

\subsection{Right cylinders and (potential) synchronization intervals}  Basic motivation for our approach to synchronization of the $T_\alpha$-orbits of $r_0(\alpha)$ and $\ell_0(\alpha)$, with $\alpha < \gamma$, comes from the following.   We will eventually show for this range of our parameters that synchronization  depends on right and left digits being related by 
\begin{equation}\label{eq:cInvAc}
 C^{-1}A C = \begin{pmatrix} 1&0\\-t & 1  \end{pmatrix} \,.
\end{equation}

\begin{Lem}\label{l:leftRegionOneBeforeSyn}   Fix $m=3, n \ge 3$,  $\alpha < \gamma$, and $i,j \in \mathbb N$.   Suppose that  $\ell_{i-1} = C^{-1}AC\cdot r_{j-1}$.    
Then 
\begin{enumerate}
\item   $\ell_i = r_j$, 
\item $\ell_{i-1} \ge r_{j-1}$ with equality if and only if both equal zero.
\end{enumerate}
\end{Lem} 
\begin{proof}  There is some 
$u \in \mathbb Z\setminus\{0\}$, that $\ell_i = A^uC \cdot\ell_{i-1} $  (recall that since $\alpha< \gamma_{3,n}$,  the exponent $l$ is always one).  This thus equals  $A^uC \;C^{-1}AC\cdot r_{j-1} = A^{u+1}C\cdot r_{j-1}$.    
In particular  $A^{u+1}C\cdot r_{j-1} \in \mathbb I_{\alpha}$.    We conclude that 
$r_j = A^{u+1}C\cdot r_{j-1}\,$.   From this, we find $\ell_i = r_j$.     
Since $C^{-1}AC$ clearly fixes zero, see~\eqref{eq:cInvAc},  it follows that $\ell_{i-1} = r_{j-1}$ holds if either is zero;  otherwise,  the fact that the exponent of $A$ is greater when passing from $r_{j-1}$ to $r_j$  than from $\ell_{i-1}$ to $\ell_i$ (in light of the ordering of digits, \eqref{e:theFullOrder})   shows that  $\ell_{i-1}> r_{j-1}$.   
\end{proof}

\bigskip

\begin{Def}\label{d:synchInts}   We say {\em synchronization} occurs at $\alpha$ if there exist $i,j$ such that $r_j =  \ell_i$.  A 
{\em synchronization   interval}  is an interval of $\alpha$ values for each which synchronization holds with the same pair of indices $i,j$. (We will assume that at least for one $\alpha$ in the interval, both $i,j$ are minimal.) 
\end{Def}
\bigskip

 \begin{Eg}\label{e:boringEg}  Fix $n=3$.   Let $\zeta = (5 - \sqrt{21})/2$ and $\eta =(-1 + \sqrt{21})/10$.  (Thus $\zeta = 0.20871\cdots$, $\eta = 0.35825\cdots\;$.)  One finds that for all $\alpha \in [\zeta, \eta)$: 
\[
 \overline{d}_{[1, \infty)}^{\alpha}      = 1, \cdots\;\;\;\text{and}\;\;   \underline{d}_{[1, \infty)}^{\alpha}     = -1,-2,-2, -1, \cdots
 \]
 and furthermore, for all such $\alpha$ the group identity given in  Lemma~\ref{l:shortRightId}   implies  $r_2(\alpha) = \ell_5(\alpha)$.  

Note that here $AC\cdot r_0(\zeta) = \ell_0(\zeta)$.  Confer Figure~\ref{firstBothDigs},  in which $ 2 \alpha = 2 \zeta$ gives the intersection point with the line $y = x-2$ of the first red branch  to the left of $x = 2 \, \gamma_{3,3}$.  

   Also in Figure ~\ref{firstBothDigs},  the visible intersection of this red branch with the blue branch (which corresponds to $\ell_1 = A^{-1}C\cdot \ell_0$) marks a point of what one could call ``accidental" synchronization.  That is, for $\alpha =  (2 - \sqrt{2})/2 = 0.29289\cdots\,$, we have $r_1(\alpha) = \ell_1(\alpha)$.    Of course, this implies that each of these is periodic.  In particular, this and indeed any  accidental synchronization occurs  at an algebraic value of $\alpha$. 
 
\end{Eg}

\bigskip 
 We seek  synchronization   intervals of the form $[\zeta, \eta)$, where the endpoints are identified by $R\cdot r_0(\zeta) = r_0(\zeta)$ and $LA  \cdot \ell_0(\eta) = r_0(\eta)$, for certain $R, L \in G_{3,n}$.   
Our synchronization intervals form a subset of full measure; to prove this, it will be very helpful to have the digits of the $\zeta, \eta$.   The following is key to finding these digits. 
    
\begin{Lem}\label{l:etaDigits}   Fix $m=3, n \ge 3$, an interval $[\zeta, \eta] \subseteq (0, \gamma_{3,n})$  and  $i,j \in \mathbb N$.  Suppose  that there are  $R, L, L'\in G_{3,n}$ (none of which is the identity) such that 
\begin{enumerate}
\item[(a)]   $L = C^{-1}AC R$,
\item[(b)]   $R\cdot r_0 = r_{j-1}$ and $L'\cdot \ell_0 = \ell_{i-2}$,  for all $\alpha \in [\zeta, \eta]$, 
\item[(c)]  $L A\cdot \ell_0  = \ell_{i-1}$ for all $\alpha \in [\zeta, \eta)$,  while  $L A\cdot \ell_0(\eta) = r_0(\eta)$, 
\item[(d)] $R \cdot r_0(\zeta) = \ell_0(\zeta)$.
\end{enumerate}

Then 

\begin{enumerate}
\item[(i)]  $\ell_{i-1}(\eta) = A^{-1} L A\cdot \ell_0(\eta) = \ell_0(\eta)$,
\item[(ii)]  $r_{j}(\eta) = A^{k+1}C\cdot r_{j-1}(\eta) = r_1(\eta)$, where $k$ is such that $A^kC \cdot r_0(\eta) = r_1(\eta)$
\item[(iii)]  $A^{-2}C L A\cdot \ell_0(\zeta) = \ell_1(\zeta)$.
\end{enumerate}
\end{Lem} 
\begin{proof}  For any $\alpha$, the identity (a) implies   $LA\cdot \ell_0 = C^{-1}AC R\cdot r_0$.  

  Recall that for all $\alpha$, $r_0(\alpha) \notin \mathbb I_{\alpha}$. Now set $\alpha = \eta$.  Hypotheses  (b,c) imply that $\ell_{i-1}  = A^{-1} L A\cdot \ell_0  = \ell_0$.   Now,  if $r_1  = A^kC \cdot r_0 $, then $r_1 = A^kC L A\cdot \ell_0 $, which again by (a)  gives $r_1  = A^kC \; C^{-1}AC R \cdot r_0  = A^{k+1}C R \cdot r_0$.  Now (b) gives $A^{k+1}C  \cdot r_{j-1} = r_1$.    

Finally, (d) with $\ell_1 = A^{-1}C \cdot \ell_0$  yield  $\ell_1(\zeta) = A^{-1} CR\cdot r_0(\zeta)$.  Hypothesis (a) now yields that $A^{-2}C LA\cdot \ell_0(\zeta) = \ell_1(\zeta)$.
\end{proof}

\begin{figure}[h]
\scalebox{.4}{
{\includegraphics{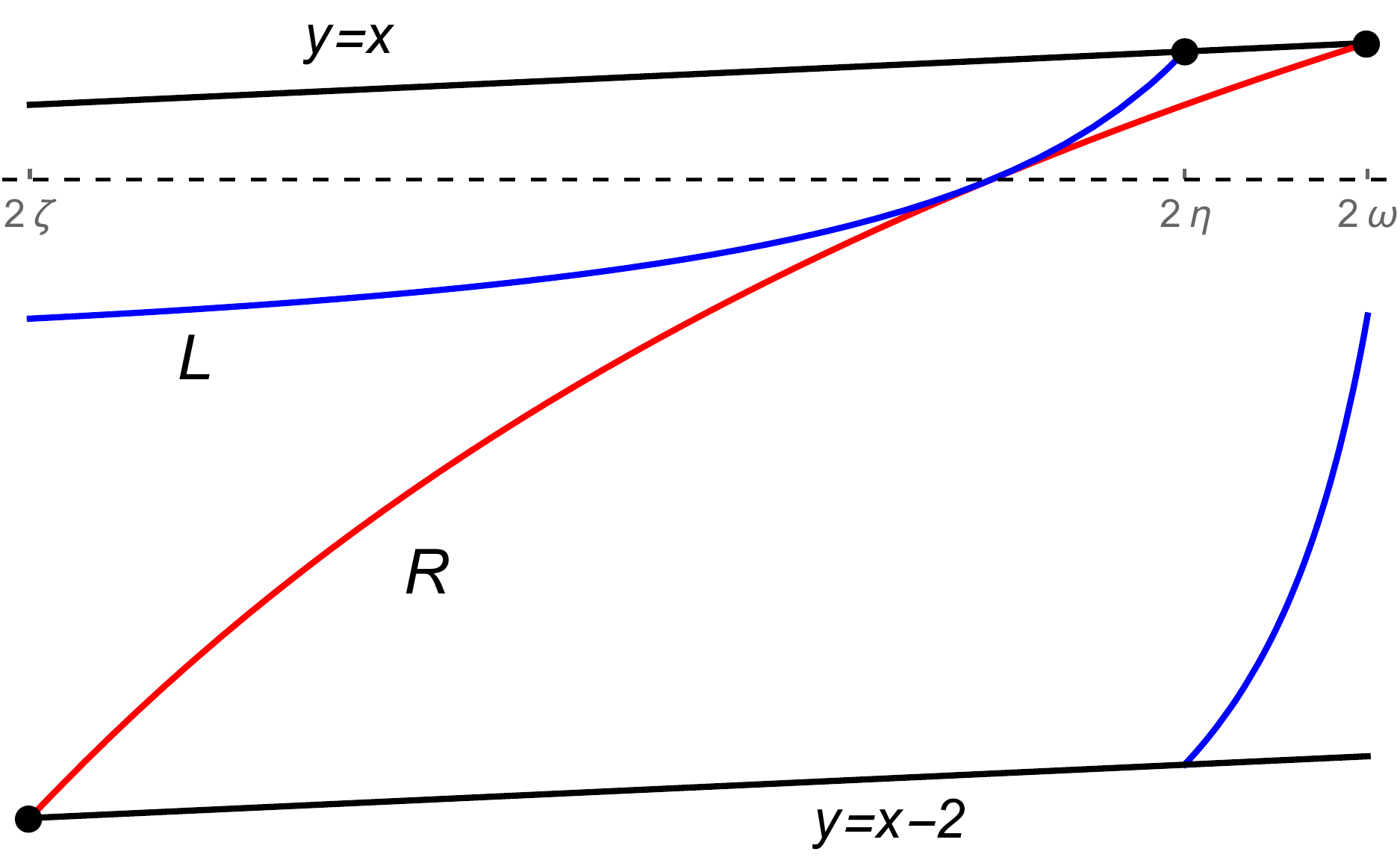}}}
\caption{Determining the synchronization interval $[\zeta, \eta) = [\zeta_{k,v}, \eta_{k,v})$.  
Here,   $m=3, n=3$,  and $k=1$, $v = 1$.  The labels $L, R$ mark respectively the curves $y = L_{1,1}\cdot r_0(\alpha), y= R_{1, 1}\cdot r_0(\alpha)$ where $\alpha = x/2 = x/t_{3,3}$.  (See Definitions ~\ref{d:upperDigsCylsMats}  and ~\ref{d:derivedWords} for $R_{k,v}$  and $L_{k,v}$ in general.)   Red gives the single branch of $y = r_1(\alpha)$ while blue colors the two  branches of  $y = \ell_4(\alpha)$ for  $(5 - \sqrt{21})/2 < x <  g^2$.   The $x$-axis is shown as a dotted line.}
\label{smallAlpLRpair}
\end{figure}
 
\bigskip

We will be describing synchronization subintervals of $\alpha \in [0,1]$ in terms of common initial portions of the digits of $r_0(\alpha)$. 
\begin{Def}\label{d:digitsWithKandA} 

\begin{enumerate}
\item  If the $\alpha$-digits for some $x$ are all of the form $(k_i, 1)$, it is convenient to suppress the notation indicating that the exponent of $C$ is simply one. 
We refer then to {\em simplified digits}, and uniformly use a $d$ instead of a $b$ in notation referring to simplified digits.  Thus the statement $\underline{d}_{[1, \infty)}^{\alpha} =   k_1,  k_2,\dots$ is equivalent to 
$\underline{b}_{[1, \infty)}^{\alpha} =  (k_1, 1)(k_2,1)\dots \,$ and similarly for expressions involving  $\overline{d}$.  Of course, sequences of simplified digits are ordered by way of the order \eqref{e:theFullOrder}.

 \item Given $s\in \mathbb N$ and integers $c_1, c_2, \dots, c_s$ and
$d_1, d_2, \dots, d_{s-1}$, let $v = c_1d_1\cdots d_{s-1} c_s$  and for any $k \in \mathbb N$, define 
the upper (simplified)  {\em digit sequence} of  $k$ and $v$ as   
\[ \overline{d}(k,v) = k^{c_1}, (k+1)^{d_1},\cdots,  (k+1)^{d_{s-1}},k ^{c_s}\,. \]
Let now  $\sigma$ be the left-shift.  (For example,   $\sigma(v) =  d_1 c_2 d_2\cdots d_{s-1} c_s$.)

\end{enumerate}
\end{Def}
 
\bigskip
 
For any $s\in \mathbb N$, let $\underline{b}_{[1, s]}$ denote the prefix of length $s$ of $\underline{b}_{[1, \infty)}^{\alpha}$, and similarly for $\overline{b}_{[1, s]}$.

\bigskip
\begin{Def}\label{d:upperDigsCylsMats}   
  For each $k \in \mathbb N$ and each  $v = c_1d_1\cdots d_{s-1} c_s$,  we define the following.  
\begin{enumerate}


\medskip 
\item  The  length of $\overline{d}(k,v)$ is $\overline{S}(v) := |\,\overline{d}(k,v)\,|  =  c_s+\sum_{i=1}^{s-1}\, (c_i +d_i)$.
  Notice that $\overline{S}(v)$  is indeed independent of $k$.

\medskip 
\item  The   {\em $\alpha$-cylinder} of  $k, v$ is 
\[ \mathscr I_{k,v} = \{ \alpha \,|\, \overline{d}{}^{\alpha}_{[1, \overline{S}(k,v)]} = \overline{d}(k,v)\}\,.\]
That is, $\mathscr I_{k,v}$ is the set of all $\alpha$, such that the initial simplified digits of $r_0(\alpha)$ are 
$k^{c_1}, (k+1)^{d_1}\cdots (k+1)^{d_{s-1}}, k^{c_s}$.

\medskip 
\item The {\em right matrix} of  $k, v$ is
\[ 
R_{k,v}  =    (A^kC)^{ c_s}\; (A^{k+1}C)^{d_{s-1}}(A^kC)^{c_{s-1}}\cdots (A^{k+1}C)^{d_1} (A^kC)^{c_1}\,.
\]

\item The {\em potential synchronization interval} associated to $k, v$ is $\mathscr J_{k,v} = [\zeta, \eta)$ where 
 $\zeta = \zeta_{k, v}$ and $\eta = \eta_{k,v}$ are such that 
\[ R_{k,v}\cdot r_0(\zeta) = \ell_0(\zeta) \;\;\; \text{and} \;\;\;   r_0(\eta)= C^{-1}ACR_v\cdot r_0(\eta)\,.\]   
 
 \end{enumerate}
  
\end{Def}
 
 \bigskip 
 \noindent
 Note that if $\alpha \in \mathscr J_{k,v}$, then $R_{k,v} \cdot r_0(\alpha) = r_{\overline{S}(v)}(\alpha)$.

\bigskip

\subsection{Tree of words and  a partition} From Lemma~\ref{l:etaDigits}, we have that 
\[\overline{d}_{[1, \infty)}^{\eta_{k,v}} = k^{c_1}, (k+1)^{d_1},\cdots,  (k+1)^{d_{s-1}},k ^{c_s},k+1,\cdots\]
and furthermore,  this sequence continues with the digits of $r_1(\eta_{k,v})$.  Thus, this sequence is periodic with pre-period $\overline{d}(k,v)$ and period $k+1, k^{c_1-1}, (k+1)^{d_1},\cdots,  (k+1)^{d_{s-1}},k ^{c_s}$.  This  period is expressible in terms of the word $v'$, which we now define.
\begin{Def}\label{d:periodFromEtaAsWord}   
For each $s > 1$ and each   word $v = c_1d_1\cdots c_{s-1}d_{s-1} c_s$, define 

\[ v' =  \begin{cases} 1 (c_1 - 1)d_1c_2 \cdots c_{s-1}d_{s-1} c_s &\text{if}\;\;\; c_1\neq 1\,,\\
                               (d_1+1) c_2 \cdots c_{s-1}d_{s-1} c_s &\text{otherwise}\,.
                              \end{cases}
 \]
We interpret this also to mean that when $v = c$ with $c>1$ then   $v' = 1 (c -1)$, and when $v=1$ then $v' = 1$.
\end{Def}

As necessary, we extend the notion $\overline{d}(k, v)$ in the natural manner to include the setting of infinite  words,  and also extend the notion of $R_{k, v}$ to include more general words.  
\begin{Lem}\label{l:etaRightDigs}  Let $k \in \mathbb N$ and  $v = c_1d_1\cdots c_{s-1}d_{s-1} c_s$.   If
$\eta_{k,v}\in \mathscr I_{k,v}$, then $\overline{d}{}^{\eta_{k,v} }_{[1,\infty)} = \overline{d}(k, v (v')^\infty)$.    
\end{Lem} 
\begin{proof}   
For simplicity,  let $\eta = \eta_{k,v}$.   From Lemma~\ref{l:etaDigits} (ii),  the simplified digit of $r_0(\eta)$ following $\overline{d}(k, v)$ is $k+1$, and thereafter the simplified digits begin with the sequence of $r_1(\eta)$.   
This can be expressed as $\overline{d}{}^{\eta}_{[1,\infty)} = \overline{d}(k, v)\, \overline{ (k+1)\, \overline{d}(k, v)_{[2, \overline{S}]}}$.   

When $v = c_1 = c$ we have $\overline{S} = c$ and we find $\overline{d}{}^{\eta }_{[1,\infty)} =  k^c,\, \overline{ (k+1), k^{c-1}}$ (when $c=1$, we take this to mean  $k, \overline{k+1}$.)
For longer $v$,  we must  group the new occurrence of $k+1$.   This grouping depends on whether $c_1 = 1$ or not.   In either case,  one indeed finds that $\overline{d}{}^{\eta_{k,v} }_{[1,\infty)} = \overline{d}(k, v (v')^{\infty})$.   
\end{proof}

\bigskip

\begin{Def}\label{d:thetaQ}   Set $\Theta_{-1}(c_1) = c_1+1$ and $\Theta_q(1) = 1q1 $ for $q \ge 1$.  For  $c>1$,  set  $\Theta_q(c) = c [ 1 (c-1)]^q 1c$ for any $q \ge 0$.  (To avoid double labeling and also to stay within our desired language,   $\Theta_0(1)$ is undefined;   note that $\Theta_1(1) = 111$, compare with $\Theta_0(c) = c 1 c$ for $c>1$.)

We now recursively define values of the operators  $\Theta_q$.  Suppose $v = \Theta_p(u) = u v''$ for some $p\ge 0$ and some suffix $v''$.   Then   define for any $q\ge 0$ 
\[ \Theta_q(v) = v (v')^q v'' \,.\]


\end{Def}

\bigskip
\begin{Def}\label{d:descendents}   Let $\mathcal V$ denote the set of all words obtainable from $v=1$ by finite sequences of applications of the various $\Theta_q$.   We call $v$ the {\em parent} of each $\Theta_q(v)$, and also refer to $\Theta_q(v)$ as a {\em child} of $v$.   See Figure~\ref{f:theTree} for a portion of this directed tree.  
 \end{Def}

\bigskip
 \begin{figure} 
 \begin{tikzpicture}[baseline= (a).base]
 \node[scale=.5] (a) at (0,0){
\begin{tikzcd}[column sep=2pc,row sep=2pc]
    &&1\ar[pos=0.5]{rr}{\Theta_{-1}}  \ar[pos=0.5]{ld}{\Theta_q}\ar[pos=0.5, swap]{lld}{\Theta_s}\ar[pos=0.5]{rd}{\Theta_1}
                                              & & 2\ar[pos=0.5]{r}{\Theta_{-1}} \ar[pos=0.7]{rd}{\Theta_{u-1}}\ar[dotted,  -]{d}& 3\ar[pos=0.5]{r}{\Theta_{-1}} \ar[pos=0.7]{rd}{\Theta_0} &\cdots\\%
1s1\ar[pos=0.5]{d}{\Theta_t}     & 1q1 \ar[pos=0.5]{d}{\Theta_s}  & \cdots&111\ar[pos=0.5]{dr}{\Theta_u}  &\cdots&2 (11)^{u-1}12&313\ar[pos=0.5]{d}{\Theta_u} \\
 1s1[(s+1) 1]^t s 1&1q1[(q+1) 1]^s q 1\ar[pos=0.5]{d}{\Theta_t}\ar[bend left, dotted,swap]{ul}{\mathscr D}\ar[pos=0.5]{r}{\Theta_0}&1q1[(q+1) 1]^sq1[(q+1) 1]^s  q 1&            &111(21)^u11   &&   313 (1213)^u13\ar[bend left, dotted]{ul}{\mathscr D}\\
                             &\makebox[\widthof{reallyverylong so moves over to the right please do moveIdont seeit moving at all}]{$1q1[(q+1) 1]^s q 1 \{[(q+1) 1]^{s+1} q 1\}^t[(q+1) 1]^s q 1$}\ar[bend left, dotted]{ul}{\mathscr D}
\end{tikzcd}
};
\end{tikzpicture}
\caption{Each vertex of the directed tree $\mathcal V$, see Definition~\ref{d:descendents},    has countably infinite valency.   A small portion of $\mathcal V$, with some values of  $\mathscr D$ (of  Definition~\ref{d:derivedWords}), is indicated.}
\label{f:theTree}
\end{figure}
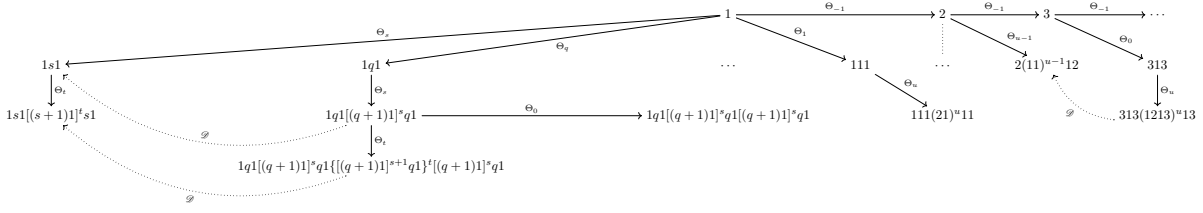

The following result gives the basic structure of the collection of potential synchronization intervals. 
\begin{Thm}\label{p:cantorSet}  We have the following partition
\[(0, \gamma_{3,n}) = \bigcup_{k=1}^{\infty}\, \mathscr I_{k,1}\,.\]

Furthermore, for each $k \in \mathbb N$ and each $v \in \mathcal V$,  the following is a partition: 
\[ \mathscr I_{k,v} = \mathscr J_{k,v} \cup \, \bigcup_{q=q'}^{\infty}\, \mathscr I_{k, \Theta_q(v)}\,,\]
where $q' = 0$ unless $v= c_1$, in which case $q' = -1$.
\end{Thm} 
We postpone the proof of this theorem until page \pageref{proofOfTheorem}.

\bigskip

When $n=3$, the first statement of the theorem describes the partition given by the intervals of definition of the leftmost red branches of Figure~\ref{firstBothDigs}.  See also Figure~\ref{f:PartI}.

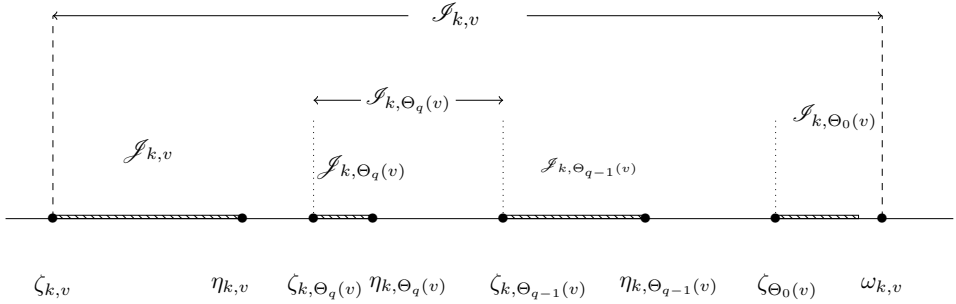
\begin{figure}[h] \scalebox{.9}{
\begin{tikzpicture}[x=3.5cm,y=5cm] 
\draw  (-0.2, 0)--(3.8, 0); 
 
\draw[thin,dashed] (0,0)--(0, 0.6); 
\draw[thin,dashed] (3.5,0)--(3.5, 0.6); 
 
\node at (0.4, 0.2) {$\mathscr J_{k,v}$};  
\node at (1.3, 0.15) {$\mathscr J_{k,\Theta_q(v)}$};          
\node at (2.26, 0.15){\tiny{$\mathscr J_{k,\Theta_{q-1}(v)}$}};  
\node at (3.3, 0.3){$\mathscr I_{k,\Theta_{0}(v)}$};

\node at (0, -0.2) {$\zeta_{k,v}$};     
\node at (0.75, -0.2) {$\eta_{k,v}$};         
\node at (1.15, -0.2) {$\zeta_{k,\Theta_{q}(v)}$};   
\node at (1.5, -0.2) {$\eta_{k,\Theta_{q}(v)}$}; 
\node at (2.05, -0.2) {$\zeta_{k,\Theta_{q-1}(v)}$}; 
\node at (2.6, -0.2) {$\eta_{k,\Theta_{q-1}(v)}$}; 
\node at (3.1, -0.2) {$\zeta_{\Theta_{0}(v)}$}; 
\node at (3.5, -0.2)  {$\omega_{k,v}$}; 
\foreach \x/\y in {0/0, 0.8/0, 1.1/0,
1.35/0,  1.9/0, 2.5/0, 3.05/0, 3.5/0%
} { \node at (\x,\y) {$\bullet$}; }

\draw[<-] (0,0.6)--(1.5, 0.6); 
\node at (1.7, 0.6)  {\large{$\mathscr I_{k,v}$}} ; 
\draw[->] (2, 0.6)--(3.5, 0.6);
\draw[<-] (1.1, 0.35)--(1.3, 0.35); 
\node at (1.5, 0.35)  {$\mathscr I_{k,\Theta_q(v)}$} ; 
\draw[->] (1.7, 0.35)--(1.9, 0.35);
\draw[thin,dotted] (1.1,0)--(1.1 , 0.3); 
\draw[thin,dotted] (1.9,0)--(1.9, 0.3); 
\draw[thin,dotted] ( 3.05,0)--( 3.05, 0.3); 

\draw[pattern=north west lines] (0,0) rectangle (0.8,0.01);   
\draw[pattern=north west lines] (1.1,0) rectangle (1.35,0.01);  
\draw[pattern=north west lines] (1.9,0) rectangle (2.5,0.01);
\draw[pattern=north west lines] (3.05,0) rectangle (3.4,0.01);
\end{tikzpicture}
}
 \caption{A hint of the partition of a general $\alpha$-cylinder $\mathscr I_{k,v}$. That $\omega_{k, \Theta_q(v)} = \zeta_{k, \Theta_{q-1}(v)}$ holds  for appropriate $v$ is part of Lemma~\ref{l:cylsOfConsecuChildrenAbut}.}\label{f:PartI}
 \end{figure}
 \bigskip

\subsubsection{Palindromes}   The notation $\overleftarrow{v}$ denotes the word formed by taking the letters of $v$ in reverse order.   Thus, $v$ is a palindrome if and only if $v = \overleftarrow{v}$.

\bigskip 
\begin{Prop}\label{p:theWordsAreNice}  Suppose $v \in \mathcal V$.   Then:

\begin{enumerate}
\item   $v$ is a word in at most three letters;
\item    $v$ is a palindrome; 
\item   if $v = \Theta_q(u)$ for some $q\ge 1$, then  $v = u x u$ for some palindrome $x$.  %
\item   if $u$ is the parent of $v$, then there is a palindrome $y$ such that $v' = yu$;
\item  \qquad with this same $y$, one has $v' v'' = y v$;
\item  \qquad  if further $v\neq c_1$, then there are palindromes $a, z$
such that $v'' = az,  y = z' a$.
\end{enumerate}
 \end{Prop} 
\begin{proof}    The first statement naturally has two cases:   if $c_1 =1$ then we claim that all $c_i =1$ and all $d_j$ are contained in $\{d_1+1, d_1\}$;  if $c_1>1$ then all $d_j = 1$ and $c_i \in \{c_1, c_1-1\}$.    We prove this by induction.   Our bases cases are: $\Theta_{-1}(c_1) = c_1+1$ and $\Theta_q(1) = 1q1 $ for $q \ge 1$;  for  $c>1$,  set  $\Theta_q(c) = c [ 1 (c-1)]^q 1c$ for any $q \ge 0$.   The statement clearly holds here.   Thereafter,  $v, v'$ are words in these small alphabets,  and $v''$ is a subword of $v$ hence every $\Theta_q(v)$ has the desired property.

\bigskip 
Note that Statement (3) follows from (2).   Also Statement (4)  implies (5), as $(yu) v'' =  y (u v'') = y v$.
It remains to prove Statements (2),  (4) and (6).

\bigskip

 Beyond easily handled cases of short $v$, there are naturally three cases to consider.

\noindent
{\bf Case 1.}  Suppose   $v = \Theta_{0}^{h}(c)$ for some $h\ge 1$ and some $c> 1$.  Induction gives $v = \Theta_{0}^{h}(c) = c (1\,c)^h$. This is obviously a palindrome, that is (2) holds.   We find
\[v' = 1 (c-1)   (1\,c)^{h} = 1 (c-1) 1 \;\;  c (1\,c)^{h-1} = 1 (c-1) 1 \;  \Theta_{0}^{h-1}(c) = y u,\]
where $y = 1 (c-1) 1$ and $u =  \Theta_{0}^{h-1}(c)$.  But, here $v = \Theta_{0}(u)$  is the parent of $v$, and hence Statement (4) also holds in this case.

  Set $z = c$ and $a = 1$.  Then  $y=z'a$ and $v'' = 1\,c = a z$.  Therefore, (6) also holds in this case. 
\bigskip 
 
\noindent
{\bf Case 2.}  Suppose   $v = \Theta_{q}(u)$ for some $q\ge 1$.

If $u = c_1$,  all of (2), (4) and (6) are easily verified. 

Assume now that (2), (4) and (6) hold for $u$, in the sense that   $u = \mathfrak Z \mathfrak   a \mathfrak   z,  u'' =  \mathfrak   a \mathfrak   z$ with $u, \mathfrak Z,  \mathfrak   a,  \mathfrak   z$ and  $\mathfrak z' \mathfrak a$ all palindromes.   Since $u = \overleftarrow{u}$, we have   $u' = ( \mathfrak z \mathfrak   a \mathfrak  Z)' =  \mathfrak z' \mathfrak   a \mathfrak  Z$.
Therefore, 

 \[ v = \Theta_q(u) = u (u')^q u'' =    u (u')^{q-1} u'   (\mathfrak   a \mathfrak   z) = u (u')^{q-1}  \mathfrak z' \mathfrak   a \mathfrak  Z (\mathfrak   a \mathfrak   z) = u (u')^{q-1}  \mathfrak z' \mathfrak   a  \, u = u ( \mathfrak z' \mathfrak   a \mathfrak  Z)^{q-1}  \mathfrak z' \mathfrak   a  \, u  .  \]

\medskip
\noindent 
Since $\mathfrak z' \mathfrak a$ and $\mathfrak Z$ are palindromes, we find that  $v$ is a palindrome; that is (2) holds.

Set  $y =  ( \mathfrak z' \mathfrak   a \mathfrak  Z)^q  \mathfrak z' \mathfrak   a $.  This is  clearly also a palindrome. Since 
\[v' = ( \mathfrak z' \mathfrak   a \mathfrak  Z)^q  \mathfrak z' \mathfrak   a \, u\,\]
 (4) also holds. 

\bigskip 
Now,  $v = u v''$ gives $v'' = ( \mathfrak z' \mathfrak   a \mathfrak  Z)^{q-1}  \mathfrak z' \mathfrak   a  \, u$.   Let $a = ( \mathfrak z' \mathfrak   a \mathfrak  Z)^{q-1}  \mathfrak z' \mathfrak   a$ and $z = u$.  Then $v'' = az$ and $y = z'a$.  That is, (6) holds.     
 
\bigskip 
 
\noindent
{\bf Case 3.} Suppose that $v = \Theta_0(u)$ and $v \neq \Theta_{0}^{h}(c)$ for any $h\ge 1$ and any $c=c_1$.     Assume that (2), (4) and (6) all hold for $u$ in the sense described in the proof of the previous case.

We find 
\[ v  = u u'' = (\mathfrak Z \mathfrak   a \mathfrak   z) \mathfrak   a \mathfrak   z\] and thus 
\[ \overleftarrow{v} =  \overleftarrow{u''}\, \overleftarrow{u} =  \overleftarrow{\mathfrak   a \mathfrak   z} \,u = \mathfrak   z \mathfrak   a \mathfrak Z \mathfrak   a \mathfrak   z.\]
 Thus (2) holds in this case. 
 
 Now, 
 \[v' = u' u'' = ( \mathfrak z' \mathfrak   a \mathfrak  Z) \mathfrak   a \mathfrak   z = ( \mathfrak z' \mathfrak   a) u.\]
 Thus (4) holds; since $v'' = u''$,  (6) also holds. 
\end{proof} 

\bigskip 
\begin{Rmk}\label{r:looksLikePalindromicSplitting}   We thus have for $v\in \mathcal V$, other than those $v$ of the form $v = c_1$, 
\[\begin{cases} v&= uv'' = uaz\\
                         v' &= yu = z'au.
\end{cases}
\]
In the case of $v = \Theta_q(u)$ with $q \ge 1$,  we have $z=u$.   (Thus in the previous proposition, $x$ is $a$.)
 \end{Rmk} 
\bigskip

The child $\Theta_q(v)$ has length less than twice the length of $v$ only when $q \in \{-1, 0\}$.  The following addresses the setting of $q=0$.
\begin{Lem}\label{l:characterizeThetaZero}  Suppose that $u$ is a child of   $Z \in \mathcal V$.   Then the palindrome $\Theta_0(u)$ is characterized by the property of $u$ being both a  prefix and a suffix, and these subwords  having exactly  the subword $Z$ in common.

 \end{Lem} 
\begin{proof} Let $v = \Theta_0(u)$.  Since $v'' = u''$, we can write  $u = Zaz$ and $v = uaz = Zazaz$ with each of $Z, u, a, z$ palindromes.   Now, $v = (Zaz)az = zaZaz$.    We have $v =   za \underbrace{Zaz}_{u} =\underbrace{zaZ}_{u}az$.  Thus the property does indeed characterize $\Theta_0(u)$.
\end{proof}

\medskip 
\begin{Lem}\label{l:orderVeePrimeDoublePrime}  If $v\in \mathcal V$, then  $v' \prec v''$.   
 \end{Lem} 
\begin{proof}    One directly verifies the result  for any $v= c_1$ and for any $v = \Theta_q(c_1)$.    The remaining possibilities can be divided into  two cases. 
\medskip

\noindent
{\bf Case 1.}  When   $v = \Theta_{q}(u)$ for some $q\ge 1$, we can write $v = uau$.  We find  $v' = (u'a)u = a \overleftarrow{u'}u$, and $v'' = au$.    Writing $u = c_1d_1\cdots c_{s-1}d_{s-1} c_s$,    from Definition~\ref{d:periodFromEtaAsWord} and that fact that $u$ is a palindrome show  that $\overleftarrow{u'}$ and $u$ agree until the relationship is determined by $d_1 c_1 \succ (d_1+1)$ or $d_1 c_1 \succ d_1 (c_1-1)\,1$, when $c_1 =1$ or $c_1 >1$ respectively. 
\medskip
 
\noindent
{\bf Case 2.} Suppose that $v = \Theta_0(u)$.  Then as in Lemma~\ref{l:characterizeThetaZero} we can write $v = uaz = zau$, $v'' = az$.   We have $v' = z'au = a \overleftarrow{z'}u$.   As in the previous case, we find that $v' \prec v''$. 
\end{proof} 

\bigskip 
\subsubsection{Derived words} 
We will often argue by induction on the length of $v$.   These arguments rely on the following map,  $\mathscr D$,  giving the {\em derived word} $\mathscr D(v)$ of $v$.

\begin{Def}\label{d:derivedWords}    Let $v = c_1 d_1 \cdots d_{s-1}c_s$.  
\begin{enumerate} 
\item[(i)]  If $c_1 =c >1$,  and $v$ is such that $d_1 = 1$ and 
the set of $c_i, 1\le s$, is contained in the set of two letters $\{a = c_1, b = c_1 -1\}$, 
express 
\[v = (a \,1)^{e_1} (b \,1)^{f_1}\cdots (b \,1)^{f_{g-1}} (a \,1)^{e_g-1} a.\]

\item[(ii)]  if $c_1 = 1$,   and $v$ is such that 
the set of $d_j$, $1\le j < s$ is contained in the set of two letters $\{a= d_1, b= d_1+1\}$, 
express 
\[v = (1 \,a)^{e_1} (1 \,b)^{f_1}\cdots (1 \,b)^{f_{g-1}} ( 1 \,a)^{e_g}  1.\]
\end{enumerate}
 
 In both cases, let 
\[ \mathscr D(v)  = e_1 f_1 \cdots f_{g-1} e_g\,.\]
Note that (the proof of) part (1) of Proposition~\ref{p:theWordsAreNice} shows that $\mathscr D(v)$ is defined for each $v \in \mathcal V$.   We call the various subwords $(c_1 \,1)^{e_i},  (c_1-1 \;1)^{f_j}, (1 \,d_1)^{e_i}, (1 \,d_1+1)^{f_j}$ {\em full blocks} for $v$.  
\end{Def}
\bigskip
\begin{figure}
 \begin{tikzpicture}[baseline= (a).base]
 \node[scale=.8] (a) at (0,0){
\begin{tikzcd}[column sep=2pc,row sep=2pc]
c \ar[pos=0.5]{d}{\Theta_{0}^{h-1}}&\\
c(1\,c)^{h-1}\ar[pos=0.5]{r}{\mathscr D} \ar[pos=0.5]{d}{\Theta_{0}}&h-1\ar[pos=0.5]{d}{\Theta_{-1}}\\
c(1\,c)^h \ar{r}{\mathscr D}&h
\end{tikzcd}
};
 \node[scale=.8] (a) at (4,0){
\begin{tikzcd}[column sep=2pc,row sep=2pc]
u  \ar{d} &\\
v\ar{r}\ar[pos=0.5]{d}{\Theta_{q\ge 1}}& \mathscr D(v)\ar[pos=0.5]{d}{\Theta_{q\ge 1}}\\
v(v')^q v'' \ar{r}{\mathscr D}& \Theta_{q}(\,\mathscr D(v)\,)
\end{tikzcd}
}; 
\node[scale=.8] (a) at (8,0){
\begin{tikzcd}[column sep=2pc,row sep=2pc]
u\neq c_1 \ar{d} &\\
v\ar{r}\ar[pos=0.5]{d}{\Theta_{0}}& \mathscr D(v)\ar[pos=0.5]{d}{\Theta_{0}}\\
vv'' \ar{r}{\mathscr D}& \Theta_{0}(\,\mathscr D(v)\,)
\end{tikzcd}
}; 
\end{tikzpicture}
\caption{Taking derived words,  $v \mapsto \mathscr D(v)$, respects parent-child relations. See  Lemma~\ref{l:dRespectsRelationships}.}
\label{f:commRelationsForScriptD3}
\end{figure}

\begin{Lem}\label{l:dRespectsRelationships}   The map $\mathscr D$ sends $\mathcal V$ to itself, preserving the parent-child relationship.  That is, if $u \in \mathcal V$ and $v = \Theta_q(u)$, then there is a $q'$ such that $\mathscr D(v) = \Theta_{q'}(\, \mathscr D(u)\,)$.   Moreover, $q'=q$ unless $v = \Theta_{0}^{h}(c)$ for some $h\ge 1$.
\end{Lem} 
\begin{proof}  Just as in the proof of Proposition~\ref{p:theWordsAreNice}, there are easily verified base cases which we leave to the reader.  We treat three main cases, see Figure~\ref{f:commRelationsForScriptD3}.

\bigskip

\noindent
{\bf Case 1.}  Suppose   $v = \Theta_{0}^{h}(c)$ for some $h\ge 1$ and some $c> 1$.  We have $v =c (1\,c)^h$.   Certainly,  $\mathscr D(v) = h+1$, and $\mathscr D(\,  \Theta_{0}^{h-1}(c)\,) = h$.   We note that $\Theta_{-1}(h) = h+1$, and of course that $v = \Theta_{0}(\,\Theta_{0}^{h-1}(c)\,)$.  That is, the result holds in this case.

\medskip

\noindent
{\bf Case 2.}  Suppose $v = \Theta_q(u)$, with  $q\ge 1$ for some  $u\in \mathcal V$.    Although this part of the proof is fairly straightforward, it has perforce a panoply of variables representing words; the reader may wish to consult \eqref{e:DandTheta}, below, as a guide.
 
 We can write $\mathscr D(v )  = \mathscr D(\, u (u')^q u''\, ) = \mathscr D(\, u (u')^{q-1}y\, u \, )$.  Calculation, simply using the definition of $\mathscr D$,  shows that $\mathscr D(\, u (u')^q u''\, )$ has prefix $\mathscr D(u)$.   The definition also yields that the derived word of any palindrome is also a palindrome, thus here $\mathscr D(v )$ also has suffix $\mathscr D(u)$.    Direct calculation shows that for any words $u, x$ and any $p\ge1$,   $\mathscr D( u (u')^p x) = \mathscr D( u)  \, \{\,[\mathscr D(u)]'\,\}^p  X$ for some $X$.  By Proposition~\ref{p:theWordsAreNice}, there is some palindrome $Y$ such that $[\mathscr D(u)]' = Y \mathscr D(Z)$ where $\mathscr D(Z)$ is the parent of   $\mathscr D(v)$.   Therefore, 
  $\mathscr D(v)   = \mathscr D(\, u (u')^{q-1}y\, u \, )  =  \mathscr D( u)  \, [Y \mathscr D(Z)]^{q-1} X   \mathscr D(u)$ for some (new) $X$.    Since  $\mathscr D(v)$ is a palindrome,  we have $X = Y$.    Again by Proposition~\ref{p:theWordsAreNice}, $Y \mathscr D(u) =  Y \mathscr D(Z) [\mathscr D(u)]'' =  [\mathscr D(u)]'\,  [\mathscr D(u)]'' $.  Therefore,    $\mathscr D(v )  = \mathscr D( u)  \, \{\,[\mathscr D(u)]'\,\}^{q-1}     \; [\mathscr D(u)]'\,  [\mathscr D(u)]''  =   \Theta_q(\,  \mathscr D(u)\,)$. 
 
As a summary, we have 
\begin{equation}\label{e:DandTheta}
\begin{aligned}
\mathscr D( u (u')^q u'') &= \mathscr D( u (u')^{q-1} u' u'') \\
                                       &= \mathscr D( u (u')^{q-1} y u) \\
                                       &=  \mathscr D( u)  \, \{[ \mathscr D(u)]'\}^{q-1} X   \mathscr D(u)\\
                                       &=  \mathscr D( u)  \, [Y \mathscr D(Z)]^{q-1} X   \mathscr D(u)\\
                                       &=  \mathscr D( u)  \, [Y \mathscr D(Z)]^{q-1} Y   \mathscr D(u)\\
                                       &=  \mathscr D( u)  \, \{[ \mathscr D(u)]'\}^{q-1} Y \mathscr D(Z) [\mathscr D(u)]''\\
                                       &=  \mathscr D( u)  \, \{[ \mathscr D(u)]'\}^{q}  [\mathscr D(u)]''\\
                                      &= \Theta_q(\, \mathscr D(u)\,).    
\end{aligned}
\end{equation}

\medskip 
 
\noindent
{\bf Case 3.} Suppose that $v = \Theta_0(u)$ and $v \neq \Theta_{0}^{h}(c)$ for any $h\ge 1$ and any $c=c_1$. 
Lemma~\ref{l:characterizeThetaZero} and the definition of $\mathscr D$ yield the result in this case.
\end{proof}

\bigskip
\subsubsection{Fullness of branches}\label{sss:fullBranched}   We aim to describe symbolically $T_{3,n,\alpha}$-orbits,  and in particular to determine intervals in the parameter $\alpha$ where initial segments of such orbits share common digits.  For any word determining  sequences of digits, we must determine the endpoints of the parameter interval along which the word does describe {\em admissible} sequences of digits, see Figure~\ref{smallAlpNonFullBranch}.   The following notion is key to this.

\begin{Def}\label{d:fullBranchedAsWord}       Let $u$ a  word with alternating letters $c_i, d_j$.      (We allow prefixes of words $v$ including those that end with some $d_j$.) 

\begin{enumerate} 
\item  If $u$ begins with say $c_1$ and ends in some $d_j$, then powers of $u$ are again alternating words in the letters $c_i, d_j$.  In the case that $u = c_1 \cdots c_j$ begins with $c_1$ and ends in some $c_j$, then we define $u^2 = c_1 \cdots (c_j +c_1)\cdots c_j$ and similarly for higher powers.

\item We say that $u$ is {\em full branched} if for any prefix $u_{[1, \ell]}$ of $u$, the inequality $u^{\infty} \preceq u_{[1, \ell]}^{\infty}$ holds.   
We denote the longest prefix of $u$ that is full branched by $\mathfrak f (u)$.    

Note that an equivalent definition is:      $\mathfrak f (u)$ is the longest prefix of $u$ satisfying $(\,\mathfrak f (u)\,)^{\infty} = \min\{ (\, u_{[1,\ell]}\,)^{\infty}\,:\,u_{[1,\ell]} \, \text{is a prefix of}\; u \}$.  

\item We  define   $\omega_{k,u}$ as the $\alpha$-value such that $R_{k,\mathfrak f (u)}\cdot r_0(\omega_{k,u}) =r_0(\omega_{k,u})$.       That is,  $r_0(\omega_{k,u})$ has the ($T_{\alpha}$-inadmissible) simplified digit expansion 
$(\,\overline{d}(k, \mathfrak f(u)\,)\,)^{\infty}$.
\end{enumerate}
 \end{Def}

\bigskip

 \begin{Eg}\label{e:computingFrakF}

 \begin{enumerate} 
 \item  Recall that the leftmost red branches in Figure~\ref{firstBothDigs} are branches of $r_1(\alpha)$ as a function of $r_0(\alpha)$, with $\alpha < \gamma_{3,3}$.  These branches agree with portions of the graphs $y = A^kC\cdot x$ and $x = r_0(\alpha)$.  For each $k$, the corresponding branch intersects with $y=x-t$ on the left, and  $y = x$ on right.  The leftmost of these two points is $x = r_0(\zeta_{k,1})$.    Certainly for every $\alpha$ between such a set of intersection points, we have $r_0(\alpha) = k, \cdots$.     On the other hand, by definition $\mathfrak f(1) = 1$,  and thus   $r_0(\omega_{k,1})$ is the fixed point of $A^kC\cdot x$.  That is,   $r_0(\omega_{k,1})$ is our right intersection point.    Furthermore,  for all $\alpha \in [\zeta_{k,1}, \omega_{k,1})$ we have $y = r_1(\alpha)$ is given by $r_1(\alpha) = A^k C\cdot r_0(\alpha)$.   That is, the first simplified digit of $r_0(\alpha)$ is $k$.  In other words,  $\mathscr I_{k,1}= [\zeta_{k,1}, \omega_{k,1})$. 
 
Being a fixed point,  $r_0(\omega_{k,1}) = k, k, \ldots\,$, is purely periodic with period $k$.   Note however that this is {\em not} a $T_{\alpha}$-admissible expansion, as were it so then $r_1$ for this value of $\alpha$ would equal $r_0$.   But,  $r_0(\alpha) \notin \mathbb I_{\alpha}$!

 \item  Ê Now suppose $c>1$.  By definition,  $\mathfrak f(c)= c$, that is  $v=c_1$ is full branched.    Therefore,   for each $k$,  $\omega_{k,c}$ is such that  $r_0(\omega_{k,c})$ is the fixed point of $(A^kC)^c$.  Of course,  $(A^kC)^c$ has the same fixed point as $A^kC$.     That is,    $\omega_{k,c} = \omega_{k,1}$.     Related to this,  there are values of $\alpha$ sufficiently close to $\omega_{k,1}$ so that for each of these $\alpha$,   the first $c$ simplified digits of $r_0(\alpha)$ are all equal to $k$.     Equivalently,  $r_0(\alpha), r_1(\alpha), \dots, r_c(\alpha)$ are all in $\Delta_{\alpha}(k,1)$.  And, this is also to say that  $\overline{d}(k,c)$ gives the first $c$ simplified digits of $r_0(\alpha)$.    The reader should easily find $\alpha$ in the complement of $\mathscr I_{k,c}$ inside of $\mathscr I_{k,1}$.
 
 \item  Consider $v = 111$.    We consider each prefix in turn.   Of course $u = 1$ is full-branched.   We next compare  $(11)^{\infty} = 11\cdot 11\cdot 11 \cdot \,\cdots$ with  $1^{\infty}$; by our convention for powers, we certainly find that $(11)^{\infty} \prec 1^{\infty}$.    We next compare $(11)^{\infty}$ with $(111)^{\infty} = 111\cdot 111\cdot 111 \cdot \,\cdots = 121212\cdots$; certainly  $(11)^{\infty}\prec (111)^{\infty}$.   Therefore,  $\mathfrak f(111) = 11$.   Compare this with Figure~\ref{smallAlpNonFullBranch}.

 \item  Consider $v = \Theta_1(313) = 313 (1213)13$.    Arguing just as for the previous case, we find that $(31)^{\infty}$ is the minimal element of $\{3^{\infty}, (31)^{\infty}, (313)^{\infty}\}$.   Since $3131 = (31)^2$,  we certainly have that $(31)^{\infty} = (3131)^{\infty}$, and thus this latter is our current candidate for the maximal length  full branched prefix of $v$.  We compare it with $(313 12)^{\infty}$. We find
 that these infinite words agree in their first four letters, but in the fifth (a ``$c_j$''-position) they differ. Confer Figure~\ref{f:calculationsOfFrakF}.
Since  $31313\prec 31315$, we find that $(313 1)^{\infty} \prec (31315)^{\infty}$.   We thus now compare $(313 1)^{\infty}$ with $(313 121)^{\infty}$.  Confer Figure~\ref{f:calculationsOfFrakF}.
we find that  $(313 121)^{\infty}$ is the smaller.  One easily sees that is it also smaller than $(313 1213)^{\infty}$.  We now compare it with $(313 12131)^{\infty}$.   See  Figure~\ref{f:calculationsOfFrakF}.
Again,    $(313 121)^{\infty}$ is the smaller.  One easily sees that also  $(313 121)^{\infty} \prec v^{\infty}$.   Therefore,  $\mathfrak f(v) = (313 121)^{\infty}$.
 \end{enumerate}

\bigskip
\begin{figure}
\begin{tabular}{lcr}
\scalebox{.7}{
\begin{tabular}{cccc|c|cccccc} 
 3&1&$\cdot$ 3&1&$\cdot$ 3&1&$\cdot$ 3&1&$\cdots$\\
  3&1&3&1&2+3&1&3&1&3&
 \end{tabular} 
}
&
\scalebox{.7}{
\begin{tabular}{cccc|c|cccccc} 
 3&1&3&1&$\cdot$ 3&1&3&1&$\cdots$\\
 3&1&3&1&2&1&$\cdot$3&1&3&
 \end{tabular}  
}
&
\scalebox{.7}{
\begin{tabular}{cccccccccc|c|ccccc} 
 3&1&3&1&2&1&$\cdot$3&1&3&1&2&1&\\
 3&1&3&1&2&1&3&1&$\cdot$ 3&1&3&$\cdots$
 \end{tabular}
}
\end{tabular}
\caption{Naive calculations for $\mathfrak f (\, \Theta_1(313)\,)$. From left to right, finding:  $(3 1)^{\infty}\prec (31312)^{\infty}$; $(3 1 3 1)^{\infty}\succ (313121)^{\infty}$; $(313121)^{\infty}  \prec (31312131)^{\infty}$.   The beginning of second copies of words are marked with a dot.}
\label{f:calculationsOfFrakF}
\end{figure}

\end{Eg}

The following shows that certain phenomena illustrated in the above examples hold in general.
\begin{Lem}\label{l:frakFatDjOnly}   If $v\in \mathcal V$ is of length greater than one, then $\mathfrak f(v)$ has even length.  In this case,  for each $k\in \mathbb N$,   a simplified digit expansion  for  $r_0(\omega_{k,v})$  is $\overline{d}(k, (\,\mathfrak f(v)\,)^{\infty}\,)$.    
\end{Lem}
\begin{proof} Our convention for powers of words shows that any prefix $u$ of odd length greater than one, thus having initial letter $c_1$  and final letter some $c_i$,   has its second power including the letter $c_j+c_1$.  Already $u^2$ is larger than the prefix $(c_1d_1)^{\infty}$.  The first statement thus holds.

 For any word $u = c_1 d_1 \cdots d_{j}$,  we have  
 \[\overline{d}(k, u^{\infty}\,) = [k^{c_1}, (k+1)^{d_1}, \cdots, (k+1)^{d_{j}}]^{\infty} =  \overline{d}(k, u\,)^{\infty}.\]
Thus, since $\omega_{k, v}$ is defined to be $(\,\overline{d}(k, \mathfrak f(u)\,)\,)^{\infty}$, the result holds. 
\end{proof} 
 
 \bigskip

\begin{figure}[h]
\noindent
\begin{tabular}{cc}
\scalebox{.3}{
{\includegraphics{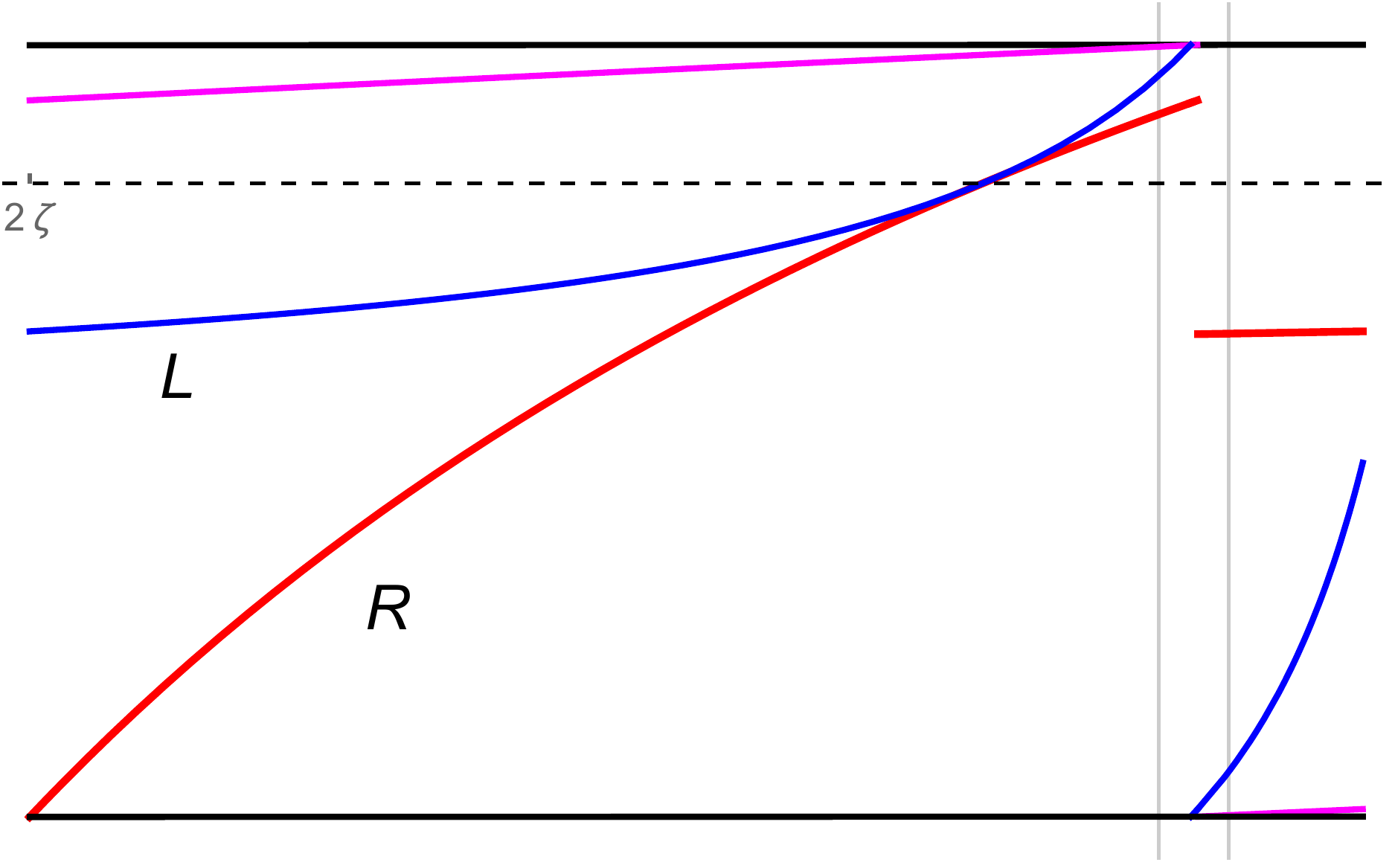}}}\quad\quad
\scalebox{.3}{
{\includegraphics{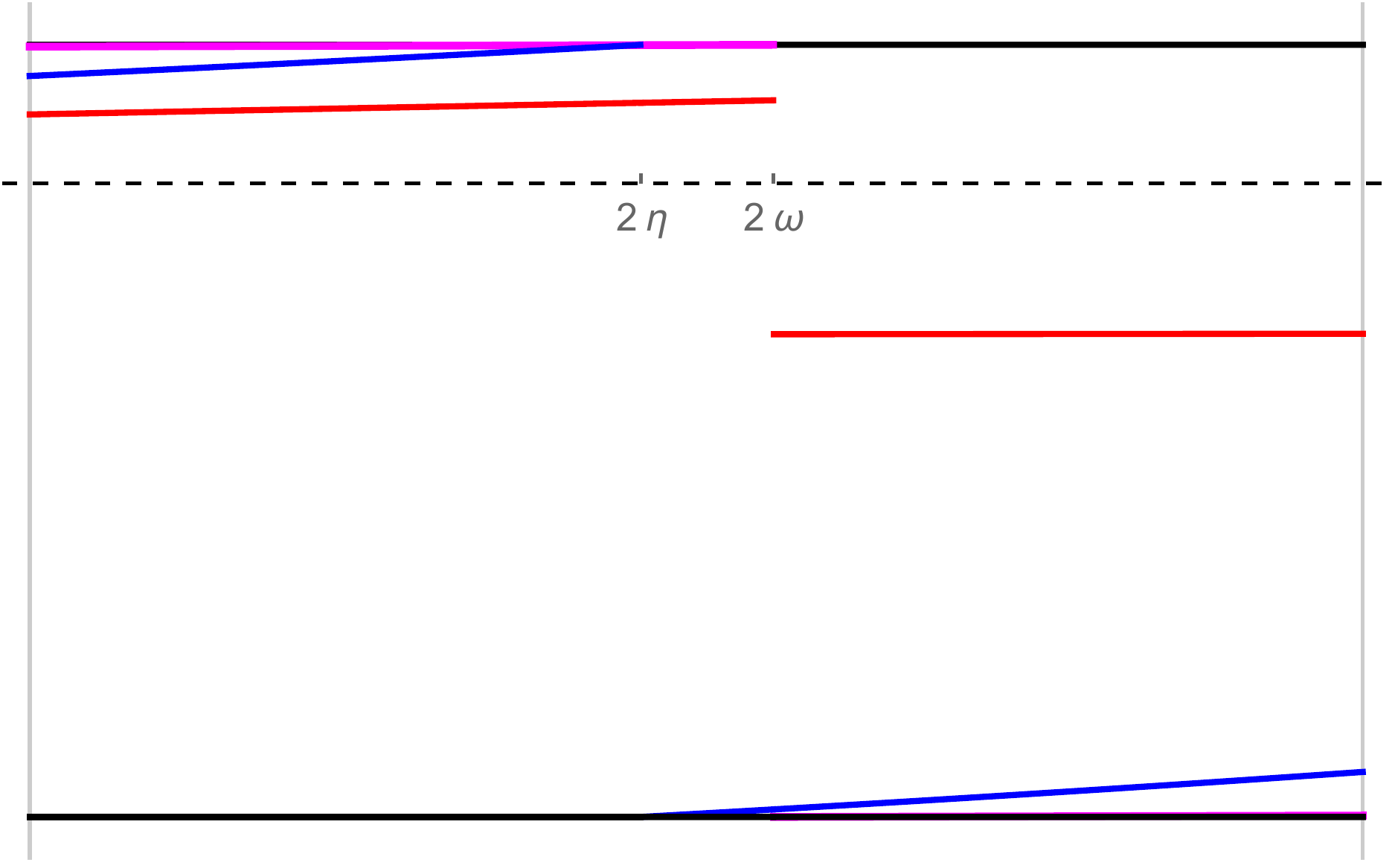}}}
\end{tabular}
\caption{A non-full branch.  Here $n=m=3$, $v= 111$ and $k=1$; we have that  $\omega_{1, 111}$ is determined by the fixed point of $R_{1, 11}$.    The labels $L, R$ mark respectively the curves $y = L_{1,111}\cdot r_0(\alpha), y= R_{1, 111}\cdot r_0(\alpha)$ where $\alpha = x/2 = x/t_{3,3}$.  Red gives branches of $y = r_3(\alpha)$, while blue colors the two  branches of  $y = \ell_9(\alpha)$;  Magenta gives the branches of   $y = r_2(\alpha)$.     The left portion has $2 (9 - \sqrt{15})/7 < x <  (11- 6 \sqrt{2})/7$. The right ``zooms in" to $0.35910 < x < 0.35915$. (This interval lies between the vertical gray lines in both portions.)  The $x$-axis is shown as a dotted line. Black curves give $y = x$ and $y = x-2$.   (This example is in fact the longest interval $\mathscr I_{k,v}$ where $R_{k,v}$ is not full for any $n\ge 3, k\in \mathbb N, v \in \mathcal V$.)}
\label{smallAlpNonFullBranch}
\end{figure}
 
\bigskip
The next result indicates the utility of the notion of full branchedness.
\begin{Lem}\label{l:potentialCylinders}   Let $v \in \mathcal V$ and fix $k\in \mathbb N$.   The $\alpha$-cylinder set $\mathscr I_{k,v}$ is a subset of   $[\zeta_{k,v}, \omega_{k,v})$. 
\end{Lem}
\begin{proof}   The set of $\alpha$ such that $\overline{d}{}^{\alpha}_{[1, \overline{S}(k,v)]} = \overline{d}(k,v)$ is contained in the interval $[\zeta, \omega)$ such that $R_{k,v}\cdot r_0(\zeta) = \ell_0(\zeta)$ and  $R_{k,v}\cdot r_0(\omega) = r_0(\omega)$.   (Note that this is implied by the connected nature of each of the $\Delta(k,l)$, confer Figure~\ref{allDigs}.) 
The left endpoint here is exactly $\zeta= \zeta_{k,v}$.      

Now,  the definition of the $\alpha$-cylinder $\mathscr I_{k,v}$ as the set of those $\alpha$ such that the digit sequence determined by $k$ and $v$ are $\alpha$-admissible  implies that $\mathscr I_{k,v}$  is contained in the intersection of the corresponding $\alpha$-cylinders for $k$ and the prefixes $u$ of $v$.    In particular,  the least right endpoint of these cylinders gives an upper bound of the right endpoint of $\mathscr I_{k,v}$.  But, each of these cylinders has its right endpoint bounded above by its own corresponding fixed point,  $\overline{d}(k,  u^\infty)$.   Hence,  we find that the right endpoint of $\mathscr I_{k,v}$ is less than or equal to the least of these $\overline{d}(k,  u^\infty)$.  Since this least point is  $ \overline{d}(k, (\, \mathfrak f (v)\,)^\infty)$, we are done. 
 \end{proof} 
 
\begin{Lem}\label{l:startsAtOmegaOneOne}   Fix $m=3$, and $n \in \mathbb N$.   Then 
$\gamma_{3,n} = \omega_{1,1}$. 
\end{Lem}
\begin{proof} Let $\gamma = \gamma_{3,n}$  By definition,  $C^{-1} \cdot \ell_0(\gamma)= r_0(\gamma) $.  Applying $AC$ to both sides of this equality yields $r_0(\gamma) = AC\cdot r_0(\gamma)$.   Since $\mathfrak f(1) = 1$,  we have that $r_0(\omega_{1,1}) = AC\cdot r_0(\omega_{1,1})$ and thus $\gamma_{3,n} = \omega_{1,1}$. 
\end{proof}

 \medskip 
 Recall that the full blocks of $v$ are defined in Definition~\ref{d:derivedWords}. 
\begin{Lem}\label{l:frakFandD}  Suppose $v \in \mathcal V$ is of length greater than one.  Then $\mathfrak f(v)$ ends with a full block of $v$. 
\end{Lem}
\begin{proof}  When $c_1 = 1$, we have  $1 \,d_1 \succ 1 \,(d_1+1)$.  When $c_1 > 1$, we have $c_1 \,1 \succ (c_1-1) \, 1$.    

We treat the case of $c_1 = 1$, the other case being similar.   First, if $\mathscr D(v) = e_1$ then one easily verifies that $\mathfrak f(v) = (1 \,d_1)^{e_1}$.  Otherwise, the argument of Lemma~\ref{l:frakFatDjOnly} showing that  $v$ of length greater than one have $\mathfrak f(v)$ of even length, gives here that $\mathfrak f(v)$ ends with some power of $[1 \,(d_1+1)]$.      Suppose we have a prefix of $v$ which ends with a non-full block of this type, say  $w = (1 \,d_1)^{e_1} [1 \, (d_1+1)]^{f_1}\cdots [1 \, (d_1+1)]^{f_i-j}$.   Then the square of $w$ has the intermediate term $[1 (d_1+1)]^{f_i-j}(1\, d_1)^{e_1}$, whereas the prefix that completes $w$ to the end of the block $[1 \,(d_1+1)]^{f_i}$ agrees with $w^2$ up to a replacement of $(1\, d_1)$ by $[1 \,(d_1+1)]$.  That is, this new prefix is smaller than $w^2$, and thus certainly its infinite power is smaller than $w^2$.   The result thus holds in this case.   
\end{proof} 
 
\bigskip
\begin{Prop}\label{p:frakFofChildren} 
 Suppose $v \in \mathcal V$.  
Then  
 
\[ \mathfrak f(\, \Theta_q(v)\,) = \begin{cases} (c\,1)^h&\text{if}\; q=0, v = \Theta_{0}^{h}(c),\\
                                                                         v (v')^{q-1}y&\text{if}\; q\ge 1,  \, \text{and}\; y\; \text{as in Proposition~\ref{p:theWordsAreNice}},\\
                                                                         (ua)^{h+1}&\text{if}\; q=0,  v =  \Theta_{0}^{h}\circ\Theta_{p}(u), p\ge 1, \,\text{and}\;\Theta_{p}(u)=uau.
                                               \end{cases}
\]

\medskip
In particular,   for all $v \in \mathcal V$,   the word $v$ is a prefix of $(\, \mathfrak f(v)\,)^2$.
\end{Prop}
\begin{proof} Direct evaluation, as in Example~\ref{e:computingFrakF} shows that $\mathfrak f(c\,1\,c) = c\, 1$.   Now suppose the result holds for some $h-1\ge 1$;  the only remaining (even length) candidate prefixes that could be $\mathfrak f(v)$ are 
$ (c\,1)^{h-1}$ and $(c\,1 )^h$. Since these words have the same infinite powers,  by definition the longer of these, that is $(c\,1)^h$, is $\mathfrak f(v) $.

\medskip 
For $q\ge 1$,  base cases can be directly verified.  We now use induction on the length of $v$, and thus assume   $\mathfrak f(\, \Theta_q(\,\mathscr D(v)\,)\,) = \mathscr D( u)  \, \{[ \mathscr D(u)]'\}^{q-1} Y$, with $Y$ as in \eqref{e:DandTheta}.  From  \eqref{e:DandTheta}, we then have 
\[\mathscr D( \,  v (v')^{q-1} yv\,) = \mathfrak f(\, \Theta_q(\,\mathscr D(v)\,)\,)\; \mathscr D( u).\]   Since the blocks of $v$ of exponent $e_i$ are larger than the blocks of exponent $f_i$,  one finds that $\mathfrak f(\, \Theta_q(v)\,)$ can be no longer than $v (v')^{q-1}y$.    Since $\mathfrak f(\, \Theta_q(v)\,)$ ends with a full block, $\mathfrak f(\, \Theta_q(v)\,)$ can also be no shorter than $v (v')^{q-1}y$.   Thus,  the result holds.

We can indeed assume that $\Theta_p(u) = uau$, see Remark~\ref{r:looksLikePalindromicSplitting}.  From the previous case, $\mathfrak f(\,\Theta_p(u)\,) = ua$.   One easily finds that $\Theta_{0}^{h}\circ\Theta_{p}(u) = u (au)^{h+1}$.   Thus, we seek to prove that $\mathfrak f(\,u (au)^{h+1}\,)$ is formed by dropping the suffix $u$.   Here also, we can apply $\mathscr D$, as the verification for bases cases is straightforward.

That $(\, \mathfrak f(v)\,)^2$ has prefix $v$ is easily checked in each case.
\end{proof}

 \bigskip

The following result could well be placed earlier, but is not used until directly hereafter.  
\begin{Lem}\label{l:zetaRightDigs}  Let $k \in \mathbb N$ and $v = c_1d_1\cdots c_{s-1}d_{s-1} c_s$.    If  $\zeta_{k,v}  \in \mathscr I_{k,v}$ then 
$\overline{d}{}^{\zeta_{k,v} }_{[1,\infty)}  = \overline{d}(k,v), \,   \underline{d}{}^{\zeta_{k,v}}_{[1,\infty)}$.  
Furthermore,  if $v = \overleftarrow{v}$, then 
$r_0(\zeta_{k,v})$ is the fixed point of $R_{k, \overleftarrow{v'}}$.    
\end{Lem} 
\begin{proof}  
By definition,  $\zeta_{k,v}$ is such that $R_{k,v}\cdot r_0(\zeta_{k,v}) = \ell_0(\zeta_{k,v})$ and hence $ r_0(\zeta_{k,v})$ is the fixed point of $A R_{k,v} = A^{k+1}  C(A^kC)^{ c_s-1}\; (A^{k+1}C)^{d_{s-1}}(A^kC)^{c_{s-1}}\cdots (A^{k+1}C)^{d_1} (A^kC)^{c_1}$.  When $v$ is a palindrome, that is when $v = \overleftarrow{v}$, this matrix is indeed $R_{k, \overleftarrow{v'}}$.  (Note that when $v=1$ this matrix must be interpreted as $A^{k+1}C$.) The definition of  $\zeta_{k,v}$ also shows that the sequence of upper simplified digits of $\zeta_{k,v}$ is formed by $ \overline{d}(k,v)$ followed by the digits of $\ell_0(\zeta_{k,v})$.  That is,   $\overline{d}{}^{\zeta_{k,v} }_{[1,\infty)}  = \overline{d}(k,v), \,   \underline{d}{}^{\zeta_{k,v}}_{[1,\infty)}$.
\end{proof}

\bigskip
 
\begin{Cor}\label{c:vD(v)InConeGotEndpt} Fix $k \in \mathbb N$.  Suppose that $v  = \Theta_p(u)$ for some $u \in  \mathcal V$ and some $p\ge 1$.       Then  for $h \ge 0$,  we have $\omega_{k, \Theta_{0}^{h}(v)} = \omega_{k,v}$.  Furthermore, 
$\lim_{h \to \infty}\, \zeta_{k, \Theta_{0}^{h}(v)} = \omega_{k,v}$.
\end{Cor}
\begin{proof}    Proposition~\ref{p:frakFofChildren} implies that $(\, \mathfrak f(\,  \Theta_{0}^{h}(v)\,)\,)^{\infty} =  (ua)^{\infty}=  (\,\mathfrak f(v)\,)^{\infty}$.   Lemma~\ref{l:frakFatDjOnly} now implies that $\omega_{k, \Theta_{0}^{h}(v)} = \omega_{k,v}$.    For each $h$, Lemma~\ref{l:zetaRightDigs} shows that   $\zeta_{k, \Theta_{0}^{h}(v)}$ has prefix $\overline{d}(k, \Theta_{0}^{h}(v)\,) = \overline{d}(k, u(au)^{h+1}\,)$. Therefore, the second statement holds as well.
\end{proof} 

\bigskip
\begin{Lem}\label{l:cylsOfConsecuChildrenAbut}   Suppose that $v  \in  \mathcal V$.  Then for  $q\in \mathbb N$ (except when $(v,q) = (1,1)\,)\;$), 
\[\mathfrak f(\Theta_q(v)) = \overleftarrow{(\Theta_{q-1}(v)\,)'}.\]

Furthermore, for each $k \in \mathbb N$, $\omega_{k, \Theta_q(v)} = \zeta_{k, \Theta_{q-1}(v)}$.
\end{Lem}
\begin{proof} Proposition~\ref{p:frakFofChildren} gives   $\mathfrak f(\Theta_q(v)) = v (v')^{q-1} y$, with notation as there.   Now $(\,\Theta_{q-1}(v)\,)' =  v'(v')^{q-1}v'' =   (v')^{q-1} yv$.  Since both $\Theta_{q-1}(v)$ and $y$ are palindromes,  $\overleftarrow{(\Theta_{q-1}(v)\,)'} = v (v')^{q-1} y$.   Therefore,  $ \overleftarrow{(\Theta_{q-1}(v)\,)'} = \mathfrak f(\Theta_q(v))$.   

 By definition,   $r_0(\omega_{k, \Theta_q(v)} )$ is the fixed point of $R_{k, \mathfrak f (\Theta_q(v))}$.   
 By Lemma~\ref{l:zetaRightDigs},  $r_0( \zeta_{k, \Theta_{q-1}(v)})$ is fixed by $R_{k, \overleftarrow{w}}$, where $w =(\Theta_{q-1}(v)\,)'$.     Therefore,   $\zeta_{k, \Theta_{q-1}(v)} = \omega_{k, \Theta_q(v)}$.
\end{proof}

\bigskip
The following illustrates how $\Theta_1(1)$ plays a role  similar to the $\Theta_0(c), c>1$. 
\begin{Lem}\label{l:consecSpecialCase}   For each $k,c \in \mathbb N$, 
\[ \omega_{k, c\,1\,c} = \zeta_{k, c+1}\,.\]
\end{Lem}
\begin{proof} Since $\mathfrak f(c\,1\,c) = c\,1$,  this follows since $A^{k+1}C (\, A^kC)^c\cdot r_0(\alpha) = r_0(\alpha)$ is equivalent to  $(\, A^kC)^{c+1}\cdot r_0(\alpha) = \ell_0(\alpha)$.
\end{proof}

\bigskip
\begin{Lem}\label{l:etaLessThanOmega}   Suppose that $v  \in  \mathcal V$.  Then 
$v\,(v')^{\infty} \prec (\, \mathfrak f(v)\,)^{\infty}$.
\end{Lem}
\begin{proof}  The result is immediate for $v=c_1$.  We treat our usual remaining cases. 
\bigskip

\noindent
{\bf Case 1.}  Suppose   $v = \Theta_{0}^{h}(c)$ for some $h\ge 1$ and some $c> 1$.  We have $v =c (1\,c)^h$, and $\mathfrak f(v) = (c\, 1)^h$.   
Since $v\,(v')^{\infty} = c (1\,c)^h\, [\, 1\, (c-1)\, 1\, 1]^{\infty}\prec (c\, 1)^{h+2}$, the result holds in this case. 

\bigskip
\noindent
{\bf Case 2.} $v = \Theta_p(u), p\ge 1$.    Write $v = uau$ in our usual decomposition.   Then $v'' = au$ and $\mathfrak f(v)= ua$.  We find
\[(\,\mathfrak f(v)\,)^3 = uau\, au\, a = v v'' a \succ v v',\]
thus certainly   $(\,\mathfrak f(v)\,)^{\infty} \succ v\,(v')^{\infty}$.

\bigskip
\noindent
{\bf Case 3.}    Finally,  suppose $v = \Theta_{0}^{h}(uau)$.  Thus,  $v = u(au)^{h+1}, v'' = au$ and $\mathfrak f(v) = (ua)^{h+1}$.   Hence,  $(\,\mathfrak f(v) \,)^2  
=  u(au)^{h+1}\cdot au\cdot a (ua)^{h-1}= v v''  a (ua)^{h-1}$.  Since $v''\succ v'$,  it follows that   $(\,\mathfrak f(v) \,)^2 \succ v v'$, and thus  $(\,\mathfrak f(v)\,)^{\infty} \succ v\,(v')^{\infty}$.
\end{proof}

\bigskip

\subsection{The complement of the potential synchronization intervals is a Cantor set}
\bigskip 
\begin{Prop}\label{p:Admissible}   For all $v \in \mathcal V$, and all $k \in \mathbb N$, both
\[ \mathscr I_{k,v} = [\zeta_{k,v}, \omega_{k,v})\;\; \text{and} \;\;  \mathscr I_{k,v} \supset \mathscr J_{k,v}\,.\]
\end{Prop} 
\begin{proof} 

Fix $k \in \mathbb N$.   We argue by induction on the length of the word $v$.     Recall that by Lemma~\ref{l:etaRightDigs}, $\eta_{k,v} \in \mathscr I_{k,v}$ implies that $\overline{d}(k, v (v')^\infty)$ gives the sequence of simplified digits of $r_0(\eta_{k,v})$.   

The base cases, given by $v = c_1$, are easily verified.   
  
\medskip

\noindent
{\bf Case 1.}  Suppose   $v = \Theta_{0}^{h}(c)$ for some $h\ge 1$ and some $c> 1$.  We have $v =c (1\,c)^h$. 
We begin by assuming the result for both words $c, c+1$.   We then have $\eta_{k,c}< \zeta_{k,c+1}<\omega_{k,c+1} = \omega_{k,c}$.  Therefore, there must be proper extensions of the word $c_1 = c$ that are admissible on $[\eta_{k,c}, \zeta_{k,c+1})$.  By Lemma~\ref{l:etaRightDigs},   $\overline{d}{}^{\eta_{k,c} }_{[1,\infty)} =  k^c,\, \overline{ (k+1), k^{c-1}}$;  that is, $c\, 1$ is admissible at this value of $\alpha$.    Now, $(A^kC)^{c+1} \cdot r_0(\alpha) = \ell_0(\alpha)$ determines the value $\alpha = \zeta_{k,c+1}$.   Equivalently,  $A^{k+1}C (A^kC)^c \cdot r_0(\alpha) = r_0(\alpha)$.     

From this last, we conclude that  $\mathfrak f(c\,1\,c) = c\,1$ is admissible throughout $[\eta_{k,c}, \zeta_{k,c+1})$, with arbitrarily high powers of $c\,1$ admissible for $\alpha$ sufficiently close but smaller than $\zeta_{k,c+1}$.  (Recall that  $\omega_{k, c\,1\,c}= \zeta_{k,c+1}$.)    Since $\Theta_0(c) = c\,1\,c$ is a prefix of   $(c\,1)^2$,  we find that there is an interval with right endpoint $\omega_{k, c\,1\,c}$ on which the word $\Theta_0(c)$ is admissible.   The left endpoint of this interval is characterized as the leftmost $\alpha$ such that the third letter is remains admissible; it must hence be $\zeta_{k, c\,1\,c}$.  Thus,  our result holds when $h=1$.  We now partition the interval $[\eta_{k,c}, \zeta_{k,c+1})$ according to the highest power of $c\,1$ that is admissible for $\alpha$, from which the result follows.  

\medskip
 
\noindent
{\bf Case 2.} $v = \Theta_p(u), p\ge 1$.    Write $v = uau$ in our usual decomposition.   By hypothesis,  $\overline{d}(k,v)$ gives a prefix of  $\overline{d}{}^{\alpha }_{[1,\infty)}$ for all $\alpha \in [\zeta_{k,v}, \omega_{k,v})$.        

We have $\mathfrak f(v) = ua$, and the admissibility of $v = uau$ implies the admissibility of $ua$ on $[\zeta_{k,v}, \omega_{k,v})$.    Thus for each  $\alpha \in [\zeta_{k,v}, \omega_{k,v})$, there exists a maximal $N=N(\alpha) \in \mathbb N$ such that $u(au)^N$ gives a prefix of $\overline{d}{}^{\alpha }_{[1,\infty)}$.
Since $\omega_{k,v}$ corresponds to the fixed point of $R_{ua}$, the values $N(\alpha)$ are unbounded.     Furthermore, if $N(\alpha) = \alpha$, then for all $\alpha' \in [\alpha, \omega_{k,v})$ we have $N(\alpha') \ge N(\alpha)$.    Thus, (since the hypothesis implies that $N=1$ is realized) each $N\in \mathbb N$ is realized as $N(\alpha)$ for some $\alpha \in \mathscr I_{k,v}$, and we can partition $\mathscr I_{k,v}$ by subintervals identified  from the value of $N(\alpha)$.  

Now, $\Theta_0(v) = vv'' = u(au)^2$, and hence $\Theta_0(v)$ is admissible on $[\zeta_{k,\Theta_0(v)}, \omega_{k,v})$.   Similarly, each $\Theta_{0}^{h}(v)$ is admissible on $[\zeta_{k,\Theta_{0}^{h}(v)}, \omega_{k,v})$.  By considering our ordering on words, it is clear that   $\zeta_{k,\Theta_{0}^{h}(v)}< \eta_{k, \Theta_{0}^{h}(v)} <  \omega_{k,v}$.   Therefore, the result holds for all $\Theta_{0}^{h}(v)$. 

The admissibility $v$ on all of $\mathscr I_{k,v}$ and the admissibility of $\Theta_0(v)$ on exactly $[\zeta_{k,\Theta_0(v)}, \omega_{k,v})$ implies that   there must be a shortest extension of $v = uau$ which is admissible for those $\alpha$  immediately to the left of $\zeta_{k,\Theta_0(v)}$.      Lemma~\ref{l:cylsOfConsecuChildrenAbut} shows that  $\mathfrak f(\, \Theta_1(v)\,)$ is this extension.     The fixed point of $R_{k,\mathfrak f(\, \Theta_1(v)\,)}$ is $r_0(\omega_{k,\Theta_1(v)})$,  and we again argue that arbitrarily high powers of this word, $\mathfrak f(\, \Theta_1(v)\,)$, must be admissible just to the left of the corresponding value $\alpha$, that is of $\omega_{k,\Theta_1(v)}$.    

Since {\em any} $v = uau$ is always a prefix of the square of the corresponding $\mathfrak f(v) = ua$, we find that  all of $\Theta_1(v)$ is admissible on an interval ending at 
$\omega_{k,\Theta_1(v)}$.  By definition of $\zeta_{k,\Theta_1(v)}$ 
it follows that this interval is all of 
$[\zeta_{k,\Theta_1(v)}, \omega_{k,\Theta_1(v)})$.     We iterate this argument for increasing $q$, to give that for each $q$,  $\Theta_q(v)$ is admissible on  exactly  $[\zeta_{k,\Theta_q(v)}, \omega_{k,\Theta_q(v)})$.     The definition of $\eta_{k,\Theta_q(v)}$ shows that it lies strictly between $\zeta_{k,\Theta_q(v)}$ and $\omega_{k,\Theta_q(v)})$.
 
\medskip

\noindent
{\bf Case 3.}   Suppose that $v =  \Theta_{0}^{h}(uau)$ for some $uau$ of Case 2.  Thus, $v = u(au)^{h+1}$.   

By the proof of Case 2, we can also assume our result for $\Theta_{0}^{h+1}(uau) = \Theta_0(v)$.   In particular, the left boundary of $\mathscr I_{k, \Theta_0(v)}$ does occur at $\zeta_{k, \Theta_0(v)}$.  By Lemma~\ref{l:cylsOfConsecuChildrenAbut},  $\omega_{k, \Theta_1(v)} = \zeta_{k, \Theta_0(v)}$.  Arguments as in the previous case yield that all of $\Theta_1(v)$ is admissible on an interval ending at 
$\omega_{k,\Theta_1(v)}$, and that this interval is indeed $[\zeta_{k,\Theta_1(v)}, \omega_{k,\Theta_1(v)})$.   In this case also,  induction on $q$ is successful.   That $\mathscr J_{\Theta_q(v)}\subset \mathscr I_{\Theta_q(v)}$ is here also straightforward.
\end{proof}

\bigskip
We are now ready for the following. \label{proofOfTheorem}
\begin{proof}[Proof of Theorem~\ref{p:cantorSet}]   That $(0, \gamma_{3,n}) = \cup_{k=1}^{\infty}\, \mathscr I_{k,1}$ follows simply from the fact that for $\alpha \in (0, \gamma_{3,n})$ and $x \in \mathbb I_{\alpha}$, 
$T_{\alpha}(x) = A^k C\cdot x$ for some $k$.       Proposition~\ref{p:Admissible} shows that for all $v \in \mathcal V$, 
$\mathscr I_{k,v} = [\zeta_{k,v}, \omega_{k,v})$ can be partitioned by $\mathscr J_{k,v} = [\zeta_{k,v}, \eta_{k,v})$ and its complement.        
By Corollary~\ref{c:vD(v)InConeGotEndpt}  (and the complementary results proven in cases 1 and 3 of the proof of Proposition~\ref{p:Admissible}), we have have that $\omega_{k, \Theta_0(v)} = \omega_{k,v}$.   By Lemma~\ref{l:cylsOfConsecuChildrenAbut},   for all $q\in \mathbb N$, $\omega_{k, \Theta_q(v)} = \zeta_{k, \Theta_{q-1}(v)}$.  Therefore, 
$\cup_{q=0}^{\infty}\, \mathscr I_{k, \Theta_q(v)}$ is a subinterval of $\mathscr I_{k,v}\setminus \mathscr J_{k,v}$ which has  $\omega_{k,v}$ as its right endpoint.   Finally, the definition of $\Theta_q(v)$ combined with Lemma~\ref{l:zetaRightDigs} shows that $\lim_{q \to \infty}\, \zeta_{k, \Theta_q(v)} = \eta_{k,v}$.   Therefore,  the left endpoint of the union is in fact the right endpoint of $\mathscr J_{k,v}$. 
\end{proof} 
 
\bigskip
\subsection{Potential  synchronization intervals  are intervals of synchronization} 

We now define $L_{k,v}$  exactly so that the group identity of Proposition~\ref{p:longWordInWmIsThree}  gives that 
$L_{k,v}=  C^{-1}AC R_{k,v}$, and thus  the main hypothesis of Lemma~\ref{l:etaDigits} will be satisfied.     
That synchronization does occur along $\mathscr J_{k,v}$ is then only a matter of showing that $L_{k,v}A\cdot \ell_0(\alpha)$ is admissible at all $\alpha\in \mathscr J_{k,v}$.

 For further ease, we  set 
 \[ w = w_{3,n} = (-1)^{n-2}, -2,  (-1)^{n-3},-2.\]  Note that the length of $w$ is $|w| = 2n - 3$. One of our first goals is to show that as $\alpha$ tends to zero,  $\underline{d}{}^{\alpha}_{[1,\infty)}$ begins with ever higher powers of $w$.  Recall from \eqref{e:WiD} that (for  any $m,n$)  the element 
$W = A^{-2}C \, (A^{-1}C)^{n-3} \, \big[A^{-2}C (A^{-1}C)^{n-2}\big]^{m-2}$,  
equals $W = A^{-1} C^{-1}   A C A$. Just as this is fundamental to understanding the case of $\alpha = 0$, so is it key to the study of left-orbits for small values of $\alpha$.     For ease of reference, the case of $m=3$ is  
\[ W =  A^{-2}C \, (A^{-1}C)^{n-3} \, A^{-2}C (A^{-1}C)^{n-2}\,.\]
  In the particular cases of $m=3$, \cite{CaltaSchmidt} show that for all $n$  the $T_{3,n, \alpha = 0}$-orbit of $\ell_0(\alpha)$ is purely periodic of period $w$.     

\subsubsection{Left digits are admissible}\label{sss:leftAdmiss}

For typographic ease,   let 
\begin{equation} \label{e:CandD} \mathcal C = \mathcal C_k =(-1)^{n-3}, -2, w^{k-1}\;\;\;\text{and}\;\;\; \mathcal D = \mathcal D_k =(-1)^{n-3}, -2, w^k\,.
\end{equation} 
Accordingly, we let 
\[ \tilde{\mathcal C} = \tilde{\mathcal C}_k = W^{k-1}A^{-2}C(A^{-1}C)^{n-3}\;\; \text{and}\;\; \tilde{\mathcal D} = \tilde{\mathcal D}_k = W^k A^{-2}C(A^{-1}C)^{n-3}.\]

\begin{Def}\label{d:lowDegSeq}     Suppose that $v \in \mathcal V$  and $k\in \mathbb N$. 

\begin{enumerate}
 \item  The lower (simplified)   digit sequence of  $k,v$  is 
\[\underline{d}(k,v) = w^k, \mathcal C^{c_1-1} \mathcal D^{d_1}\cdots \mathcal D^{d_{s-1}}\mathcal C^{c_s}, (-1)^{n-2}=  (-1)^{n-2}, -2,  \mathcal C^{c_1} \mathcal D^{d_1}\cdots \mathcal D^{d_{s-1}}\mathcal C^{c_s}, (-1)^{n-2} \,,\]
of length

\medskip
\item  $\underline{S}(k,v) = |\,\underline{d}(k,v)\,| = n + [(k-1)(2n-3)+n-2] \sum_{i=1}^s\, c_i +  [k(2n-3)+n-2]\sum_{j=1}^{s-1}\, d_j$.

\smallskip
\item The {\em left matrix} of  $k, v$ is
 \[ 
  \begin{aligned} 
L_{k,v} &=  (A^{-1} C)^{n-2} \;\tilde{\mathcal C}^{c_s}\tilde{\mathcal D}^{d_{s-1}} \cdots  \tilde{\mathcal D}^{d_1}  \tilde{\mathcal C}^{c_1-1} W^k A^{-1}\\
 &= (A^{-1} C)^{n-2} \; \tilde{\mathcal C}^{c_s}\tilde{\mathcal D}^{d_{s-1}} \cdots  \tilde{\mathcal D}^{d_1}  \tilde{\mathcal C}^{c_1}  A^{-2}C  (A^{-1}C)^{n-2}A^{-1}.
  \end{aligned}
\]
 \end{enumerate}
\end{Def}
\noindent 
Note that Proposition~\ref{p:longWordInWmIsThree}, below,  implies that $L_{k,v} = C^{-1}A C\, R_{k,v}$.

\bigskip
Our aim is to show the admissibility of $\underline{d}(k,v)$ on $\mathscr J_{k,v}$.  
In the following, we give both the left and right simplified digit sequence for each of the endpoints of $\mathscr J_{k,v}$.  The right sequences follow from the results above.  We present them here for ease of comparison.   
\begin{Lem}\label{l:leftExpansionsAtEndpoints}   Let $v = c_1 d_1\cdots c_s \in \mathcal V$ and 
$k \in \mathbb N$.    Suppose that $\underline{d}(k,v)$ is admissible on $\mathscr J_{k,v}$.   Set $\zeta = \zeta_{k,v}$ and $\eta = \eta_{k,v}$.
Then 
\[
\begin{aligned} \underline{d}{}^{\zeta}_{[1,\infty)} &=   \underline{d}(k, (\,\overleftarrow{v'}\,)^{\infty}\,)\,;\\
\\
\overline{d}{}^{\zeta}_{[1,\infty)} &= \overline{d}(k,v), \,   \underline{d}{}^{\zeta}_{[1,\infty)}\,;\\
\\
\underline{d}{}^{\eta}_{[1,\infty)} &=  \overline{ w^k, \mathcal C^{c_1-1} \mathcal D^{d_1}\cdots \mathcal D^{d_{s-1}}\mathcal C^{c_s}, (-1)^{n-3},-2}\,;\\
\\
\overline{d}{}^{\eta}_{[1,\infty)} &=  \overline{d}(k,v (v')^{\infty})\,.                  
\end{aligned}
\]  
\end{Lem}
\begin{proof}  Lemma~\ref{l:zetaRightDigs} yields the expressions for $\overline{d}{}^{\zeta}_{[1,\infty)}$ and  $\overline{d}{}^{\eta}_{[1,\infty)}$.

We first show that the expansions for $\ell_0(\zeta), \ell_0(\eta)$ are correct assuming admissibility of $\underline{d}(k, v)$ for all $\alpha \in [\zeta, \eta)$.

 Letting $L = L_{k,v}$,  Lemma~\ref{l:etaDigits} gives $A^{-2}C L A\cdot \ell_0(\zeta) = \ell_1(\zeta)$.  Thus, the digits of $\ell_0(\zeta)$ are periodic, with preperiod of length one.   Since $ w = (-1)^{n-2}, -2,  (-1)^{n-3},-2$, we find

\[
\begin{aligned} \underline{d}{}^{\zeta}_{[1,\infty)} &= -1, \overline{(-1)^{n-3}, -2,  (-1)^{n-3},-2, w^{k-1}, 
\mathcal C^{c_1-1} \mathcal D^{d_1}\cdots \mathcal D^{d_{s-1}}\mathcal C^{c_s},  (-1)^{n-2}, -2}\\
                                                                          &= w^{k}, \overline{
\mathcal C^{c_1-1} \mathcal D^{d_1}\cdots \mathcal D^{d_{s-1}}\mathcal C^{c_s},  (-1)^{n-2}, -2,  (-1)^{n-3}, -2,  (-1)^{n-3},-2, w^{k-1}}\\
                                                                             &= w^{k}, \overline{
\mathcal C^{c_1-1} \mathcal D^{d_1}\cdots \mathcal D^{d_{s-1}}\mathcal C^{c_s},  w,  (-1)^{n-3},-2, w^{k-1}}\\
                                                               &= w^{k}, \overline{
\mathcal C^{c_1-1} \mathcal D^{d_1}\cdots \mathcal D^{d_{s-1}}\mathcal C^{c_s-1},  \mathcal D,  \mathcal C}\\
                                                               &= w^{k},
\mathcal C^{c_1-1} \mathcal D^{d_1}\cdots \mathcal D^{d_{s-1}}\mathcal C^{c_s-1},  \mathcal D, \overline{ \mathcal C^{c_1} \mathcal D^{d_1}\cdots \mathcal D^{d_{s-1}}\mathcal C^{c_s-1}\mathcal D}\,.
 \end{aligned}
 \]
 Since $v$ is a palindrome,  this last is indeed the infinite sequence $ \underline{d}(k, (\,\overleftarrow{v'}\,)^{\infty}\,)$.    
 
Note that in the special case that $v=c_1, c>1$, we find $\underline{d}{}^{\zeta_{k,c}}_{[1,\infty)} = w^{k}, \overline{\mathcal C^{c-1}, w, \mathcal C} =  w^{k}, \overline{\mathcal C^{c-2}, \mathcal D, \mathcal C}$ and  $\underline{d}{}^{\zeta_{k,1}}_{[1,\infty)} = w^{k+1}, \overline{(-1)^{n-3}, -2, w^k} =  w^{k+1}, \overline{\mathcal D} $.
 
 Since $LA\cdot \ell_0(\eta) = r_0(\eta)$, we have that $A^{-1}LA$ fixes $\ell_0(\eta)$ and hence $\underline{d}{}^{\eta}_{[1,\infty)}$ is indeed purely periodic, with the indicated period.   
\end{proof}

\begin{Lem}\label{l:leftWAdmiss}   Fix $j\in \mathbb N$ and $0\le i< |w|$. If there is some $\alpha \le \gamma_{3,n}$ such that 
$\underline{d}{}^{\alpha}_{[1, j |w|+i ]\,} = w^j\,w_{[1, i]}$, then $\underline{d}{}^{\alpha'}_{1, j |w|+i \,} = w^j\,w_{[1, i]}$ for all $\alpha'<\alpha$.
\end{Lem} 
\begin{proof}  (Of course, if $i=0$, then $w_{[1, i]}$ is the empty word.) Since it is shown in \cite{CaltaSchmidt} that all powers of $W$ are admissible when $\alpha = 0$,  there are thus branches of digits corresponding to each $W^k\cdot \ell_0, A^{-1}CW^k\cdot \ell_0$, $(A^{-1}C)^2W^k\cdot \ell_0, \dots, (A^{-1}C)^{n-3}W^k\cdot \ell_0$ and $A^{-2}C (A^{-1}C)^{n-3}W^k\cdot \ell_0\,$ that continue to the right from $\alpha = 0$.   For each, by Lemma~\ref{l:admissBetween2Pts}, admissibility at $\alpha$  thus guarantees admissibility at each $\alpha' \le \alpha$. 
\end{proof}

\begin{Lem}\label{l:leftExpansionsAtOmegaK}   Fix 
$k \in \mathbb N$.    We have  

\[
\underline{d}{}^{\omega_{k,1}}_{[1,\infty)} =    w^k, \overline{(-1)^{n-3}, -2, w^{k-1}} = w^k, \overline{\mathcal C}\,.\]  

The digits 
$\underline{d}{}^{\alpha}_{[1, n-3 + k |w|\,]} = w^k, (-1)^{n-3}$ are admissible for all $\alpha \le \omega_{k,1}$. 
\end{Lem}
\begin{proof} 

Since $\mathfrak f(1) = 1$, the definition of $\omega_{k,v}$ yields $A^{k-1}C A\cdot \ell_0(\omega_{k,1}) = \ell_0(\omega_{k,1})$.  Lemma~\ref{l:shortRightId}, below,  shows that $A^{k-1}C A = C^{-1}A^{-1}C(A^{-1}C)^{n-2}W^{k-1}$.   For $\alpha \le \gamma_{3,n}$, we certainly have that $\ell_1(\alpha) = A^{-1}C\cdot \ell_0(\alpha)$, thus $\ell_1(\omega_{k,1}) =  A^{-2}C(A^{-1}C)^{n-2}W^{k-1}\cdot \ell_0(\omega_{k,1})$.  Thus,  the $T_{\alpha}$-orbit of $\ell_0(\omega_{k,1})$ is periodic, with minimal preperiod of length one.  Elementary manipulations give the claimed expression for the simplified digits,  assuming   admissibility.   

We have that the graph of the function $x \mapsto A^{-2}C(A^{-1}C)^{n-2}W^{k-1}A^{-1}\cdot x$ meets the vertical line $x = r_0(\omega_{k,1})$ at $y = \ell_1(\omega_{k,1})$.  Since the $T_{\omega_{k,1}}$-cylinder $\Delta(-2, 1)$ is full, there is also a point $y \in \mathbb I_{\omega_{k,1}}$ where the graph of the function $x \mapsto  (A^{-1}C)^{n-2}W^{k-1}A^{-1}\cdot x$ meets $x = r_0(\omega_{k,1})$.  By Lemma~\ref{l:leftWAdmiss}, this implies that  $ (A^{-1}C)^{n-2}W^{k-1}\cdot  \ell_0(\omega_{k,1})$ is admissible.   It follows that $A^{-2}C(A^{-1}C)^{n-2}W^{k-1}\cdot \ell_0(\omega_{k,1})$ is also admissible.  The rest of $w^k, \overline{(-1)^{n-3}, -2, w^{k-1}}$ is determined by periodicity and is thus also admissible.  

The second statement now follows immediately from Lemma~\ref{l:leftWAdmiss}.
\end{proof}

\begin{Rmk}\label{r:Omk+1EqZetaKwhenVis1}   Note that the above yields $\underline{d}{}^{\omega_{k+1,1}}_{[1,\infty)} = \underline{d}{}^{\zeta_{k,1}}_{[1,\infty)}$, in accordance with the fact that $\mathfrak f(1) = 1$ implies $\omega_{k+1,1} = \zeta_{k,1}$.
\end{Rmk}

 In the following proof and occasionally thereafter, we will have need of the following. 
\begin{Def}\label{d:notationForDroppingFinalDigits} 
For any word $z$,   let $z_{[-1]}, z_{[-2]}$ denote  the excision of the last two letters, or last letter, from $z$ respectively.  

\end{Def}
\bigskip

\begin{Prop}\label{p:leftDigitsAreGood}     Suppose that $v \in \mathcal V$ and $k\in \mathbb N$.  Then for 
all $\alpha \in \mathscr J_{k,v}$,  
\[  \underline{d}{}^{\alpha}_{_{[1,\underline{S})}}  =  \underline{d}(k,v)\,.\]
\end{Prop} 

\begin{table} 
\[\arraycolsep=3pt\def\arraystretch{1.2}
\begin{array}{c|ccc} 
                            &v&\alpha'&\alpha''\\[2pt]
  \hline
                            &1&0&\omega_{k,1}\\
                            &c>1&\eta_{k, c-1}&\omega_{k,c} =\omega_{k,1} \\[2pt]
\text{Base cases}&c\,1\,c, c\ge 1  &\eta_{k,c}&\omega_{k, c\,1\,c} = \zeta_{k, c+1}\\[2pt]
                            &\Theta_q(c), q\ge 2&\eta_{k,c}&\zeta_{k,\Theta_{q-1}(c)}\\[2pt]
                            &\Theta_1(c), c>1&\eta_{k,c}&\zeta_{k,\Theta_{0}(c)}\\[2pt]                   
  \hline \rule{0pt}{\dimexpr.7\normalbaselineskip+1mm}
\text{Case 1}&\Theta_{0}^{h}(c)&\eta_{k,\Theta_{0}^{h-1}(c)}&\zeta_{k,c+1}\\[5pt]
\hline \rule{0pt}{\dimexpr.7\normalbaselineskip+1mm}
\text{Case 2}&\Theta_q\circ \Theta_{p}(u), \; p,q\ge 1&\eta_{k, \Theta_{p}(u)}&\zeta_{k,\Theta_{q-1}\circ \Theta_{p}(u)}\\[5pt]
\hline \rule{0pt}{\dimexpr.7\normalbaselineskip+1mm}
\text{Case 3}&\Theta_{0}^{h}\circ \Theta_{p}(u), \, p\ge 1&\eta_{k, \Theta_{0}^{h-1}\circ \Theta_{p}(u)}&\zeta_{k,\Theta_{p-1}(u)}\\[2pt]
 \end{array}
\]
\bigskip
\caption{Admissibility of $\underline{d}(k,v)$ on $\mathscr J_{k,v}$ is shown by finding $\alpha'< \zeta_{k,v} < \eta_{k,v}< \alpha''$ such that  $\underline{d}_{[1, \infty)}^{\alpha}$  agrees with $\underline{d}(k,v)$ through to its penultimate digit for both $\alpha = \alpha', \alpha''$. See the proof of Proposition~\ref{p:leftDigitsAreGood}.}\label{t:casesForGoodPtsShowingAdmissibility}
\end{table}

\begin{proof}   For all $v$, we will exhibit  $\alpha' < \zeta_{k, v}$ and $\alpha'' > \eta_{k, v}$  such that both $\overline{d}{}^{\alpha'}_{[1,\infty)}$ and $\overline{d}{}^{\alpha''}_{[1,\infty)}$ are known to be admissible (by induction), and  share as a common prefix all but the final letter of $\underline{d}(k,v\,)$.   By Lemma~\ref{l:admissBetween2Pts}, the admissibility of these first $\underline{S}(k,v)-1$ digits then holds on $[\alpha', \alpha'']$.  Since we already know the admissibility of our right digits,   Lemma~\ref{l:alsoFinalDigitIsAdmissible},  below, applies and we can conclude admissibility of all of $\underline{d}(k,v\,)$ on $\mathscr J_{k,v}$.    See Table~\ref{t:casesForGoodPtsShowingAdmissibility} for a summary of the pairs $\alpha', \alpha''$ used for the various cases of $v$.
\bigskip 

\noindent
{\bf Base cases.} 
Consider $v=c_1 = c$.   The lower simplified digit sequence  $\underline{d}{}^{\omega_{k,1}}_{[1,\infty)}  = w^k, \overline{\mathcal C}$ agrees with  $\underline{d}(k, c) = w^k, \mathcal C^{c-1}, (-1)^{n-2}$ through to its penultimate digit. 
When $v=1$, Lemma~\ref{l:leftExpansionsAtOmegaK} 
shows that these shared digits are admissible for all $\alpha \le \omega_{k,1}\,$; therefore our proof template succeeds in this case.   
For $c>1$, by induction  $\underline{d}{}^{\eta_{k,c-1}}_{[1,\infty)} = \overline{ w^k, \mathcal C^{c-2} , (-1)^{n-3},-2}$ is admissible, and has $\underline{d}(k, c)$ as a prefix.   
\medskip 

\noindent
Suppose $v=c\,1\,c$.   Here we use $\alpha' = \eta_{k,c}$, and $\alpha'' =  \omega_{k,c\,1\,c}$.   Since $\underline{d}{}^{\eta_{k,c}}_{[1,\infty)} = w^k, \overline{\mathcal C^{c-1}\mathcal D}$, it has  the prefix $w^k, \mathcal C^{c-1}\mathcal D \mathcal C^c, w$.  This in turn has $\underline d(k,c\,1\,c)$ as a prefix. 
 Lemma~\ref{l:consecSpecialCase} and Lemma~\ref{l:leftExpansionsAtEndpoints} give that 
$\underline{d}{}^{\omega_{k,c\,1\,c}}_{[1,\infty)} =  \underline{d}{}^{\zeta_{k,c+1}}_{[1,\infty)} =  w^{k}, \overline{\mathcal C^{c-1}, \mathcal D, \mathcal C}$.  This agrees with  $\underline d(k,c\,1\,c)$  through to the penultimate digit of this latter.   Thus the admissibility holds on $\mathscr J_{k, c\,1\,c}$.

\medskip 

\noindent
Suppose $v=\Theta_q(c), q\ge 1$ and $c>1$.  We again use  $\alpha' = \eta_{k,c}$.  The sequence $\underline{d}{}^{\eta_{k,c}}_{[1,\infty)}$ and $\mathcal D = \mathcal C, w$,  here $\underline{d}{}^{\alpha'}_{[1,\infty)}$  has  the prefix $w^k, (\mathcal C^{c-1}\mathcal D)^{q+1} \mathcal C^c, w$.  This in turn has $\underline d(k,\Theta_q(c))$ as a prefix.      Let $\alpha'' = \zeta_{k, \Theta_{q-1}(c)}$.  Then 
$\underline{d}{}^{\alpha''}_{[1,\infty)} =    w^k, \overline{(\mathcal C^{c-1}\mathcal D)^{q+1}\mathcal C}$,  which  has the prefix $w^k, (\mathcal C^{c-1}\mathcal D)^{q+1} \mathcal C^c, (-1)^{n-3}, -2$.  It thus agrees with  $\underline d(k,\Theta_q(c))$   through to the penultimate digit of this latter.   Thus the admissibility holds.   One checks that the same form of $\alpha', \alpha''$ works for $v=\Theta_q(c), q\ge 2$ and $c=1$.

\medskip  

\noindent
{\bf Case 1.}  Suppose   $v = \Theta_{0}^{h}(c), h\ge 2$ for some $c> 1$.  We have $v =c (1\,c)^h$.   We induce on $h$, setting $\alpha' = \eta_{k,  \Theta_{0}^{h-1}(c)}$.   Then,  $\underline{d}{}^{\alpha'}_{[1,\infty)}$  has  the prefix $w^k, \mathcal C^{c-1}, (\mathcal D\mathcal C^c)^h, (-1)^{n-3}, -2$, which agrees through to the  penultimate digit of $\underline d(k, v)$.  One easily checks that setting $\alpha'' = \zeta_{k,c+1}$ yields $\underline{d}{}^{\alpha''}_{[1,\infty)}$ of prefix $\underline d(k, v)$.

\medskip  
\noindent
{\bf Case 2.}  
 Suppose that the result holds for $v = uau = \Theta_p(u)$ with $u \in \mathcal V$ and $p \ge 1$.  We prove that the result holds for   $\Theta_q(v), q\ge 1$.    We take $\alpha' = \eta_{k, v}$ and  $\alpha'' = \zeta_{k, \Theta_{q-1}(v)}$ if $q\ge 1$. 
To appropriately restrict the use of $\underline d(k,v)$  to prefixes $u$ of $v\in \mathcal V$, note that we must in particular suppress the final digit of $-1$;  we denote this by  $\underline d'(k,u)$.   Recall that $v' =  a \overleftarrow{u'} u$, $v'' = au$ and that $u$ both begins and ends with the letter $c_1$.

  When $c_1 > 1$, from $\overleftarrow{u'} = u_{[-1]} (c_1-1)\,1$ we find
 \[
 \begin{aligned}
 \underline{d}{}^{\eta_{k,v}}_{[1,\infty)} &= \underline{d}(k,v (v')^{\infty}) = \underline{d}(k,v (v')^q a \overleftarrow{u'} u\, (v')^{\infty})\\
          & = \underline{d}'(k,v (v')^q a u_{[-1]}) \,\mathcal C^{c_1-1}  \mathcal D  \,\mathcal C^{c_1} \cdots \,.
\end{aligned}
 \]
 On the other hand, still with $c_1 >1$, we have 
 \[\underline{d}(k,\Theta_q(v)\,) = \underline{d}'(k,v (v')^q a u_{[-1]}) \, \mathcal C^{c_1}\,.\]
Recall that $\mathcal D = \mathcal C, w$,  thus $\mathcal C^{c_1-1}  \mathcal D = \mathcal C^{c_1}, w$;   we see that 
$\underline{d}(k,\Theta_q(v)\,)$ is indeed a prefix of $\underline{d}{}^{\eta_{k,v}}_{[1,\infty)}$.  

 When $c_1 = 1$,   $\underline{d}(k,\Theta_q(v)\,) = \underline{d}'(k,v (v')^q a u_{[-1]})  \mathcal D^{d_1}\mathcal C$ while $ \underline{d}{}^{\eta_{k,v}}_{[1,\infty)} = \underline{d}'(k,v (v')^q a u_{[-2]})  \mathcal D^{d_1+1}  \,\mathcal C \cdots \,$.   Thus, since $\mathcal C$ is a prefix of $\mathcal D$,  we find that here also 
 $\underline{d}(k,\Theta_q(v)\,)$ is a prefix of $\underline{d}{}^{\eta_{k,v}}_{[1,\infty)}$. 

 Since $\Theta_{q-1}(v)= v(v')^{q-1} v''$ is a palindrome,    $\overleftarrow{(\Theta_{q-1}(v)\,)' } = v(v')^{q-1}a \overleftarrow{u'}$.  By Proposition~\ref{p:theWordsAreNice}, $v'v'' = a \overleftarrow{u'} v$, we find
\begin{equation}\label{eq:zetaWordCalc}
\begin{aligned}  
(\, \overleftarrow{(\Theta_{q-1}(v)\,)' }  \,)^{\infty} &=  v\; [ (v')^{q-1} v' v''\,]^{\infty}\\
                                                                   &=  v \;[\, (v')^q v''\,]^{\infty}.
\end{aligned}
\end{equation}  
Thus
 Lemma~\ref{l:leftExpansionsAtEndpoints}   yields 
$\underline{d}{}^{\zeta_{k,\Theta_{q-1}(v)} }_{[1,\infty)} =\underline{d}(k, v \,[\, (v')^q v''\,]^{\infty} \,)$.
Therefore, $\underline{d}'(k, v (v')^{q} v''), (-1)^{n-3}, -2$ is a prefix of $\underline{d}{}^{\zeta_{k,\Theta_{q-1}(v)} }_{[1,\infty)}$.  That is, $\underline{d}{}^{\zeta_{k,\Theta_{q-1}(v)} }_{[1,\infty)}$  agrees with $\underline{d}(k, \Theta_{q}(v))$ exactly through to its penultimate digit.    

\medskip   
\noindent
{\bf Case 3.}  Again suppose that $v = \Theta_p(u) = uau$ in our usual notation.   Then $v'' = (u')^{p} u''$, and one finds that $\Theta_{0}^{h}\circ \Theta_p(u) = \Theta_{p}(u)\, [(u')^{p} u'']^{h}$.   For these words, we use   $\alpha'' = \zeta_{k, \Theta_{p-1}(u)}$.  
Indeed,  from \eqref{eq:zetaWordCalc}, $\underline{d}{}^{\zeta_{k,\Theta_{p-1}(u)} }_{[1,\infty)}$   agrees with $\underline{d}(k, \Theta_{0}^{h}\circ \Theta_p(u)\,)$ exactly through to its penultimate digit.

 Since $(\,\Theta_{0}^{h-1}\circ \Theta_p(u) \,)' = u'\, [(u')^{p} u'']^{h}$, we find that 
\begin{equation}\label{eq:thetaZeroNeighborsAndPrimes}
\Theta_{0}^{h-1}(v) \cdot (\,\Theta_{0}^{h-1}(v) \,)' = \Theta_{0}^{h}(v)\cdot u' [(u')^{p} u'']^{h}.
\end{equation}
 Therefore, Lemma~\ref{l:leftExpansionsAtEndpoints} yields that $\alpha' = \eta_{k, \Theta_{0}^{h-1}(v)}$ allows the proof to succeed.  
 \end{proof}

\medskip 
\begin{Lem}\label{l:alsoFinalDigitIsAdmissible}    With notation as above,  
suppose that  for all $\alpha \in [\zeta_{k,v}, \eta_{k,v}]$, both  
\begin{enumerate}
\item[(a)] $\overline{d}{}^{\alpha}_{[1,\overline{S}]} =  \overline{d}(k,v)$; 

\smallskip
\item[(b)] $\underline{d}{}^{\alpha}_{[1,\underline{S}-1]}$ agrees with the initial subword of length $\underline{S}-1$ of $\underline{d}(k,v)$.  
\end{enumerate}

\medskip 
\noindent
Then for all $\alpha \in [\zeta_{k,v}, \eta_{k,v})$ we have $\underline{d}{}^{\alpha}_{[1,\underline{S}]} =  \underline{d}(k,v)$.  
\end{Lem} 
\begin{proof}   By the first hypothesis,  there is some $\chi = \chi_{k,v}$ such that $R_v\cdot r_0(\chi_v) = 0$.   Since $C^{-1}ACR_{k,v} = L_{k,v}$ and it is trivially verified that $C^{-1}AC$ fixes zero,  we find that $L_{k,v}\cdot \ell_0(\chi) = 0$.   In particular, we find that  $\underline{d}{}^{\chi}_{[1,\underline{S}]} =   \underline{d}(k,v)$.   

Now by continuity and the fact that M\"obius functions are increasing functions, we can invoke the second conclusion of Lemma~\ref{l:leftRegionOneBeforeSyn} on an interval around $\chi$ to find in particular that  $L_{k,v}\cdot \ell_0(\alpha)> R_v\cdot r_0(\alpha)$ holds from $\alpha = \zeta_{k,v}$ until $L_{k,v}\cdot \ell_0(\alpha) = r_0(\alpha)$.  But, this describes exactly the interval $[\zeta_{k,v}, \eta_{k,v})$.   
\end{proof}

\subsection{There are no other points of synchronization}  Suppose that $\alpha < \gamma_{3,n}$ is not in any $\mathscr J_{k,v}$ and is also not equal to any $\eta_{k,v}$.   There is some $k$ such that $\alpha \in \mathscr I_{k,1}$,  but of course $\alpha \notin \mathscr J_{k,1}$;  there is thus a unique $q_1$ such that $\alpha \in \mathscr I_{k,\Theta_{q_{1}}(1)}$.   Again,  $\alpha \notin \mathscr J_{k,\Theta_{q_{1}}(1)}$ and thus there is a unique $q_2$ with $\alpha \in \mathscr I_{k,\Theta_{q_{2}}\circ\Theta_{q_{1}}(1)}$.  Clearly this process iterates, and we find that there is an infinite sequence of $q_i$ such that $\alpha \in \cap_{j=1}^{\infty}\, \mathscr I_{k, \Theta_{q_{j}}\circ \cdots \circ\Theta_{q_{1}}(1)}$.  Recall that  for any $v \in\mathcal V$  and any $q$, $\Theta_q(v)$ has $v$ as a prefix.  Therefore,  the sequence of the $q_i$ uniquely determines both $\underline{d}{}^{\alpha}_{[1,\infty)}$ and $\overline{d}{}^{\alpha}_{[1,\infty)}$.   In particular,  $\overline{d}{}^{\alpha}_{[1,\infty)}$ has digits only in $\{k, k+1\}$,  while  $\underline{d}{}^{\alpha}_{[1,\infty)}$ has digits only in $\{-1, -2\}$.  Therefore, the two orbits cannot synchronize.

Note that when $\alpha$ is some $\eta_{k,v}$ then again   $\overline{d}{}^{\alpha}_{[1,\infty)}$ has digits only in $\{k, k+1\}$,  while  $\underline{d}{}^{\alpha}_{[1,\infty)}$ has digits only in $\{-1, -2\}$. 

\subsection{The non-synchronization set is of measure zero}\label{ss:nonSynchMeasZero}    
Although we have proven that the complement of the $\mathscr J_{k,v}$ is a Cantor set, it is then still possible that it could be a so-called fat Cantor set thus one of positive measure.    We will easily show that the non-synchronization points have left expansions that involve only $-1$ and $-2$ as simplified digits.   We will argue by way of the maps of \cite{CaltaSchmidt} that the set of such $\alpha$ is of measure zero.

The $T_{\alpha}$-orbit of $\ell_0(\alpha)$ is always in the set of points whose simplified digits are $-1$ or  $-2$.  In particular, this orbit certainly remains in $[-t,0)$.  But $[-t,0)$ is the interval of definition of the map $g = g_{3,n}$ studied in \cite{CaltaSchmidt}, where $g(x) = A^k C\cdot x$ with $k$ defined exactly so that this image lies in $[-t,0)$.  Therefore,  for any $\alpha$ whose $T_{\alpha}$ orbit of $\ell_0(\alpha)$  remains in the set with simplified digits $-1$ or  $-2$ (note that every point in such an orbit is less than zero), this $T_{\alpha}$-orbit is the $g$-orbit of $\beta := \ell_0(\alpha)$.    Furthermore, \cite{CaltaSchmidt}  shows that $g$ is ergodic with respect to what is naturally called a Gauss measure, although this   measure is infinite.    A so-called acceleration of $g$, a map $f$ on $[-t,0)$ is then shown in \cite{CaltaSchmidt} to be ergodic with respect to a {\em finite}  measure (which is equivalent to Lebesgue).   The process of acceleration involves taking well-defined subsequences of $g$-orbits.      We thus find that the $f$-orbit of $\beta$ remains of small digits, and due to the Ergodic Theorem,  $\beta$ lies in a measure zero subset.

\section{The group element identities for the setting $\alpha < \gamma_{3,n}$}\label{s:groupIdsSmallAlp}

We gather key group identities, used above,  in this section.

\begin{Lem}\label{l:shortRightId}   Fix $m=3$.   For integers $a, k \ge 1$ and $u \in \mathbb Z$, we have   
\[ A^u C (A^{k}C)^a =  A^{u-1}C \, (A^{-1}C)^{n-2}\, [W^{k-1} A^{-2}C (A^{-1}C)^{n-3}]^{a-1} W^{k}A^{-1}\;.
              \]
\end{Lem} 
\begin{proof}   
For $a=1$, we  give an induction proof.  We use repeatedly \eqref{e:WiD}:  $W = A^{-1} C^{-1}   A C A$.  First, 
\[A^uCACA = A^u C\; CAW = A^u C^3\; (C^{-1} A) \; W = A^u C^3 (A^{-1}C)^{n-1} W\,,\]
giving the case of $k=1$, as $m=3$. 
Now,  we complete the proof in the case of $a=1$ by induction: 
\[ 
\begin{aligned}
 A^u CA^{k+1}C &= (A^u C A^kC) \; C^{-1}AC \\
                          &= A^{u-1} C  (A^{-1}C)^{n-2} W^k\; (A^{-1} C^{-1}AC) \\
                          &= A^{u-1} C  (A^{-1}C)^{n-2} W^{k+1}A^{-1}\,.
\end{aligned}
\]

Recalling that with $m=3$, we have $W = A^{-2}C \, (A^{-1}C)^{n-3} \,  A^{-2}C (A^{-1}C)^{n-2} $, the induction step for increasing values of $a$ is given by  
\[ 
\begin{aligned}
 A^u C (A^{k}C)^{a+1} &= A^{u-1} C  (A^{-1}C)^{n-2} [W^{k-1}A^{-2}C(A^{-1}C)^{n-3}]^{a-1} W^kA^{-1}\; A^k C \\
                          &= A^{u-1} C  (A^{-1}C)^{n-2} [W^{k-1}A^{-2}C(A^{-1}C)^{n-3}]^{a}\,A^{-2}C(A^{-1}C)^{n-3} \, A^{-1}C \; A^{k-1} C \\
                          &= A^{u-1} C  (A^{-1}C)^{n-2} [W^{k-1}A^{-2}C(A^{-1}C)^{n-3}]^{a}\\
                          & \;\;\;\;\;\; \;\;\;\;\;\;\times A^{-2}C(A^{-1}C)^{n-3}  \, A^{-2}C (A^{-1}C)^{n-2} W^{k-1} A^{-1} \\
                          &= A^{u-1} C  (A^{-1}C)^{n-2} [W^{k-1}A^{-2}C(A^{-1}C)^{n-3}]^{a}\,  W^{k} A^{-1} \,,
\end{aligned}
\]
where in the passage from the second to the third line we have applied the identity in the case of $a=1$.            
\end{proof}

\begin{Lem}\label{l:wInTheMiddleMis3}   With $m= 3$, for any integer $k \ge 1$   we have   
\[W A^{-1}\,\cdot  A^k C A^{-1}C =   A^{-2}C \,(A^{-1}C\,)^{n-3}  W^{k} \,A^{-2}C \;.
              \]
\end{Lem} 
\begin{proof} We show this by induction.  
 
Recall \eqref{e:WiD}: $W=  A^{-1} C^{-1}   A C A$.   Our base case is  $k=1$, where we find 
\[
\begin{aligned}
W A^{-1}\,\cdot  A^{1} C A^{-1}C &=   A^{-1} C^{-1}   A C  A C A^{-1}C \\
         & = A^{-1} C^{-1}   A C  A CA  A^{-2}C\\
         &= A^{-1} C^{-1}   A C ( CA A^{-1} C^{-1} ) A CA  A^{-2}C\\
         & = A^{-1} (C^{-1}   A)  C^2   A W  A^{-2}C \\
         &= A^{-1} (A^{-1}C)^{n-1}C^2 A W  A^{-2}C\\
         &= A^{-2} C (A^{-1}C)^{n-3} A^{-1} C C^2 A W  A^{-2}C\\
         &=  A^{-2} C (A^{-1}C)^{n-3} W  A^{-2}C         
  \end{aligned}
  \]       
We have used that $C = AB$ gives $A^{-1}C$ is of order $n$, and that $C^3$ is the projective identity.

Now assume that the identity holds for $k$.  We find 
\[
\begin{aligned}
W A^{-1}\,\cdot  W^{k+1}A^{1} C A^{-1}C &=   (WA^{-1}A) \; A^k C A^{-1}C \\
         &= W \; A W^{-1} A^{-2} C (A^{-1}C)^{n-3} W^k  A^{-2}C  \\
         &= W \; A W^{-1}  A^{-1} \; (A^{-1}C)^{n-2}W^k  A^{-2}C  \\
         &= W \; A W^{-1}  A^{-1} \; (C^{-1}A)^2 W^k  A^{-2}C  \\
         &= W \; A W^{-1}  A^{-1}\;(C^{-1}A)^2 (W^{-1} W) W^k  A^{-2}C  \\
         &= W \; A W^{-1}  A^{-1}\;  C^{-1}A(C^{-1}A A^{-1} C^{-1})A^{-1} CA \;  W^{k+1}  A^{-2}C  \\
         &= W \; A (W^{-1})  A^{-1}\;  C^{-1}AC^{-2}A^{-1} CA \;  W^{k+1}  A^{-2}C  \\
         &= W \; (A A^{-1}) C^{-1}(A^{-1} CA  A^{-1}\;  C^{-1}A)C^{-2}A^{-1} CA \;  W^{k+1}  A^{-2}C  \\
         &= W \;   (C^{-3}) A^{-1} CA \;  W^{k+1}  A^{-2}C  \\
         &= (W) \; A^{-1} CA \;  W^{k+1}  A^{-2}C  \\
         &=  A^{-1} C^{-1}   A C (A A^{-1}) CA \;  W^{k+1}  A^{-2}C  \\
         &=  A^{-1} C^{-1}   A (C^2) A \;  W^{k+1}  A^{-2}C  \\
         &=  A^{-1} (C^{-1}   A)^2  \;  W^{k+1}  A^{-2}C  \\
         &=  A^{-1} (A^{-1}   C)^{n-2}  \;  W^{k+1}  A^{-2}C  \\
         &=  A^{-2} C (A^{-1}   C)^{n-3}  \;  W^{k+1}  A^{-2}C.
\end{aligned}         
  \]          
\end{proof}

\begin{Prop}\label{p:longWordInWmIsThree}   With $m=3$, for any integer $u$    we have   
\[
\begin{aligned} 
 &A^uC (A^kC)^{a_s}\, (A^{k+1}C)^{b_{s-1}}\, (A^kC)^{a_{s-1}}\,\cdots (A^{k+1}C)^{b_1}\, (A^kC)^{a_1}\\
&=\;A^{u-1} C A^{-1} C   [(A^{-1}C)^{n-3}W^{k-1}A^{-2}C]^{a_s} [(A^{-1}C)^{n-3}W^k A^{-2}C]^{b_{s-1}}\,  [(A^{-1}C)^{n-3}W^{k-1}A^{-2}C]^{a_{s-1}}\,\cdots\\
&\;\;\; \cdots   [(A^{-1}C)^{n-3}W^k A^{-2}C]^{b_{1}}\,  [(A^{-1}C)^{n-3}W^{k-1}A^{-2}C]^{a_1-1 } (A^{-1}C)^{n-3} W^k A^{-1}\,.
\end{aligned} 
\]

\end{Prop} 
\begin{proof}   Lemma~\ref{l:shortRightId} yields that each 
\[
 (A^{k+1}C)^b(A^kC)^a  =  (A^{k+1}C)^{b-1} A^kC (A^{-1}C)^{n-2} (W^{k-1}A^{-2}C(A^{-1}C)^{n-3} )^{a-1} W^k A^{-1}.
\]
Thus, if $b>1$ we have 
\[
\begin{aligned} 
 (A^{k+1}C)^b(A^kC)^a  &=  (A^{k+1}C)^{b-2}  A^kC (A^{-1}C)^{n-2} W^k A^{-2}C (A^{-1}C)^{n-3}(W^{k-1}CA^{-2}C)^{a-1} W^k A^{-1}\\
 &\; \vdots\\
                                       &= A^k C  (A^{-1}C)^{n-2} \; [W^k A^{-2}C (A^{-1}C)^{n-3}]^{b-1} [W^{k-1}CA^{-2}C(A^{-1}C)^{n-3}]^{a-1} W^k A^{-1}\,,
\end{aligned} 
\]
giving a formula that is also valid when $b=1$.

Therefore, 
\[
\begin{aligned} 
&(A^{k+1}C)^{b_{j}}\, (A^kC)^{a_{j}}\;(A^{k+1}C)^{b_{j-1}}\, (A^kC)^{a_{j-1}}  \\
&=\; A^k C  (A^{-1}C)^{n-2} \; [W^k A^{-2}C (A^{-1}C)^{n-3}]^{b_{j}-1} [W^{k-1}CA^{-2}C (A^{-1}C)^{n-3}]^{a_{j}-1} W^k A^{-1} \\
&\;\;\;\;\;\;\;\; \times A^k C  (A^{-1}C)^{n-2} \; [W^k A^{-2}C (A^{-1}C)^{n-3}]^{b_{j-1}-1} [W^{k-1}CA^{-2}C (A^{-1}C)^{n-3}]^{a_{j-1}-1} W^k A^{-1}\\
&=\; A^k C  (A^{-1}C)^{n-2}  [W^k A^{-2}C (A^{-1}C)^{n-3}]^{b_{j}-1} [W^{k-1}CA^{-2}C (A^{-1}C)^{n-3}]^{a_{j}-1} W^{k-1}\;  \\
&\;\;\;\;\;\;\;\; \times W A^{-1} A^k C A^{-1}C\\
&\;\;\;\;\;\;\;\;\;\;\;\;\;\times \;  [ (A^{-1}C)^{n-3}W^k A^{-2}C]^{b_{j-1}-1}  (A^{-1}C)^{n-3} [W^{k-1}CA^{-2}C (A^{-1}C)^{n-3}]^{a_{j-1}-1} W^k A^{-1}\\
&=\; A^k C  (A^{-1}C)^{n-2} \; [W^k A^{-2}C (A^{-1}C)^{n-3}]^{b_{j}-1} [W^{k-1}A^{-2}C (A^{-1}C)^{n-3}]^{a_{j}-1} W^{k-1}\;  \\
&\;\;\;\;\;\;\;\; \times A^{-2} C (A^{-1}C)^{n-3}W^k  A^{-2}C\\
&\;\;\;\;\;\;\;\;\;\;\;\;\;\times \;  [(A^{-1}C)^{n-3}W^k A^{-2}C ]^{b_{j-1}-1} [(A^{-1}C)^{n-3}W^{k-1}A^{-2}C]^{a_{j-1}-1} (A^{-1}C)^{n-3}W^k A^{-1}\\ 
&=\; A^k C  A^{-1}C\; [(A^{-1}C)^{n-3} W^k A^{-2}C]^{b_{j}-1} [(A^{-1}C)^{n-3}W^{k-1}A^{-2}C ]^{a_{j}}   \\
&\;\;\;\;\;\;\;\; \times \;  [(A^{-1}C)^{n-3}W^k A^{-2}C]^{b_{j-1}}(A^{-1}C)^{n-3} [(A^{-1}C)^{n-3}W^{k-1}CA^{-2}C]^{a_{j-1}-1} (A^{-1}C)^{n-3}W^k A^{-1}\,,
\end{aligned} 
\]
where for the penultimate line we apply Lemma~\ref{l:wInTheMiddleMis3}.

We now repeatedly use  Lemma~\ref{l:wInTheMiddleMis3} so as to find that  \[
\begin{aligned} 
&(A^{k+1}C)^{b_{s-1}}\, (A^kC)^{a_{s-1}}\,\cdots (A^{k+1}C)^{b_1}\, (A^kC)^{a_1} \\
&=A^k C A^{-1}C\; [(A^{-1}C)^{n-3}W^k A^{-2}C]^{b_{s-1}-1}\,  [(A^{-1}C)^{n-3}W^{k-1}A^{-2}C]^{a_{s-1}}\,\cdots\\
&\;\;\; \cdots   [(A^{-1}C)^{n-3}W^k A^{-2}C]^{b_{1}}\,  [(A^{-1}C)^{n-3}W^{k-1}A^{-2}C]^{a_1-1 } (A^{-1}C)^{n-3} W^k A^{-1}\,.
\end{aligned} \]

Now multiply both sides on the left by    $A^u C (A^kC)^{a_s}$.  On the right hand side, we collect an $A^kC$ and thus apply   Lemma~\ref{l:shortRightId}   to replace   
$A^u C (A^kC)^{a_s + 1}$ by \newline 
$A^{u-1} C A^{-1}C \, [(A^{-1}C)^{n-3}W^{k-1} A^{-2}C]^{a_s} (A^{-1}C)^{n-3}W^{k}A^{-1}$.   The final $A^{-1}$ of this substitution is now adjacent to the leftmost $A^{-1}C$ in the above display, and thus we can combine and regroup to find the desired expression.
\end{proof} 


\section{Synchronization for  $\alpha >\epsilon_{3,n}, n \ge m$}\label{s:LargeAlps}

Fix $m=3$ and $n \ge 3$.   Let $\epsilon = \epsilon_{3,n}$ be such that $A^{-1}C\cdot \ell_0(\epsilon) = r_0(\epsilon)$.    Then the parameter subinterval $[\epsilon, 1)$ is partitioned by subintervals indexed by $k\ge 2$ and characterized by  $\ell_1(\alpha) = A^{-k} C \cdot \ell_0(\alpha)$.   (When $n=3$,  one finds that $\epsilon_{3,3} = G/2$, see Figure~\ref{firstBothDigs}.)
\begin{Thm}\label{t:FullMeasureLargeAlps}   For $m = 3$ and $n\ge m$,  let $\epsilon = \epsilon_{3,n}$.
The set of $\alpha \in (\epsilon,1)$ such that there exists $i = i_{\alpha}, j = j_{\alpha}$ with 
$T_{3,n,\alpha}^{i}(\,r_0(\alpha)\,) = T_{3,n,\alpha}^{j}(\,\ell_0(\alpha)\,)$ is of full measure.  
\end{Thm} 

Synchronization for these large values of $\alpha$ holds in a manner closely analogous to that for small $\alpha$.  We will find that the intervals indexed by exactly the same set of  words $\mathcal V$, although the indexing will depend on negative integers and be given in terms of left digits.   There are differences: in particular,  each potential synchronization interval is the union of what could fairly be called two distinct synchronization intervals.   The right orbit requires an extra step before synchronization on one of the two subintervals.  

It will naturally be important to know the initial digits of $r_0(\alpha)$ in this range.   For this,  define 
$\delta = \delta_{3,n}$ as the value of $\alpha$ such that $C^{-1}\cdot \ell_0(\delta) = A^{-1}C\cdot \ell_0(\delta)$.   (Note that $ \delta_{3,n} <  \epsilon_{3,n}$, see Figure~\ref{firstBothDigs}.)    Thus, $C A^{-1}C A^{-1} \cdot r_0(\delta) = \ell_0(\delta)$.   That is,   $(AC^2)^{n-2}\cdot r_0(\delta) = \ell_0(\delta)$.  By  Proposition~\ref{p:veeOnTheRight}, $(AC^2)^{n-2}\cdot r_0(\alpha)$ is admissible for $\alpha = 1$  and we conclude   that $(AC^2)^{n-2}\cdot r_0(\alpha)$ is admissible for all $\alpha \in (\, \delta, 1]$ and in particular for all $\alpha> \epsilon_{3,n}$.

\subsection{Synchronization intervals have Cantor set complement}
For these large $\alpha$,   synchronization is signaled by  left and right digits being related by 
$C^{-1}A C^{-1}$.

\begin{Lem}\label{l:rightRegionOneBeforeSyn}  Fix $m=3$.   Suppose that $\alpha$ is such for all $x\in \mathbb I_{\alpha}$,  $d^{\alpha}(x) = (k, \ell)$ with $\ell \in \{1,2\}$.     Fix $i,j \in \mathbb N$.   Suppose that  $\ell_{i-1} = C^{-1}AC^{-1} \cdot r_{j-1}$ and $C \cdot r_{j-1} \in  \mathbb I_{\alpha}$.    
Then 
\begin{enumerate}
\item[(i)]   If   $r_{j} = AC^2\cdot r_{j-1}$ then  $\ell_i = r_{j+1}$;
\item[(ii)] otherwise,   $\ell_i = r_j$
\end{enumerate}
\end{Lem} 
\begin{proof} Since $C \cdot r_{j-1} \in \mathbb I_{\alpha}$, there is some $u$ such that $r_{j} = A^uC^2\cdot r_{j-1}$.   

 If  $r_{j} = AC^2\cdot r_{j-1}$ then $C\cdot \ell_{i-1} = C \cdot (C^{-1}AC^{-1} \cdot r_{j-1}) \in \mathbb I_{\alpha}$.  Therefore, there is some $s$ such that $\ell_i = A^sC^2\cdot \ell_{i-1}$.   But then $\ell_i = A^sC^2\cdot (C^{-1}AC^{-1} \cdot r_{j-1}) = A^sCAC^2\cdot r_{j-1} = A^sC \cdot r_j$.   By definition, $A^{-s}\cdot \ell_i \notin \mathbb I_{\alpha}$, it follows that $C \cdot r_j \notin \mathbb I_{\alpha}$ and therefore we conclude that $r_{j+1} = A^sC \cdot r_j = \ell_i$.   

If $r_{j} =  A^uC^2\cdot r_{j-1}$ with $u \neq 1$, then $C\cdot \ell_{i-1}= AC^2\cdot r_{j-1} \notin \mathbb I_{\alpha}$.  Therefore, there is some $s$ such that $\ell_i = A^s C \cdot \ell_{i-1}$.  We find that  $\ell_i = A^s C \cdot  (C^{-1}AC^{-1} \cdot r_{j-1}) = A^{s+1} C^2 \cdot r_{j-1}$.   We conclude that $u= s+1$ and $\ell_i = r_j$. 
\end{proof} 
\bigskip 

\begin{figure}[h]
\scalebox{.4}{
{\includegraphics{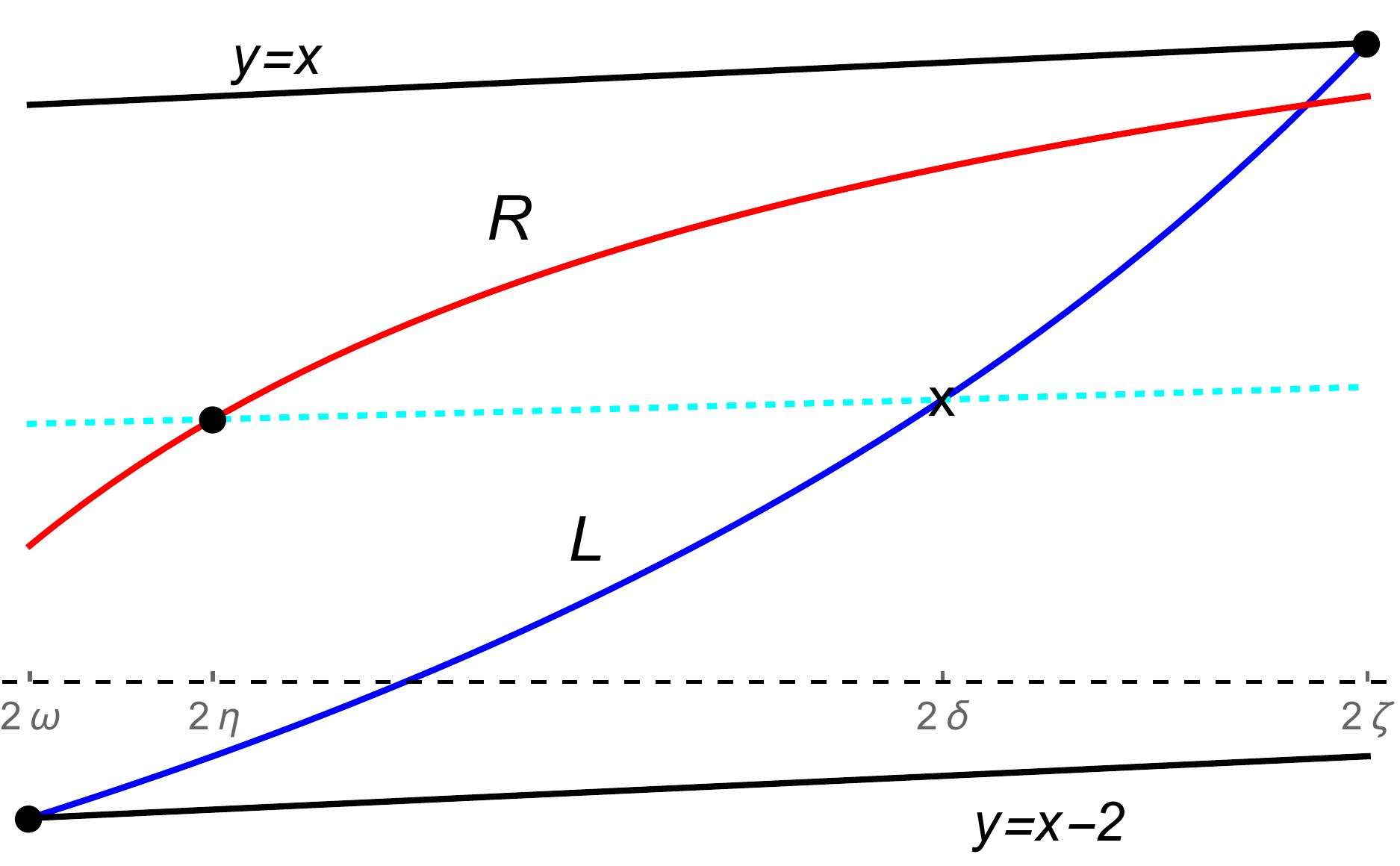}}}
\caption{Determining the synchronization interval $[\eta, \zeta)$, when $\alpha > \epsilon_{3,n}$.   Here,   $m=3, n=3$,  and $k=2$, $v = 1$.  The labels $L, R$ mark respectively the curves $y = L_{-2,1}\cdot r_0(\alpha), y= R_{-2, 1}\cdot r_0(\alpha)$ where $\alpha = x/2 = x/t_{3,3}$.   These are  $y = \ell_1(\alpha)$ (in blue) and $y = r_2(\alpha)$ (in red) for   $G < x <    (-1 + \sqrt{21}/2)$.   The dotted cyan curve is the image of the left endpoint under $C^{-1}$ (giving the values of $\mathfrak b_{\alpha}$).   The $x$-axis is shown as a dotted line.  For $\alpha >\delta = \delta_{-2,1}$, synchronization occurs after an extra step in the orbit of $r_0$.}
\label{largeAlpLRpair}
\end{figure}
 
\bigskip
    The result (ii) below leads to the conclusion that synchronization intervals for large $\alpha$ can be described by  the same set of words  as for small $\alpha$.

\begin{Lem}\label{l:LargeAlpsEtaDigits}    Fix $m=3$, an interval $[\eta, \zeta]$ of $\alpha$  such that  $r_1 = AC^2\cdot r_0$ holds on this interval,   and  $i,j \in \mathbb N$.  Suppose  that there are matrices $R, L, L'$ (none of which is the identity) such that 
\begin{enumerate}
\item[(a)]   $L = C^{-1}AC^2R$,
\item[(b)]   $R\cdot r_0 = r_{j-1}$ and $L'\cdot \ell_0 = \ell_{i-2}$,  for all $\alpha \in [\eta, \zeta]$, 
\item[(c)]  $L A\cdot \ell_0  = \ell_{i-1}$ for all $\alpha \in [\eta, \zeta)$,  while  $\ell_{i-1}(\zeta) = A^{-1}L A\cdot \ell_0(\zeta) = \ell_0(\zeta)$, 
\item[(d)] $R \cdot r_0(\eta) = C^{-1}\cdot \ell_0(\eta)$.
\end{enumerate}

Suppose further that $A^{-k}C \cdot \ell_0(\eta) = \ell_1(\eta)$.  Then 

\begin{enumerate}
\item[(i)]  $r_{j}(\eta) = A^{-k}C^2\cdot r_{j-1}(\eta) = \ell_1(\eta)$;
\item[(ii)]   $\ell_i(\eta) = A^{-(k+1)}C\cdot \ell_{i-1}(\eta) = \ell_1(\eta)$; 
\item[(iii)]   $r_j(\zeta) = AC^2\cdot r_{j-1}(\zeta)$ and $r_{j+1}(\zeta) = AC\cdot r_j(\zeta) = r_1(\zeta)$.
\end{enumerate}
\end{Lem} 
\begin{proof}  The equality $r_{j-1}(\eta) = C^{-1}\cdot \ell_0(\eta)$ implies that $r_j(\eta) = A^{-k}C^2\cdot r_{j-1}(\eta) = \ell_1(\eta)$. 
At $\alpha = \eta$ we also have $r_{j-1} = CA^{-1}C \cdot \ell_{i-1}$, therefore $A^{-k}C^2 \cdot r_{j-1} = A^{-(k+1)}C\cdot \ell_{i-1}$ allows one to easily confirm the admissibility of this expression for $\ell_i(\eta)$, as well as that $\ell_i(\eta) = \ell_1(\eta)$.  

Finally, $r_{j-1}(\zeta)> C^{-1}\cdot\ell_0(\zeta)$ and hence  $r_j(\zeta) = A^s C^2 \cdot r_{j-1}(\zeta)$ for some $s$.  This gives 
$r_j(\zeta) = A^s C^2 \cdot (CA^{-1}C \cdot r_0(\zeta)\,) = A^{s-1}C \cdot r_0(\zeta)$.  But, $r_1 = AC^2\cdot r_0$ holds throughout this region of large $\alpha$;  in particular, $C\cdot r_0(\alpha) \in \mathbb I_{\alpha}$ here.  Therefore, $s=1$ and $r_j(\zeta) = C\cdot r_0(\zeta)$.   It then follows that $r_{j+1}(\zeta) = AC \cdot (C\cdot r_0(\zeta)\,) = r_1(\zeta)$. 
\end{proof}

 
\bigskip
\begin{Def}\label{d:upperDigsCylsMatsLargeAlps}   Let $\mathcal V$ be as in the treatment of $\alpha < \gamma_{3,n}$.    
  For each $k \in \mathbb N$ and $v = c_1d_1\cdots d_{s-1} c_s \in \mathcal V$,  we define the following.  
\begin{enumerate}

\item  The lower (simplified)  {\em digit sequence} of  $-k,v$ is  
\[ \underline{d}(-k,v) = (-k)^{c_1}, (-k-1)^{d_1},\cdots,  (-k-1)^{d_{s-1}},(-k) ^{c_s}\,, \]


\medskip 
\item  The   {\em $\alpha$-cylinder} of  $-k, v$ is 
\[ \mathscr I_{-k,v} = \{ \alpha \,|\, \underline{d}{}^{^{\alpha}}_{[1, |v|\,]} = \underline{d}(-k,v)\}\,.\]

\medskip 
\item The {\em left matrix} of  $-k, v$ is
\[ 
L_{-k,v}  =    (A^{-k}C)^{ c_s}\; (A^{-k-1}C)^{d_{s-1}}(A^{-k}C)^{c_{s-1}}\cdots (A^{-k-1}C)^{d_1} (A^{-k}C)^{c_1}A^{-1}\,.
 \]

\item The {\em synchronization interval} associated to $-k, v$ is $\mathscr J_{-k,v} = [\eta, \zeta)$ where 
 $\eta = \eta_{k, v}$ and $\zeta = \zeta_{k,v}$ are such that 
\[ L_{-k,v}A\cdot \ell_0(\zeta) = r_0(\zeta) \;\;\; \text{and} \;\;\; CA^{-1}CL_{-k,v}\cdot r_0(\eta)= C^{-1}\cdot\ell_0(\eta)\,.\]   
 
 \end{enumerate}
  
\end{Def}
\bigskip

The following implies that the complement of the union of the $\mathscr J_{k,v}$ is a  Cantor set.  This is the main result of this subsection.  
\begin{Thm}\label{t:cantorSetLargeAlps}  We have the following partition
\[ [\epsilon_{3,n}, 1) = \bigcup_{k=2}^{\infty}\, \mathscr I_{-k,1}\,.\]

Furthermore, for each $k \ge 2$ and each $v \in \mathcal V$,  the following is a partition: 
\[ \mathscr I_{-k,v} = \mathscr J_{-k,v} \cup \, \bigcup_{q=q'}^{\infty}\, \mathscr I_{-k, \Theta_q(v)}\,,\]
where $q' = 0$ unless $v= c_1$, in which case $q' = -1$.
\end{Thm} 
We defer the proof of this result until page \pageref{proofTheoremLargeAlps}.

\bigskip
Note that one extends the definition of $\underline{d}(-k,v)$ to infinite words in the obvious fashion. 

\begin{Lem}\label{l:etaLargeAlps}  Let $k \in \mathbb N$ and $v \in  \mathcal V$. Assume that 
$\eta_{-k,v}, \zeta_{-k,v} \in \mathscr I_{-k,v}$. Then $\underline{d}{}^{\eta_{-k,v} }_{[1,\infty)} = \underline{d}(-k, v (v')^\infty)$ and 
$\underline{d}{}^{\zeta_{-k,v} }_{[1,\infty)}$ is purely periodic of period $\underline{d}(-k,  \overleftarrow{v'}\,)$.
\end{Lem} 
\begin{proof}    The result for $\eta_{-k,v}$  follows from Lemma~\ref{l:LargeAlpsEtaDigits} (ii).   The definition of $\zeta_{-k,v}$ gives the second result, since 
$\underline{d}(-k,  \overleftarrow{v'}\,) = (-k)^{c_1}, (-k-1)^{d_1},\cdots,  (-k-1)^{d_{s-1}},(-k) ^{c_s-1}, (-k-1)$.   
\end{proof}

\bigskip

The following is an immediate implication of the definition of the ordering \eqref{e:theFullOrder}.
\begin{Lem}\label{l:orderReversed}  Fix $m=3$.   For any $v, w$  words and for any $k \in \mathbb N$, we have 
\[ v \prec w \;\;\; \text{if and only if}\;\;\; \underline{d}(-k,v) \succ \underline{d}(-k,w)\,.\]
\end{Lem} 
\medskip 
 
\begin{Def}\label{d:fullBranchedAsWordBigAlps}     Define    $\omega_{-k,v}$ such that 
$\underline{d}{}^{\omega_{-k,v}}_{[1, \infty)}$ is purely periodic of period $\underline{d}(-k, (\, \mathfrak f (v)\,)$,
where $\mathfrak f (v)$ is the full branched prefix of $v$.  
 \end{Def}

 \bigskip

\begin{Lem}\label{l:potentialCylindersLargeAlps}   Let $v \in \mathcal V$ and fix $k\ge 2$.   The $\alpha$-cylinder set $\mathscr I_{-k,v}$ is a subset of   $(\omega_{-k,v}, \zeta_{-k,v}]$. 
\end{Lem}
\begin{proof}   The set of $\alpha$ such that $\underline{d}{}^{\alpha}_{[1, |v|]} = \underline{d}(k,v)$ is contained in the interval $(\omega,\zeta]$ such that $L_{-k,v}A\cdot \ell_0(\omega) = \ell_0(\omega)$ and  $L_{-k,v}A\cdot \ell_0(\zeta) = r_0(\zeta)$.   The right endpoint here is exactly $\zeta= \zeta_{-k,v}$.       

Due to the order reversing relationship between words and simplified digits with $-k<0$,  the proof of Lemma~\ref{l:potentialCylinders} shows that $\omega_{-k,v}$ is a greatest lower bound for $\mathscr I_{-k,v}$.
\end{proof} 
 
\bigskip
\begin{Lem}\label{l:cylsOfConsecuChildrenAbutLargeAlps}   Suppose that $v  = \Theta_p(u)$ for some $u \in  \mathcal V$ and some $p\ge 1$.    Fix $k \ge 2$.   Then for $q\in \mathbb N$, $\zeta_{-k, \Theta_{q-1}(v)} = \omega_{-k, \Theta_{q}(v)}$.
\end{Lem}
\begin{proof} From Lemma~\ref{l:cylsOfConsecuChildrenAbut}, 
$\mathfrak f(\Theta_q(v)) = \overleftarrow{(\Theta_{q-1}(v)\,)'}$.  Now Lemma~\ref{l:etaLargeAlps}  and the definition of $\omega_{-k, \Theta_q(v)}$ yield the result.
\end{proof}

\bigskip
\begin{Lem}\label{l:omegaZerosAreOmega}   For all $v\in \mathcal V$ and all $h\in \mathbb N$,  $\omega_{-k, \Theta_{0}^{h}(v)} = \omega_{-k, v}$. 
\end{Lem}
\begin{proof}  Proposition~\ref{p:frakFofChildren} shows that $\mathfrak f(\, \Theta_{0}^{h}(v)\,) = \mathfrak f(v)$ for all $v\in \mathcal V$. The result thus holds. 
\end{proof}

\bigskip 

The following result, and its proof, are completely analogous to Proposition~\ref{p:Admissible} where  the case of small $\alpha$ is treated.  
\begin{Prop}\label{p:AdmissibleLargeAlps}   For all $v \in \mathcal V$, and all $k \ge 2$, both
\[ \mathscr I_{-k,v} = (\omega_{k,v}, \zeta_{-k,v}]\;\; \text{and} \;\;  \mathscr I_{-k,v} \supset \mathscr J_{-k,v}\,.\]
\end{Prop} 
\begin{proof} 

Fix $k$.   
 We argue by induction of the length of the word $v$.     We have already seen that $\eta_{-k,v} \in \mathscr I_{-k,v}$ implies $\underline{d}(-k, v (v')^\infty)$ gives the sequence of simplified digits of $\ell_0(\eta_{-k,v})$.   
The base cases, given by $v = c_1$, are easily verified.   
\bigskip

\noindent
{\bf Case 1.}   Suppose   $v = \Theta_{0}^{h}(c)$ for some $h\ge 1$ and some $c> 1$.  We have $v =c (1\,c)^h$.   The argument as for Proposition~\ref{p:Admissible} goes through (compare with the remaining  cases). 

\bigskip

\noindent
{\bf Case 2.} Suppose $v = uau$ in our usual decomposition.  Since $\mathfrak f(v) = ua$, we have  $L_{-k, ua}A\cdot \ell_0(\omega_{-k,v}) = \ell_0(\omega_{-k,v})$.   This equality then implies that given $N$, there are $\alpha$ sufficiently close to, and larger than,  $\omega_{-k,v}$ such that $\underline{d}(-k, u(au)^N)$ is admissible for  $\alpha$.   It follows that we can partition $\mathscr I_{-k,v}$ by subintervals corresponding to the values $N$.   Now,  $\Theta_0(v) = vv'' = uauau$, and hence $\Theta_0(v)$ is admissible on $[\omega_{-k,v}, \zeta_{k,\Theta_0(v)})$.   Similarly, each $\Theta_{0}^{h}(v)$ is admissible on $[\omega_{-k,v}, \zeta_{k,\Theta_{0}^{h}(v)})$.  By considering our ordering on words, it is clear that   $\zeta_{k,\Theta_{0}^{h}(v)}> \eta_{k, \Theta_{0}^{h}(v)} >  \omega_{k,v}$. 

We now  proceed inductively for larger values of $q$.   To begin, the admissibility $v$ on all of $\mathscr I_{-k,v}$ and the admissibility of $\Theta_0(v)$ on exactly $[\omega_{-k,v}, \zeta_{k,\Theta_0(v)})$ implies that   there must be a shortest extension of $v = uau$ admissible for those $\alpha$  immediately to the right of $\zeta_{k,\Theta_0(v)}$.    Lemma~\ref{l:cylsOfConsecuChildrenAbutLargeAlps}  shows that  $\mathfrak f(\, \Theta_1(v)\,)$ is this extension.   Since  $\underline{d}{}^{\omega_{-k, \Theta_1(v)}}_{[1, \infty]}$ is purely periodic of period $\underline{d}(-k, (\, \mathfrak f ( \Theta_1(v))\,)$,  arbitrarily high powers of this period give admissible expansions for $\alpha$ just to the right of $\omega_{-k, \Theta_1(v)}$.   Recall that Proposition~\ref{p:frakFofChildren} shows that  {\em any} $v = uau$ is always a prefix of the square of the corresponding $\mathfrak f(v) = ua$;  we hence find that  all of $\Theta_1(v)$ is admissible on an interval beginning at $\omega_{-k,\Theta_1(v)}$.  We iterate this argument for increasing $q$, to give that for each $q$,  $\Theta_q(v)$ is admissible on  $[\zeta_{k,\Theta_q(v)}, \omega_{k,\Theta_q(v)})$.

 From Lemmas~\ref{l:etaLessThanOmega} and ~\ref{l:orderReversed}   
we can conclude that   $\eta_{-k, \Theta_q(v)} >\omega_{-k, \Theta_q(v)}$.      It follows that $\mathscr J_{-k,\Theta_q(v)} \subset \mathscr I_{-k,\Theta_q(v)}$.

\bigskip

\noindent
{\bf Case 3.}  As in the proof of Proposition~\ref{p:Admissible}, the remaining case is that $v =  \Theta_{0}^{h}(uau)$.  Thus, $v = u(au)^{h+1}$ in our usual decomposition.     The right endpoint of $\mathscr I_{-k, \Theta_{0}^{h+1}(uau)}$ is at $\zeta_{-k, u(au)^{h+2}}$.    Lemma~\ref{l:etaLargeAlps} now yields that  the left endpoint of $\mathscr I_{k, \Theta_{0}^{h}(uau)} \setminus \mathscr I_{k, \Theta_{0}^{h+1}(v)}$ is the point with purely periodic lower simplified digit expansion of  period  $\underline{d}(k,  v u' a)$.    In the proof of Proposition~\ref{p:Admissible} we showed that  $\mathfrak f(\, \Theta_1\circ \Theta_{0}^{h}(u)\,)  = v u' a$.    Therefore,  $\omega_{-k, \Theta_1(v)}= \zeta_{-k, \Theta_{0}^{h+1}(uau)}$ and also since     $\Theta_1(v)$ is a prefix of the square of $vu'a$, it follows that $\overline{d}(-k, \Theta_1(v)\,)$ is admissible on $[\omega_{-k, \Theta_1(v)}, \zeta_{-k,  \Theta_1(v)})$.  That $\eta_{-k, \Theta_1(v)}$ belongs to this interval is easily shown.  

Induction shows the result for $\Theta_q(v)$ when $q\ge 1$.
\end{proof}  

\bigskip

We are now ready for the following.\label{proofTheoremLargeAlps}
\begin{proof}[Proof of Theorem~\ref{t:cantorSetLargeAlps}]  That $[\epsilon_{3,n},1) = \cup_{k=1}^{\infty}\, \mathscr I_{-k,1}$  is simply a consequence of the fact that for each $\alpha$ in this range,  there is some $k$ such that 
$T_{\alpha}(\,\ell_0(\alpha)\, ) = A^{-k} C\cdot x$.

    Proposition~\ref{p:AdmissibleLargeAlps} shows that for all $v \in \mathcal V$, 
$\mathscr I_{-k,v} = [\omega_{-k,v}, \zeta_{-k,v})$ can be partitioned by $\mathscr J_{-k,v} = [\eta_{-k,v}, \zeta_{-k,v} )$ and its complement.    Recall Lemma~\ref{l:omegaZerosAreOmega} states that  for all $h$,  we have $\omega_{-k, \Theta_{0}^{h}(v)} = \omega_{-k,v}$.     By Lemma~\ref{l:cylsOfConsecuChildrenAbutLargeAlps} (and the complementary results in the proof of Proposition~\ref{p:AdmissibleLargeAlps}),   for all $q\in \mathbb N$, $\omega_{k, \Theta_q(v)} = \zeta_{k, \Theta_{q-1}(v)}$.  Therefore, 
$\cup_{q=0}^{\infty}\, \mathscr I_{-k, \Theta_q(v)}$ is a subinterval of $\mathscr I_{-k,v}\setminus \mathscr J_{-k,v}$ which has  $\omega_{-k,v}$ as its left endpoint.   Finally, the definition of $\Theta_q(v)$ combined with Lemma~\ref{l:etaLargeAlps} shows that $\lim_{q \to \infty}\, \zeta_{k, \Theta_q(v)} = \eta_{k,v}$.   Therefore,  the right endpoint of the union is in fact the left endpoint of $\mathscr J_{k,v}$. 
\end{proof} 

\subsection{Synchronization holds on a set of full measure}

\subsubsection{Right digits are admissible; synchronization occurs on each $\mathscr J_{-k,v}$}
 
 We now define $R_{-k,v}$  exactly so that the group identity of Proposition~\ref{l:longWordInU}, below,     gives that 
$L_{-k,v}=  C^{-1}AC^2 R_{-k,v}$, and thus  the main hypothesis of Lemma~\ref{l:LargeAlpsEtaDigits} is satisfied.     
That synchronization does occur along $\mathscr J_{-k,v}$ is then only a matter of showing that $R_{k,v}\cdot r_0(\alpha)$ is admissible at all $\alpha\in \mathscr J_{-k,v}$.  

 For further ease, we  set 
 \[ u =
  u_{3,n} = (1,2)^{n-2}, (1,1).\]
Note that the length of $u$ is $|u| =n - 1$. As $\alpha$ tends to one,  $\overline{b}{}^{\alpha}_{[1,\infty)}$ begins with ever higher powers of $u$, compare with Proposition~\ref{p:veeOnTheRight}.

For typographic ease,   for $k\ge 2$ let 
\[ \mathcal E = \mathcal E_k =(1,1) u^{k-2} (1,2)^{n-3}, \; 
\mathcal F = \mathcal F_k =(1,1) u^{k-1} (1,2)^{n-3}\;\text{and}\;\mathcal G =(1,2)(1,1) (1,2)^{n-3} \,.\]
Note that $\mathcal E \mathcal G = \mathcal F$, and $\mathcal E_{k+1}= \mathcal F_k$.   We also use 
that $u^{k-1} (1,2)^{n-3} =  [(1,2)^{n-2}, (1,1)]^{k-1}(1,2)^{n-3} = (1,2)^{n-2}, (1,1)  [(1,2)^{n-2}, (1,1)]^{k-2}(1,2)^{n-3}  = (1,2)^{n-2} \mathcal E$. 

We accordingly let 
\[\tilde{\mathcal E}= \tilde{\mathcal E}_k = (AC^2)^{n-3}U^{k-2}AC,\;\; \tilde{\mathcal F}= \tilde{\mathcal F}_k = (AC^2)^{n-3}U^{k-1}AC\;\;\text{and}\;\;\tilde{\mathcal G} = (AC^2)^{n-3} ACAC^2.\;\;   \]

\begin{Def}\label{d:derivedWordsLargerAlps}     Suppose that $v = c_1 d_1 \cdots c_s \in \mathcal V$  and $k\ge 2$. 

\begin{enumerate}
 \item  The upper  digit sequence of  $-k,v$  is  
 \[\overline{b}(-k,v) =  (1,2)^{n-2} \mathcal E^{c_1}  \mathcal F^{d_1} \,\mathcal E^{c_2 } \,\mathcal F^{d_2} \cdots \mathcal E^{c_{s-1}} \,\mathcal F^{d_{s-1}} \,\mathcal E^{c_s},
                                      \]
whose length is denoted 

\medskip
\item  $\overline{S}(-k,v) = |\,\overline{b}(-k,v)\,| $

\smallskip
\item The {\em right matrix} of  $-k, v$ is 
 \[ 
 \begin{aligned} 
R_{-k,v} &=  \tilde{\mathcal E}^{c_s}  \tilde{\mathcal F}^{d_{s-1} }\;\tilde{\mathcal E}^{c_{s-1}} \tilde{\mathcal F}^{d_{s-2}}\cdots \tilde{\mathcal E}^{c_{2}}\tilde{\mathcal F}^{d_{1}} \tilde{\mathcal E}^{c_1}(AC^2)^{n-2}. 
\end{aligned}
\]
 \end{enumerate}
\end{Def}
\noindent 
Note that Proposition~\ref{l:longWordInU}, below,  implies that $L_{-k,v} = C^{-1}A C^2\, R_{-k,v}$.   

\bigskip
We will prove admissibility of $\overline{b}(-k,v)$ on $\mathscr J_{-k,v}$ by induction, similar to our proof of admissibility of 
$\underline{d}(k,v)$ on $\mathscr J_{-k,v}$.  However,  the role of the various $\eta_{k,v}$ will now be played by certain points lying to the left of the corresponding $\eta_{-k,v}$, the  $\beta_{-k,v,N}$ introduced in the next statement.
\begin{Lem}\label{l:leftExpansionsAtEndpointsLargeAlp}   Let $v = c_1 d_1\cdots c_s \in \mathcal V$ and 
$k \in \mathbb N$.    Suppose that $\overline{b}(-k,v)$ is admissible on $\mathscr J_{-k,v}$.   Set $\zeta = \zeta_{-k,v}$ and $\eta = \eta_{-k,v}$.
Then for each $N \in \mathbb N$,  there exists $\beta_{-k,v,N}$ less than $\eta$ such that 
\[\overline{b}{}^{ \beta_{-k,v,N}}_{_{[1,(1+\overline{S})N\,]}} = [\, \overline{b}(-k,v) (1,1)\,]^N.\]

Furthermore,
\[
\overline{b}{}^{\zeta}_{[1,\infty)}  = 
                      \begin{cases} \overline{b}(-k, v), \overline{\mathcal G \mathcal E^{c_1}}&\text{if}\; v=c_1,\\
                      \\
                                             \overline{b}(-k,  (\,\overleftarrow{v'}\,)^{\infty}\,) &\text{otherwise}.
                      \end{cases}                
\]  
\end{Lem}
\begin{proof}  

The definition of $\eta_{-k, v}$ shows that $AC R_{-k,v}$ fixes $r_0(\eta_{-k, v})$.   Therefore, for each $N\in \mathbb N$,  there exists $\alpha = \alpha'_N< \eta_{-k,v}$ and sufficiently close,  with $\overline{b}{}^{\alpha}_{_{[1,(1+\overline{S})N\,]}} = [\, \overline{b}(-k,v) (1,1)\,]^N$.

Lemma~\ref{l:LargeAlpsEtaDigits}, (iii) yields 
$\overline{b}{}^{\zeta}_{[1,\infty)} = \overline{b}(-k,v), (1,2)(1,1) \overline{b}{}^{\zeta}_{[2,\infty)}$.
This sequence is of course periodic.   The prefix $u^{k-1}(1,2)^{n-3} $ of $\overline{b}(-k,v)$ can be rewritten as $(1,2)^{n-2}(1,1) u^{k-2}(1,2)^{n-3}$, thus $(1,2)(1,1) \overline{b}{}^{\zeta}_{[2,\infty)}$  has the prefix 
$(1,2)(1,1) (1,2)^{n-3}(1,1) u^{k-2}(1,2)^{n-3}  = \mathcal G \,\mathcal E$.   This prefix is followed by the complement of the prefix $u^{k-1}(1,2)^{n-3} $ of $\overline{b}(-k,v)$.     The case of $v = c_1$ is now easily verified.  

For general $v$,  
$\overline{b}{}^{\zeta}_{[1,\infty)} = \overline{b}(-k,v), \overline{ \mathcal G \,\mathcal E\; \mathcal E^{c_1-1} \, \mathcal F^{d_1 } \,\mathcal E^{c_2} \,\mathcal F^{d_2} \cdots \mathcal E^{c_{s-1} }  \,\mathcal F^{d_{s-1}} \,\mathcal E^{c_s}}$.   This last equals 
$\overline{b}(-k, (c_1 d_1 \cdots d_{s-1} (c_s-1) 1)^{\infty}\,)$.  Since $v$ is a palindrome, this in turn equals \newline $\overline{b}(-k, (c_s d_{s-1} \cdots d_1 (c_1-1) 1)^{\infty}\,)$.   Since $c_s d_{s-1} \cdots d_1 (c_1-1) 1 = \overleftarrow{v'}$, the result holds. 
\end{proof}

\begin{Lem}\label{l:leftWAdmissLargeAlps}   Fix $j\in \mathbb N$ and $0\le i < |u|$. If there is some $\alpha \ge \epsilon_{3,n}$ such that
$\overline{b}{}^{\alpha}_{[1, j |u|+i ]\,} = u^j u_{[1, i]}$, then $\overline{b}{}^{\alpha'}_{[1, j |u|+i ]\,} = u^ju_{[1, i]}$ for all $\alpha'>\alpha$.
\end{Lem} 
\begin{proof}  (Of course, if $i=0$, then $w_{[1, i]}$ is the empty word.) Using Proposition~\ref{p:veeOnTheRight} with Lemma~\ref{l:admissBetween2PtsRvals}, one has that  all powers of $U$ are admissible when $\alpha = 1$,  there are thus branches of digits corresponding to each power and appropriate suffix of $U$ that continue to the left from $\alpha = 1$.   For any of these, admissibility at any given $\alpha$  thus guarantees admissibility at each $\alpha' \ge \alpha$. 
\end{proof}

\begin{Lem}\label{l:leftExpansionsAtOmegaKBigAlps}   Fix 
$k\ge 2$.    We have   

\[
\overline{b}{}^{\omega_{-k,1}}_{[1,\infty)} = u^{k-1} (1,2)^{n-3}\, \overline{\mathcal E}\,. 
\]

Furthermore, the digits  
$\overline{b}{}^{\alpha}_{[1, n-3 + (k-1) |u|\,]} =  u^{k-1} (1,2)^{n-3}$ are admissible for all $\alpha \ge \omega_{k,1}$. 
\end{Lem}
\begin{proof}  We first give one proof for $k\ge 3$.
Since $\mathfrak f(1) = 1$,  $\zeta_{-k,1} = \omega_{-(k+1),1}$.   From
\[
\begin{aligned} 
 \overline{b}{}^{\zeta_{-k,1}}_{[1,\infty)}  &=   u^{k-1}, (1,2)^{n-3} \,\overline{\mathcal G \mathcal E_{k}} \\
     &=    u^{k-1}, (1,2)^{n-3} \mathcal G\, \overline{ \mathcal E_{k}\mathcal G}\\      
     &=   u^{k} (1,2)^{n-3}\, \overline{\mathcal E_{k+1}}. 
\end{aligned} 
\]

The following works for all $k$. 
 For typographical ease, let $\omega= \omega_{-k,1}$.  Since $\mathfrak f(1) = 1$, by definition,    $A^{-k}C \cdot \ell_0(\omega) = \ell_0(\omega)$.   
By Lemma~\ref{l:oneStepId},  $CA^{-1}CA^{-k}CA^{-1} =  (AC^2)^{n-3} U^{k-1}$.   By Lemma~\ref{l:leftExpansionsAtOmegaKBigAlps},  $(AC^2)^{n-3} U^{k-1}\cdot r_0(\omega)$ is an admissible expansion if 
it has value in $\mathbb I_{\omega}$.      

For {\em any} $\alpha$,
$CA^{-1}C \cdot \ell_0(\alpha) =  CA^{-1}CA^{-1}\cdot r_0(\alpha) = (AC^2)^{n-2}\cdot r_0(\alpha)$.    Hence,   $(AC^2)^{n-3} U^{k-1}\cdot r_0(\omega) = CA^{-1}CA^{-k}CA^{-1}\cdot  \ell_0(\omega) = CA^{-1}CA^{-1}\cdot  \ell_0(\omega) = (AC^2)^{n-2}\cdot r_0(\omega)$.   Since $(AC^2)^{n-2}$ is a suffix of $U$, we have that $(AC^2)^{n-2}\cdot r_0(\omega)$ is admissible.  In particular it has value in  $\mathbb I_{\omega}$.  It follows that 
 $(AC^2)^{n-3} U^{k-1}\cdot r_0(\omega)$ is an admissible expansion.   Furthermore, the equality
 $(AC^2)^{n-2}\cdot r_0(\omega) = (AC^2)^{n-3} U^{k-2}AC (AC^2)^{n-2}\cdot r_0(\omega)$ shows that  $(AC^2)^{n-2}\cdot r_0(\omega)$ is fixed by $(AC^2)^{n-3} U^{k-2}AC$.  Therefore,  
 \[\overline{b}{}^{\omega_{-k,1}}_{[1,\infty)} = (1,2)^{n-2}\overline{(1,1) u^{k-2} (1,2)^{n-3}} = u^{k-1} (1,2)^{n-3}\overline{(1,1) u^{k-2} (1,2)^{n-3}} = u^{k-1} (1,2)^{n-3}\, \overline{\mathcal E}\,. \]

 Lemma~\ref{l:leftWAdmissLargeAlps}  yields that 
the digits  
$\overline{b}{}^{\alpha}_{[1, n-3 + (k-1) |u|\,]} =  u^{k-1} (1,2)^{n-3}$ are admissible for all $\alpha \ge \omega_{-k,1}$.    
\end{proof}

\begin{table} 
\[\arraycolsep=3pt\def\arraystretch{1.2}
\begin{array}{c|ccc} 
                            &v&\alpha'&\alpha''\\[2pt]
  \hline
                            &1&1&\omega_{-k,1}\\
                            &c>1&\beta_{-k, c-1,2}&\omega_{-k,1} \\[2pt]
\text{Base cases}&c\,1\,c, c\ge 1  &\beta_{-k,c,3}&\zeta_{-k, c+1}\\[2pt]
                            &\Theta_q(c), q\ge 2&\beta_{-k,c,q+3}&\zeta_{-k, \Theta_{q-1}(c)}\\[2pt]
                            &\Theta_1(c), c>1&\beta_{-k,c,4}&\zeta_{-k,\Theta_{0}(c)}\\[2pt]                   
  \hline \rule{0pt}{\dimexpr.7\normalbaselineskip+1mm}
\text{Case 1}&\Theta_{0}^{h}(c)&\beta_{-k,  \Theta_{0}^{h-1}(c), 2}&\zeta_{-k,c+1}\\[5pt]
\hline \rule{0pt}{\dimexpr.7\normalbaselineskip+1mm}
\text{Case 2}&\Theta_q\circ \Theta_{p}(u), \; p,q\ge 1&\beta_{-k,  \Theta_{p}(u), q+3}&\zeta_{-k,\Theta_{q-1}\circ\Theta_{p}(u)}\\[5pt]
\hline \rule{0pt}{\dimexpr.7\normalbaselineskip+1mm}
\text{Case 3}&\Theta_{0}^{h}\circ \Theta_{p}(u), \, p\ge 1&\beta_{-k, \Theta_{0}^{h-1}\circ\Theta_{p}(u),3}&\zeta_{-k,\Theta_{p-1}(u)}\\[2pt]
 \end{array}
\]
\bigskip
\caption{Admissibility of $\overline{d}(-k,v)$ on $\mathscr J_{-k,v}$ is shown by finding $\alpha'> \zeta_{-k,v}>\eta_{-k,v}> \alpha''$ such that  $\overline{b}_{[1, \infty)}^{\alpha}$  has prefix $\overline{b}(-k,v)$ for both $\alpha = \alpha', \alpha''$. See the proof of Proposition~\ref{p:rightDigitsAreGood}. Compare with Table~\ref{t:casesForGoodPtsShowingAdmissibility}.}\label{t:casesForPtsGivingRightAdmiss}
\end{table}

\bigskip 
\begin{Prop}\label{p:rightDigitsAreGood}     Suppose that $v \in \mathcal V$  and $k\ge 2$.  Then for 
all $\alpha \in \mathscr J_{-k,v}$,  
\[  \overline{b}{}^{\alpha}_{_{[1,\overline{S}\,]}}  =  \overline{b}(-k,v)\,.\]
\end{Prop} 
\begin{proof}   By Lemma~\ref{l:admissBetween2PtsRvals}, for each $v$ it suffices to find $\alpha'<\eta_{-k,v}$ and $\alpha''>\zeta_{-k,v}$ such that  $\overline{b}{}^{\alpha'}_{_{[1,\overline{S}\,]}}  = \overline{b}{}^{\alpha''}_{_{[1,\overline{S}\,]}}  = \overline{b}(-k,v)$.  See Table~\ref{t:casesForPtsGivingRightAdmiss} for a summary of the choices of $\alpha', \alpha''$ in the various cases.

Since $\alpha'$ is always of the general form  $\beta_{-k,u,N}$,  we note immediately that the identity  $\mathcal F = \mathcal E \mathcal G$ yields 

\begin{equation}\label{e:thePowersThatBetaBe}
\begin{split}& \lbrack \,\overline{b}(-k,v) (1,1)\,\rbrack^N\\
 & \phantom{score} =\,\overline{b}(-k,v)  [\mathcal F \mathcal E^{c_1-1} \, \mathcal F^{d_1} \,\mathcal E^{c_2} \,\mathcal F^{d_2} \cdots \mathcal E^{c_{s-1}} \,\mathcal F^{d_{s-1}} \,\mathcal E^{c_s}]^{N-1}(1,1)\\
& \phantom{score} =\,\overline{b}(-k,v)  [\mathcal E \mathcal G \mathcal E^{c_1-1} \, \mathcal F^{d_1} \,\mathcal E^{c_2} \,\mathcal F^{d_2} \cdots \mathcal E^{c_{s-1}} \,\mathcal F^{d_{s-1}} \,\mathcal E^{c_s}]^{N-1}(1,1)\\
& \phantom{score} =\, \overline{b}(-k,v) \mathcal E  [\mathcal G \mathcal E^{c_1-1 } \, \mathcal F^{d_1} \,\mathcal E^{c_2}  \,\mathcal F^{d_2} \cdots \mathcal E^{c_{s-1}}\,\mathcal F^{d_{s-1}} \,\mathcal E^{c_s+1}]^{N-2}\\
&\phantom{score}\;\;\;\;\;\;\;\;\;\;\cdot\mathcal E \mathcal G \mathcal E^{c_1-1} \, \mathcal F^{d_1} \,\mathcal E^{c_2} \,\mathcal F^{d_2} \cdots \mathcal E^{c_{s-1}}   \,\mathcal F^{d_{s-1}} \,\mathcal E^{c_s}(1,1)\\
& \phantom{score} =\,\overline{b}'(-k,v (v')^{N-2})X,
\end{split}
\end{equation}
where $X = \mathcal E \mathcal G \mathcal E^{c_1-1} \,  \mathcal F^{d_1} \,\mathcal E^{c_2}   \,\mathcal F^{d_2} \cdots \mathcal E^{c_{s-1}} \,\mathcal F^{d_{s-1}} \,\mathcal E^{c_s}(1,1)$ and we admit to our abuse of notation by writing $\overline{b}'$.

\bigskip 

\noindent
{\bf Base cases.} 
Consider $v=c_1 = c$.  Since $\overline{b}{}^{\omega_{-k,1}}_{[1,\infty)}$ has $\overline{b}(-k,1)$ as a prefix,  Lemma~\ref{l:leftWAdmissLargeAlps} yields the result when $v=1$.    From \eqref{e:thePowersThatBetaBe}, 
one finds that for $c\ge 2$,  $\overline{b}(-k,c)$ is a prefix of $[\overline{b}(-k,c-1) (1,1)]^2$, which in turn is a prefix of 
$\overline{b}_{[1, \infty)}^{\beta_{-k, c-1,2}}$.      Lemma~\ref{l:leftExpansionsAtOmegaKBigAlps} shows that also the right digit sequence of $\omega_{-k, 1}$ has each $\overline{b}(-k,c)$ as a prefix.  Therefore,  the result holds for all $v$ of length one. 
 
\bigskip 

\noindent
When $v=c\,1\,c$, we use $\alpha' = \beta_{-k,c,3}$, and $\alpha'' =  \zeta_{-k,c+1}$.     Since $\overline{b}(-k, v) = u^{k-1}, (1,2)^{n-3}\, (\mathcal E^c \mathcal G)^{1}\, \mathcal E^c$, \eqref{e:thePowersThatBetaBe} leads to $[\,\overline{b}(-k,v) (1,1)\,]^3$  having  $\overline b(-k, v)$ as a prefix.     
Lemma~\ref{l:leftExpansionsAtEndpointsLargeAlp}   yields that $\overline{b}{}^{\zeta_{-k,c+1}}_{[1,\infty)}$ also has $\overline b(-k,v)$ as a prefix.     
 
\bigskip 

\noindent
Suppose $v=\Theta_q(c), q\ge 1$ and $c>1$.  Thus,   $\overline{b}(-k, v) = u^{k-1}, (1,2)^{n-3}\, (\mathcal E^c \mathcal G)^{q+1}\, \mathcal E^c$.   We set $\alpha' = \beta_{-k,c, N}$, with $N = q+2$.  Using \eqref{e:thePowersThatBetaBe}, one  shows that  $\overline{b}_{_{[1,(1+\overline{S})N\,]}}^{\alpha'}$ has  $\overline b(-k,v)$ as a prefix.     
We set $\alpha'' = \zeta_{-k, \Theta_{q-1}(c)}$. Lemma~\ref{l:leftExpansionsAtEndpointsLargeAlp} yields that  $\overline{b}{}^{\alpha''}_{[1,\infty)}$ has prefix $\overline{b}(-k, \Theta_{q-1}(c))\, \mathcal G \mathcal E^{c+1}$.  The result thus holds in this case.  Since $c=1$ gives  $\overline{b}(-k, v) = u^{k-1}, (1,2)^{n-3}\, \mathcal E  \mathcal G \mathcal F^{q-1}\, \mathcal E$,   again using \eqref{e:thePowersThatBetaBe}  shows that the same argument succeeds.

\bigskip 

\noindent
{\bf Case 1.}  Suppose   $v = \Theta_{0}^{h}(c), h\ge 2$ for some $c> 1$.  We have $v =c (1\,c)^h$ and thus   $\overline{b}(-k, v) = u^{k-1}, (1,2)^{n-3}\,\mathcal E^c \mathcal G\, (\mathcal E^{c+1} \mathcal G)^{h-1}\, \mathcal E^c$.  Lemma~\ref{l:leftExpansionsAtEndpointsLargeAlp}  easily yields that  $\overline{b}{}^{\zeta_{-k,c+1}}_{[1,\infty)}$  has $\overline b(-k,v)$ as a prefix. That is, we can take $\alpha'' =  \zeta_{-k,c+1}$.    With  $\alpha' = \beta_{-k,  \Theta_{0}^{h-1}(c), 2}$ from \eqref{e:thePowersThatBetaBe} one  finds that $\overline{b}{}^{\alpha'}_{[1,\infty)}$ also has $\overline b(-k,v)$ as a prefix.

\bigskip  
\noindent
{\bf Case 2.}    Suppose that the result holds for $v = uau = \Theta_p(u)$ with $u \in \mathcal V$ and $p \ge 1$.  We prove that the result holds for   $\Theta_q(v), q\ge 1$.    
By Lemma~\ref{l:leftExpansionsAtEndpointsLargeAlp}, for general $k$ and $v$, one has $\overline{b}{}^{\zeta_{-k,v}}_{[1,\infty)}  =   \overline{b}(-k,  (\,\overleftarrow{v'}\,)^{\infty}\,)$, and thus \eqref{eq:zetaWordCalc} gives 
$\overline{b}{}^{\zeta_{-k, \Theta_{q-1}(v)}}_{[1,\infty)}  =   \overline{b}(-k, v \,[\, (v')^q v''\,]^{\infty} \,)$, which clearly has $\overline{b}(-k, \Theta_q(v ) \,)$ as a prefix.  Therefore,  we can take $\alpha'' = \zeta_{-k, \Theta_{q-1}(v)}$.

We first specialize to $c>1$, that $v' =  a \overleftarrow{u'} u = a u_{[-1]} (c_1-1)\,1\,u$, $v'' = au$ and that $u$ both begins and ends with the letter $c_1$.  
Since  terms in general $\overline{b}(-k,v)$ corresponding to any $d_j = 1$ vanish, a final $v'$ appearing inside of $\overline{b}'$ in \eqref{e:thePowersThatBetaBe} has a prefix coming from $v''$.  From this, one finds 
that $[\overline{b}(-k,v) (1,1)]^{q+3}$ has $\overline{b}(-k, \Theta_q(v ) \,)$ as a prefix.   That is, we can take $\alpha' = \beta_{-k,  v, q+3}$.

We now consider the case of $c=1$, where $v' =  a u_{[-2]} (d_1+1)$ and $v'' = au$. In   \eqref{e:thePowersThatBetaBe} we can  let  the first  $\mathcal F^{d_1}$  corresponding to the final power of $v'$ revert to  $\mathcal E  \,\mathcal G\, \mathcal F^{d_1 - 1}$ so as to confirm that here also a final $v'$ contributes a subword that has as a prefix the contribution of $v''$.  Using this, 
  $[\overline{b}(-k,v) (1,1)\,]^{q+3}$ has $\overline{b}(-k, \Theta_q(v ) \,)$ as a prefix.   That is, we can take $\alpha' = \beta_{-k,  v, q+3}$.

\bigskip  
\noindent
{\bf Case 3.}  Suppose that $\Theta_p(u) = uau$ in our usual notation.  Recall that $\Theta_{0}^{h}\circ \Theta_p(u) = u\, [(u')^{p} u'']^{h+1}$.     As in Case 2,   Lemma~\ref{l:leftExpansionsAtEndpointsLargeAlp}  and  \eqref{eq:zetaWordCalc} 
give
$\overline{b}{}^{\zeta_{-k, \Theta_{p-1}(u)}}_{[1,\infty)}  =   \overline{b}(-k, u \,[\, (u')^p u''\,]^{\infty} \,)$.  Thus,   we 
 use   $\alpha'' = \zeta_{-k, \Theta_{p-1}(u)}$.       From \eqref{eq:thetaZeroNeighborsAndPrimes} and \eqref{e:thePowersThatBetaBe},  we can take $\alpha' = \beta_{-k, \Theta_{0}^{h-1}\circ\Theta_p(u), 3}$.  
 \end{proof}
 
\bigskip
\subsubsection{There are no other points of synchronization}   That all $\alpha> \epsilon_{3,n}$ for which there is synchronization lie in the union of the $\mathscr J_{-k,v}$ is shown as for the case of small $\alpha$.     We find that there is some sequence of $q_i$ such that  $\alpha \in \cap_{j=1}^{\infty}\, \mathscr I_{k, \Theta_{q_{j}}\circ \cdots \circ\Theta_{q_{1}}(1)}$.   But,  this implies that  $\underline{b}{}^{\alpha}_{[1,\infty)}$ has digits only in $\{(-k,1), (-k-1,1)\}$,  while  $\overline{b}{}^{\alpha}_{[1,\infty)}$ has digits only in $\{(-1,1), (-2,1), (1,2)\}$, with the digit $(1,2)$ appearing infinitely often.  
  Therefore, the two orbits obviously cannot synchronize.  

\bigskip
\subsubsection{The non-synchronization set is of measure zero}   

  For $\alpha> \epsilon_{3,n}$ not in any $\mathscr J_{-k,v}$,  there is some $k$ such that the $T_{\alpha}$-orbit of $\ell_0(\alpha)$ is always in the set of points whose simplified digits are $-k$ or  $-k-1$.  In particular, this orbit certainly remains in $[-t,0)$ and we can argue as in Subsection~\ref{ss:nonSynchMeasZero}      to conclude that the set of these $\alpha$ values has measure zero.

\section{Synchronization for $\gamma_{3,n} < \alpha < \epsilon_{3,n}\,,\;\; n\ge m$}
Recall    that $\gamma = \gamma_{3,n}$ is characterized by $C^{-1}\cdot \ell_0(\gamma) = r_0(\gamma)$.
Recall also that  $\ell_1 = A^{-1}C\cdot \ell_0$ for all $\alpha \in [0, \epsilon_{3,n})$.   Thus, for $\gamma_{3,n} < \alpha < \epsilon_{3,n}$ we use the notation of Section~\ref{s:LargeAlps}, but now with $k=1$ (and certain further technical adjustments as noted below).    
\begin{Thm}\label{t:synchMidAlpsGenNCantor}   Fix $m = 3$ and $n\ge m$.  
If $n=3$,  then there is synchronization for all $\alpha \in (\gamma_{3,3}, \epsilon_{3,3})$. 
If $n > 3$,  then the set of $\alpha \in  (\gamma_{3,n}, \epsilon_{3,n})$ for which there is not synchronization is uncountable, but of Lebesgue measure zero. 
\end{Thm} 

\bigskip
For $n\ge 3$,  define $\check{\mathcal V}_n\subset \mathcal V$  to be the trimming of $\mathcal V$ such that for all $v = c_1 d_1 \cdots c_s \in \check{\mathcal V}_n$, $c_i \le n-2$  and furthermore  such that the only word with prefix $n-2$ is $v=n-2$ itself.    Define each  $\mathscr I_{-1,c_1}$  to be as above, except that  we insist on a left endpoint at least $\gamma$. 

\medskip 

\begin{Thm}\label{t:cantorSetMidAlps}  Fix $m = 3$ and $n\ge m$.  We have the following  equality
\[  [\gamma_{3,n}, \epsilon_{3,n}) = \bigcup_{v=1}^{n-2}\, \mathscr I_{-1,v}\,.\]

Furthermore, for each $v \in \check{\mathcal V}_n\setminus\{n-2\}$,  the following is a partition: 
\[ \mathscr I_{-1,v} = \mathscr J_{-1,v} \cup \, \bigcup_{q=q'}^{\infty}\, \mathscr I_{-1, \Theta_q(v)}\,,\]
where $q' = 0$ unless $v= c_1$, in which case $q' = -1$.
Moreover,     $\mathscr I_{-1,n-2}  = \mathscr J_{-1,n-2}  = [\gamma_{3,n}, \zeta_{-1, n-2})$.
\end{Thm} 
 
 \begin{proof}
Note that for $\alpha = \zeta_{-1,1}$, we have   $A^{-1}C\cdot \ell_0(\alpha) = r_0(\alpha)$, which is exactly the definition of $\alpha = \epsilon_{3,n}$.   Arguing as we did for $\alpha>\epsilon_{3,n}$ shows that   
first equality of the theorem, giving the basic partition,   holds.    Similarly, that $\mathscr J_{-1,v} $ with the union of the $ \mathscr I_{-1, \Theta_q(v)}$ partition $ \mathscr I_{-1,v}$ holds since the proof of Theorem~\ref{t:cantorSetLargeAlps} is easily checked to extend to this case.


Recall that  
$\eta_{-k, v}$ is such that $R_{-k,v}\cdot r_0(\eta)= C^{-1}\cdot\ell_0(\eta)$, and thus  $\eta_{-1,n-2}$ is such that $r_0(\eta_{-1, n-2}) = C^{-1}\cdot\ell_0(\eta_{-1, n-2})$.  That is, $\eta_{-1,n-2} = \gamma_{3,n}$.   
The statement $\mathscr I_{-1,n-2}  = \mathscr J_{-1,n-2}  = [\gamma_{3,n}, \zeta_{-1, n-2})$ thus follows.\end{proof}

\bigskip

\begin{Prop}\label{p:synchonoForMidAlps}  Fix $m = 3$ and $n\ge m$.   For all $v\in \check{V}_n$, synchronization occurs along $\mathscr J_{-1,v}$.
\end{Prop} 
\begin{proof}[Sketch] 
Here also the arguments of the previous section give the proof, however we must make minor adjustments.   Note that 
$U^{-1} = [AC (AC^2)^{n-2}]^{-1} = (AC^2)^{2-n} (A C)^{-1}$.    
Thus in this setting of $k=1$,  any term of the form $ [ (AC^2)^{n-3} U^{k-2} AC]^a$ as in the statements of 
 Lemma~\ref{l:oneStepId} and Proposition~\ref{l:longWordInU}  becomes  
 $ [ (AC^2)^{-1}]^a$.    Note that each such term is followed by an appearance of $(AC^2)^{n-3}$;  since each exponent $a$ will arise as either $c_j$ or $c_j+1$ for some $c_j$ letter of some $v \in \check V$,  we have $a \le n-2$.   Since an exponent of the form $c_j+1$ occurs only when the block is also preceeded by an occurrence of $AC^2$,  the identity guarantees an expression that has only positive powers of $AC$ and of $AC^2$.  (Example~\ref{e:seeingUinverseCausesNoProblem} exhibits this phenomenon.)
  Thus our definition of $R_{-k,v}$ extends to include the case of $d=-1$.  Similarly,  $\mathcal E_1$ must now denote $(1,2)^{-1}$ where we recognize that occurrences of $\mathcal E_1$ are surrounded by $(1,2)$ occurring to a sufficiently high power so that the usual arithmetic of exponents results in a sensible word.  
Note that the key relation $\mathcal E \mathcal G = \mathcal F$ thus holds in this setting.  
\end{proof} 
  
  Finally, both that synchronization only occurs along the $\mathscr J_{-1,v}$ and that this is a set of full measure follow from the arguments of the previous section.   
  
\bigskip
\begin{Eg}\label{e:seeingUinverseCausesNoProblem}  Note first that  $R_{-1, n-2} = CA^{-1}C\; (A^{-1}C)^{n-2}A^{-1} = Id$ holds for all $n\ge 3$.  We also have $R_{-1,1} = (AC^2)^{-1} \cdot(AC^2)^{n-2} = (AC^2)^{n-3}$.  Note that these calculations agree when $n=3$.  Let $n>3$ and let us calculate one longer right matrix: 
 \[R_{-1,111} = \tilde{\mathcal E}\tilde{\mathcal F}\tilde{\mathcal E}(AC^2)^{n-2} = (AC^2)^{-1}\cdot(AC^2)^{n-3}AC\cdot (AC^2)^{-1}\cdot(AC^2)^{n-2} = (AC^2)^{n-4}AC (AC^2)^{n-3}.\]
\end{Eg}

\section{Group element identities for the setting $\alpha > \gamma_{3,n}$}

Analagously to Section~\ref{s:groupIdsSmallAlp},
we gather group identities used in the preceding two sections,  here in this section.
\begin{Lem}\label{l:oneStepId}   Suppose that $m=3$, and $k, a \in \mathbb N$.  Then 
\[ CA^{-1}C (A^{-k}C)^a A^{-1} = [ (AC^2)^{n-3} U^{k-2} AC]^{a-1} (AC^2)^{n-3} U^{k-1}\,.\]
\end{Lem}

\begin{proof} We first prove the result when $a=1$.   As a base case,  we consider $k=1$.   Since $C = AB$, 
$CA^{-1}$ is a conjugate of $B$ and hence has order $n$.  Therefore,  
$CA^{-1}C \,A^{-1}C\,  A^{-1}  = (AC^{-1})^{n-3} = (AC^2)^{n-3}$.    Now suppose that the identity holds for $a=1$ and a value of $k$.  Then 
\[ 
\begin{aligned} 
CA^{-1}C (A^{-k-1}C) A^{-1} &=  CA^{-1}C (A^{-k}C) A^{-1} \; A C^{-1} \; A^{-1} C A^{-1}\\
                                             &= (AC^2)^{n-3} U^{k-1}\; A C^{-1}A^{-1} C A^{-1}\\
                                             &=  (AC^2)^{n-3} U^k,
\end{aligned} 
\]
since $U = AC(AC^2)^{n-2} = AC (CA^{-1})^2 = A C^2 A^{-1}C A^{-1} =    A C^{-1}A^{-1} C A^{-1}$. 

The proof of the result for general $a$ has an induction step as follows.                                              
\[ 
\begin{aligned} 
CA^{-1}C (A^{-k}C)^{a+1} A^{-1}  &=  CA^{-1}C (A^{-k}C)^a A^{-1} \; A ^{-k+1}C A^{-1}\\
                                                     &= [ (AC^2)^{n-3} U^{k-2} AC]^{a-1} (AC^2)^{n-3} U^{k-1}\;  A ^{-k+1}C A^{-1}\\
                                                     &= [ (AC^2)^{n-3} U^{k-2} AC]^{a} (AC^2)^{n-2} \; (C^{-1}AC^{-1}\; CA^{-1}C)\; A ^{-k+1}C A^{-1}\\
                                                     &=  [ (AC^2)^{n-3} U^{k-2} AC]^{a} (AC^2)^{n-2}   C^{-1}AC^2\;  (AC^2)^{n-3} U^{k-2} \\
                                                     &= [ (AC^2)^{n-3} U^{k-2} AC]^{a} (AC^2)^{n-3}  AC  (AC^2)^{n-2} U^{k-2} \\
                                                     &= [ (AC^2)^{n-3} U^{k-2} AC]^{a} (AC^2)^{n-3}  U^{k-1}. 
\end{aligned} 
\]
\end{proof} 

\begin{Prop}\label{l:longWordInU}   Fix  $m=3$ and $k \in \mathbb N$.  If $s \in \mathbb N$, $s>1$ and  $a_1, \dots, a_s, b_1, \dots, b_{s-1} \in \mathbb N$,  then  
\[
\begin{aligned} 
 &CA^{-1}C\; (A^{-k}C)^{a_s}\, (A^{-(k+1)}C)^{b_{s-1}}\, (A^{-k}C)^{a_{s-1}}\,\cdots (A^{-(k+1)}C)^{b_1}\, (A^{-k}C)^{a_1}A^{-1}\\
&= [ (AC^2)^{n-3} U^{k-2} AC]^{a_s}\; \;  [ (AC^2)^{n-3}U^{k-1} AC]^{b_{s-1} -1 }(AC^2)^{n-3}ACAC^2\\
&\;\;\;\;\cdot [ (AC^2)^{n-3} U^{k-2} AC]^{1+a_{s-1}}\;\; \;  [(AC^2)^{n-3}U^{k-1} AC]^{b_{s-2} -1}(AC^2)^{n-3}ACAC^2\\
&\;\;\;\;\;\;\;\;\;\;\;\;\;\;\;\;\;\;\;\;\;\;\;\;\;\;\;\;\;\;\;\;\;\;\;\;\;\;\;\;\;\;\;\;\;\;\;\;\vdots\\
&\;\;\;\;\cdot [ (AC^2)^{n-3} U^{k-2} AC]^{1+a_{2}}\;\; \;  [(AC^2)^{n-3}U^{k-1} AC]^{b_{1} -1 }(AC^2)^{n-3}ACAC^2\\
&\;\;\;\;  \cdot \;\; [ (AC^2)^{n-3} U^{k-2} AC]^{a_1}(AC^2)^{n-3}U^{k-1}.\\
\end{aligned} 
\]

\end{Prop} 
\begin{proof} Note that we can rewrite the result of Lemma~\ref{l:oneStepId}  as
\[ CA^{-1}C\, (A^{-k}C)^a A^{-1} = [ (AC^2)^{n-3} U^{k-2} AC]^{a} (AC^2)^{n-2}.\]  
Therefore,  for any $a,b\in \mathbb N$, 
\[
\begin{aligned} 
&CA^{-1}C\; (A^{-k}C)^{a}\, (A^{-(k+1)}C)^{b}  = CAC^{-1}\; (A^{-k}C)^{a}\,A^{-1}\cdot A^{-k}C  (A^{-(k+1)}C)^{b-1}\\
                 &=   [ (AC^2)^{n-3} U^{k-2} AC]^{a} (AC^2)^{n-2}\; (C^{-1}A C^2\cdot CA^{-1}C)\; A^{-k}C A^{-1}\cdot A^{-k}C(A^{-(k+1)}C)^{b-2}\\
                 &=  [ (AC^2)^{n-3} U^{k-2} AC]^{a} (AC^2)^{n-3} AC A C^2\cdot (AC^2)^{n-3} U^{k-1} \cdot A^{-k}C (A^{-(k+1)}C)^{b-2}\\
                  &= [ (AC^2)^{n-3} U^{k-2} AC]^{a} (AC^2)^{n-3} AC A C^2\cdot(AC^2)^{n-3} U^{k-2}AC (AC^2)^{n-3}ACAC^2\\
                  &\;\;\;\;\;\;\cdot CA^{-1}C\;    A^{-k}C (A^{-(k+1)}C)^{b-2}\\
                  &= [ (AC^2)^{n-3} U^{k-2} AC]^{a} (AC^2)^{n-3}  U^{k-1} AC(AC^2)^{n-3}ACAC^2\cdot CA^{-1}C\;    A^{-k}C (A^{-(k+1)}C)^{b-2}\\
                  &\vdots\\
                &= [ (AC^2)^{n-3} U^{k-2} AC]^{a} (AC^2)^{n-3} \; [U^{k-1}AC(AC^2)^{n-3} ]^{b-1}ACAC^2\;  \cdot CA^{-1}C A^{-k}C. 
\end{aligned} 
\]
The result now follows by applying this repeatedly,  beginning with $CA^{-1}C\; (A^{-k}C)^{a_s}\, (A^{-(k+1)}C)^{b_{s-1}}$,  along with a final application of Lemma~\ref{l:oneStepId} to the resulting final term.
\end{proof}

 \section{To be continued} As already mentioned, in work in progress we apply the results of this paper to determine the natural extensions of the $T_{3,n,\alpha}$ as well the entropy functions  $\alpha \mapsto h(T_{3,n,\alpha})$.

\end{document}